%% file: FreeEnergyJumpsUp_Memoirs_rev2_Neil2.tex
\documentclass{memo-l}

\usepackage{amsmath}
\usepackage{amsfonts}
\usepackage{amstext}
\usepackage{amsbsy}
\usepackage{amsopn}
\usepackage{amsxtra}
\usepackage{upref}
\usepackage{amsthm}
\usepackage{amsmath}
\usepackage{amssymb}
\usepackage{amscd}
\usepackage{euscript}
\usepackage{graphicx}
\usepackage[cspex,bbgreekl]{mathbbol}
\usepackage{enumerate}
\usepackage{mathrsfs}
\usepackage{enumitem}
\usepackage{color}

\usepackage[original]{imakeidx}
\makeindex[title=Notation] 
\makeindex[name=def, title=Definitions]
\usepackage[bookmarks=false]{hyperref}

\renewcommand\indexname{Notation} 

\theoremstyle{plain}

\newtheorem{theorem}{Theorem}[chapter]
\newtheorem{lemma}[theorem]{Lemma}

\newtheorem{fact}[theorem]{Fact}
\newtheorem{corollary}[theorem]{Corollary}
\newtheorem{proposition}[theorem]{Proposition}

\theoremstyle{definition}
\newtheorem{definition}[theorem]{Definition}

\theoremstyle{remark}
\newtheorem{remark}[theorem]{Remark}

\numberwithin{section}{chapter}
\numberwithin{equation}{chapter}

\def\sm{\setminus}
\def\st{such that }
\def\R{\mathbb{R}}
\def\N{\mathbb{N}}

\def\F{\mathcal{F}}

\def\H{\mathcal{H}}

\def\eps{\varepsilon}
\def\phi{\varphi}
\def\le{\leqslant}
\def\ge{\geqslant}
\def\cyl{{\bf C}}
\def\indcyl{{\bf Z}}
\def\T{\mathcal{T}}

\def\ie{{\em i.e.,~}}
\def\M{\mathcal{M}}

\def\P{\mathcal{P}}
\def\Q{\mathcal{Q}}

\def\D{\mathcal{D}}

\def\L{\mathcal L}

\def\level{\mathrm{lev}}

\def\wtm{\widetilde{\M}}

\def\K{\mathcal{K}}

\def\D{{\mathcal D}}
\def\E{{\mathcal E}}
\def\V{{\mathcal V}}
\def\A{{\mathcal A}}

\def\T{{\mathcal T}}
\def\Y{{\mathcal Y}}

\def\dist{\mathrm{dist}}
\def\Orb{\mathrm{Orb}}

\def\X{{\mathcal X}}

\def\FNSD{\F_\mathrm{NSD}}
\def\C{{\mathcal C}}

\def\ul{\underline}

\makeindex

\begin{document}

\title{Free energy and equilibrium states for families of interval maps}

\author{Neil Dobbs}\address{
    School of Mathematics and Statistics,
University College Dublin,
Belfield, Dublin 14,
Ireland.}
\email{\href{neil.dobbs@ucd.ie}{neil.dobbs@ucd.ie}}
\urladdr{\url{http://www.maths.ucd.ie/~ndobbs/}}

\author{Mike Todd}\address{
Mathematical Institute,
University of St Andrews,
North Haugh,
St Andrews,
Fife,
KY16 9SS,
Scotland}
\email{\href{mailto:m.todd@st-andrews.ac.uk}{m.todd@st-andrews.ac.uk}}
\urladdr{\url{http://www.mcs.st-and.ac.uk/~miket/}}

\thanks{ MT was partially supported by FCT grant SFRH/BPD/26521/2006 and NSF grants DMS 0606343 and DMS 0908093. ND was supported by ERC Bridges project, the Academy of Finland CoE in Analysis and Dynamics Research and an IBM Goldstine fellowship.
}

\date{\today}

\begin{abstract} 
    We study continuity, and lack thereof, of thermodynamical properties for one-dimensional dynamical systems. 
    Under quite general hypotheses, 
    the free energy is shown to be almost upper-semicontinuous: some normalised component of a limit measure will have free energy at least that of the limit of the free energies. From this, we deduce results concerning existence and continuity of equilibrium states (including statistical stability). Metric entropy, not semicontinuous as a general multimodal map varies, is shown to be upper semicontinuous under an appropriate hypothesis on critical orbits.
    Equilibrium states vary continuously, under mild hypotheses, as one varies the parameter and the map. We give a general method for constructing induced maps which automatically give strong exponential tail estimates. This also allows us to recover, and further generalise, recent results concerning statistical properties (decay of correlations, etc.). 
    Counterexamples to statistical stability are given which also show sharpness of the main results. 
\end{abstract}

\maketitle

\tableofcontents 


\chapter{Introduction} \label{sec:intro}

\section{General introduction} \label{sec:generalintro}

An ideal gas is a collection of particles interacting solely via random collisions. Keeping track of individual particles is neither important nor feasible. Rather, macroscopic \emph{thermodynamic} quantities (such as pressure, temperature, free energy) describe the system. One can use statistical mechanics to study dependence of equilibrium states and of thermodynamic quantities on parameters. In a deterministic (and for our purposes, discrete-time) dynamical system $f : X \to X$, orbits of nearby points remain close for a certain amount of time. If the system is expanding, the orbits diverge and become decorrelated. If this decorrelation occurs sufficiently quickly, the system exhibits random characteristics and chaotic behaviour. Individual orbits lose importance, and macroscopic, typical, statistical properties are of interest. In the 1970s, Sinai, Ruelle and Bowen \cite{Sin72, Rue76, Bow75}  introduced \emph{thermodynamic formalism} to dynamical systems to remarkable success\footnote{A striking early result of Ruelle concerned Hausdorff dimension of Julia sets of quadratic complex polynomials \cite{Rue82}, a subject with, \emph{a priori}, no connection to physics.}.  

    In dynamical systems, questions of existence of equilibrium states and continuous dependence of the equilibrium state on parameters (aka \emph{statistical stability}) 
are of prime importance. Uniformly hyperbolic systems are well-understood: existence holds and one even has smooth dependence of the equilibrium state on parameters \cite{KatKPW, Contreras, Rue}. Non-hyperbolic systems are harder to study. Existence need not hold; even when it does, dependence need not be continuous, as we shall see. 
    This article contributes a new tool to the study of such questions in the context of one-dimensional dynamics. 
We give natural and general proofs of new results concerning statistical stability and linear response (topics studied in 
\cite{Kel82, RycSor, Tsu_cty, Frei, Aru, Bal_pwexp, FreT, Ru_Mis, BalSma09, BalBenSch13, AlvSou13, VarVia10})
or lack thereof, and we generalise recent work concerning existence and uniqueness of equilibrium states \cite{IomTeq, PrzRL}. For families of maps we make important contributions, showing under reasonable hypotheses that metric entropy is upper-semicontinuous and that equilibrium states frequently vary continuously.

An equilibrium state is, roughly speaking, an invariant probability measure maximising the free energy. 
The naive approach to find an equilibrium state---taking a sequence of measures whose free energy converges to the supremum---does not work in general. Failure occurs because the free energy does not depend semicontinuously on the measure. We show that this unsophisticated approach has its merits, thanks to some hidden semicontinuity properties.

Our techniques lead us to propose (and prove) the following principle. For a large class of maps and potentials, \emph{the free energy is almost upper-semicontinuous}. Let us explain briefly. Let $f$ be a piecewise-continuous map of the unit interval $I$ say, let $\mu$ be an invariant probability measure with entropy $h(\mu)$,  and let $\phi: I \to [-\infty,\infty]$ be a potential function.    The free energy $E(\mu)$ of $\mu$ is defined to be $h(\mu) + \int \phi \, d\mu$. 
\footnote{This is the convention in  dynamics literature, see  \cite[Chapter 6]{Kellbook} and \cite{BT08}, while 
physicists would consider $E(\mu)$ to be
\emph{minus} the free energy. With our convention, equilibrium states \emph{maximise} the free energy. }
Let $E_+$ denote the supremum of $E(\mu)$ (or equivalently of $\int \phi \,d\mu$) over  ergodic invariant probability measures supported on repelling periodic orbits.
Take a convergent sequence of ergodic invariant probability measures $\mu_n$ with uniformly positive entropies and with $\liminf_{n \to \infty} E(\mu_n) \geq E_+$. 
Then some component of the limit measure, when normalised, should have free energy at least $\limsup_{n \to \infty} E(\mu_n)$. 

    This powerful principle is somewhat surprising as, if the potential is unbounded above,
    $\mu \mapsto \int \phi\, d\mu$ is often lower-semicontinuous, 
    while (for a fixed map) entropy is upper-semicontinuous. 
    In particular,  in the setting where $f$ is $C^1$ with derivative $Df$ and  $f$ has a critical point, the \emph{geometric potential} 
    $-t\log|Df|$ is unbounded above for $t>0$.
    \emph{A priori} and in fact, the free energy is neither upper- nor lower- semicontinuous, yet we have some sort of hidden upper-semicontinuity. 

    What may happen is the following. Lower-semicontinuity of the Lyapunov exponent occurs when diminishing mass approaches the critical set and disappears in the limit. 
    A typical orbit passing near the critical set leads to a bound period (not quite that of \cite{BenCarl}) where there is no freedom in the subsequent stretch of orbit, until one reaches the large scale again. Such orbital stretches considered as a whole make very little contribution to the free energy, so losing them in the limit and normalising the remainder (which requires positive entropy) does not cause free energy to drop.  
    One may hope that some form of almost upper-semicontinuity will hold in higher dimensions; currently we do not have the techniques at our disposition to show this. 

    Hidden semicontinuity, which extends to families of maps, allows us to give strikingly straightforward proofs of existence  and (almost) continuous dependence of equilibrium states for the geometric potentials $-t \log |Df|$, generalising many previous works. 
%
    Continuous dependence as one varies the map is new, to the best of our knowledge, even for the quadratic family (Raith \cite{Raith_stab} 
    has shown that the measure of maximal entropy varies continuously). As a bonus, our techniques often give important statistical properties (in particular, the \emph{almost sure invariance principle} and \emph{exponential decay of correlations}) and we derive regularity of pressure functions.

    Some of our main results are unavoidably quite technical, even to state, witness \S\ref{sec:ausce}. Prior to describing them, and at risk of minor repetition, we shall proceed by presenting more readily accessible ones. The first, upper-semicontinuity of metric entropy, follows. Then we detail a couple of results concerning absolutely continuous invariant probabilities in Section~\ref{sec:introacip}.   
    We define equilibrium states in Section~\ref{sec:introeq} which enables us, in Section~\ref{sec:introquad}, to give applications of our results to the quadratic family.  
    Further applications can be found in \cite{DemTod19, DobMih19}.

    \section{Upper-semicontinuity of metric entropy}
    We shall eventually study  free energy, the  sum of metric entropy and the integral of a potential. 
    For sequences of maps, even metric entropy (of invariant measures) is not necessarily semi-continuous \cite{Mis_jump}. 
    This would appear to render hopeless our endeavour to prove almost upper-semicontinuity of the free energy. However, imposing a topological hypothesis allows us to surmount this obstacle and gives us the following result, 
    postponing formal definitions to Section~\ref{sec:introentropy}. 
         \begin{theorem}[Upper-semicontinuity of metric entropy]\label{thm:introusc}
             Let $(f_k)_{k\geq0}$ be a sequence of piecewise monotone $d$-branched maps converging to $f_0$ as $k \to \infty$ with decreasing critical relations. Suppose $(\mu_k)_{k\geq1}$ is a  convergent sequence of ergodic $f_k$-invariant probability measures. Then
             $$
             h\left(\lim_{k\to\infty}\mu_k\right) \geq \limsup_{k\to\infty} h(\mu_k).$$
         \end{theorem}
    Raith's result \cite{Raith_stab} on continuous dependence of the measure of maximal entropy for unimodal maps is  the only one, of which we are aware, which leads in this direction. 
    The hypothesis, \emph{decreasing critical relations},  
    is relatively weak: for the quadratic family for example, only  super-attracting parameters have an extra critical relation in the limit.

    \section{Absolutely continuous invariant probability measures} \label{sec:introacip}
    Absolutely continuous (with respect to Lebesgue measure) invariant probability measures (\emph{acip}s),\index[def]{Acip} when they exist, are important because they describe the typical long-term dynamics of a positive Lebesgue measure set of points. They tend to be equilibrium states for the potential $-\log|Df|$ (the same as the geometric potential above but with $t=1$). 
    It is natural to ask how these measures (and properties) vary as the parameters of a dynamical system change. The following result follows from  a more general version in Section~\ref{sec:ausce}.
    \begin{theorem} \label{thm:epsquad}
        In the quadratic family, given $\eps>0$ and a sequence  $(f_k)_{k\geq1}$  of maps with $f_k \to f$ as $k\to \infty$, if each map $f_k$ has an absolutely continuous invariant probability measure $\mu_k$ with entropy $h(\mu_k) \geq \eps$, then $f$ has an absolutely continuous invariant probability measure $\mu$ and $(\mu_k)_{k\geq1}$ admits a subsequence converging to an $f$-invariant measure $\mu'$ with $\mu$ absolutely continuous with respect to $\mu'$. 
    \end{theorem}
    One could hope that the limit measure itself would necessarily be absolutely continuous (as, indeed, we did). However, in the other direction we have the following. 
        Consider the family of maps 
    $$f_a : [0,1] \to [0,1],$$ 
    with $a > 0$, defined by 
        $$f_a(x) = \left\lbrace \begin{split}
            1 -2x, \quad &  \text{ if } \quad  x\in [0,1/2];\\
    a(x-1/2)(x-1) + 1, \quad & \text{ if } \quad x\in (1/2,1]. \end{split} \right.
            $$
    \begin{theorem}\label{THM:EXO1BIS} 
        There exists a sequence $(f_{a_k})_{k\geq0}$ of such maps having decreasing critical relations with $f_{a_k} \to f_{a_0}$ 
        as $k \to \infty$ and having the following properties. Each $f_{a_k},$ $k\geq 0$, has an acip $\mu_k$ with entropy uniformly bounded away from $0$; the measures $\mu_k$ converge to a strictly convex combination of the acip $\mu_0$ for $f_{a_0}$  and a Dirac mass on a repelling fixed point.
    \end{theorem}
    The family of maps $(f_a)_a$ can be considered a toy model for a first return map to an interval containing the critical point of a unimodal map. 
    Such a first return map will typically have one unimodal branch and infinitely many non-critical branches. 
    Our toy model maps have one expanding branch and one unimodal branch. We expect that in the quadratic family similar discontinuity of acips will hold.

    Worse things happen without the lower bound on the entropy. Motivated by Ruelle's paper \cite{Ru_Mis} showing smoothness of dependence of the acip on parameter for a family of unimodal Misiurewicz maps which \emph{remain in the same topological class}, we consider what happens if ever one leaves the class: Even restricting to non-renormalisable Misiurewicz parameters, the limit measure can be more or less anything, see Theorem~\ref{THM:THUN2}. 


    \section{Equilibrium states} \label{sec:introeq}
Given a dynamical system $f:I\to I$ on a topological space $I$, and a \emph{potential} $\phi:I\to [-\infty, \infty]$, both of which preserve the Borel structure, we define the (standard, \emph{variational}) \emph{pressure}\index[def]{Pressure} of $(I,f,\phi)$ to be
$$\mathrm{VP}_f(\phi):=\sup\left\{h(\mu)+\int\phi~d\mu:\mu\in \M \text{ and } \int\phi~d\mu> -\infty\right\},
\index{VP@$\mathrm{VP}_f(\phi)$, variational pressure}
$$
    where $h(\mu)=h_f(\mu)$ denotes the (metric) entropy of $\mu$ and ${\M} ={\M}_f$ \index{measures@${\M}_f$, ergodic $f$-invariant probability measures} denotes the set of ergodic, $f$-invariant, probability measures. 
    Spaces of measures are equipped with the weak$^*$ topology. 

    Henceforth, let $I$ be a bounded, non-degenerate interval.
    For a piecewise-$C^1$ one-dimensional map $f$ and $\mu \in \M_f$, the Lyapunov exponent $\lambda(\mu)$ of $\mu$ is  defined as $\int \log |Df|\, d\mu$. 
    Ergodic measures with negative Lyapunov exponent have zero entropy, are typically finite in number and are supported on attracting periodic orbits (at least in the smooth case, see \cite[Proposition~A.1]{Riv12}). They do not contribute much by way of knowledge. Often they are defined away by, for example, assuming that the system is transitive. We shall rather consider the following pressure function, 
    $$
        P(\phi)=P_f(\phi):=\sup_{\mu \in \wtm_f}\left\{h(\mu)+\int\phi~d\mu: \quad \int\phi~d\mu> -\infty  \right\},
        \index{pressure@$P_f(\phi)$}
        $$ 
        where 
\begin{equation}
\wtm = \wtm_f := \left\{\mu \in {\M}_f : \int \log |Df|\, d\mu \geq 0\right\}.\label{eq:meas}
\index{measuresple@$\wtm_f$, $f$-invariant probability measures of non-negative Lyapunov exponent}
\end{equation}
    Results for $\mathrm{VP}_f$ can be deduced from those for $P_f$. 
    If $h(\mu) >0$, then $$\lambda(\mu) \geq h(\mu) >0$$ by Ruelle's Inequality \cite[Theorem~1]{Hof91}. Hence any measure in $\M_f$ with positive entropy is also in $\wtm_f$. 

    As is traditional in dynamical systems \cite{Kellbook}, for an invariant  measure $\mu$, set  
    $$ E(\mu) = E(f, \mu, \phi) := h_f(\mu)+\int\phi~d\mu 
    \index{Ef@$ E(f, \mu, \phi)$, free energy}  
    $$ 
    and call it the \emph{free energy} \index[def]{free energy} of $\mu$ (with respect to $(I,f,\phi)$).  Note $E(\mu + \alpha \mu') = E(\mu) + \alpha E(\mu')$ for $\alpha \in \R^+. $\footnote{Note that the definition of metric entropy extends to any finite measure.}
    A periodic point $x$ of period $p$ is called \emph{(hyperbolic) repelling} if $|Df^p(x)|>1$.\index[def]{Repelling (hyperbolic) periodic point}  We define
     \begin{equation*}
        E_+ = E_+(f) = E_+(f,\phi) := \sup \left\{E(\mu) :\begin{split} \mu\in \widetilde\M_f \mbox{ is supported on a} \\ 
         \mbox{ repelling periodic orbit} \end{split} \right\}.
    \index{Eplus@$E_+(f,\phi)$}
    \end{equation*}
    If a measure $\mu\in \wtm$ \emph{maximises} the free energy over measures in $\wtm$, thus satisfying  $h(\mu)+\int\phi~d\mu=P(\phi)$, then we call $\mu$ an \emph{equilibrium state}\index[def]{Equilibrium state} for $(I,f,\phi)$.  
    Note that we require equilibrium states to be ergodic. 
    
    The family of potentials $\{-t\log|Df|:t\in \R\}$ is of especial interest; the integral $\int \log |Df|~d\mu$ is the Lyapunov exponent of the measure.  For example, since an $f$-invariant measure which is absolutely continuous with respect to Lebesgue measure describes the statistical behaviour of a set of points of positive Lebesgue measure, such measures are of great import.  It was shown in \cite{Ledrap} that if $f$ is $C^2$ with non-flat critical points, then  a measure $\mu$ with positive entropy is an acip if and only if it is an equilibrium state for the potential $-\log|Df|$ (the non-flat condition is unnecessary \cite{Dob08}).  For quadratic maps, Lebesgue measure is ergodic and any acip has positive entropy \cite{BloLyu89}. Moreover, equilibrium states of other potentials $-t\log|Df|$ give information on the Lyapunov spectrum, see for example \cite{IomTod10}.

    The function $t \mapsto P(-t \log |Df|)$ is  convex (it can be viewed as a supremum of lines $t \mapsto E(f, \mu, -t\log|Df|)$) and in standard examples is analytic at all but finitely many values of $t$. A point $t_0$ where it is not analytic is called a \emph{phase transition}. If the function is differentiable at $t$ and has an equilibrium state $\mu_t$, then the slope of the pressure function at $t$ is minus the Lyapunov exponent of $\mu_t$. If $h(\mu_t) >0$, the measure $\mu_t$ has maximal dimension (equal to $h(\mu_t)/\lambda(\mu_t)$ by the Dynamic Volume Lemma \cite{HofRai92}) among measures with the same Lyapunov exponent. 

Let us set $$P^0(\phi) = P^0(f,\phi)  := \sup \{E(f, \mu, \phi) : h(\mu) = 0 \text{ and } \mu \in \wtm_f\}.$$\index{Pzero@ $P^0(f,\phi) $} Of course, $P^0(f,\phi) \geq E_+(f,\phi)$, the supremum in $E_+$ being taken over repelling orbits; for reasonable classes of maps one may expect equality. 
Consider 
        $$t^-:= \inf \left\{t : P(-t \log |Df| ) > P^0 (-t \log |Df|) \right\},
        \index{tplus@$t^-$, $t^+$}$$
       
        $$t^+:= \sup \left\{t : P(-t \log |Df| ) > P^0 (-t \log |Df|) \right\}. 
        $$
        These need not be finite; $t^+$ especially depends strongly (often discontinuously) on the map $f$.   
       %
        As in \cite[Proposition 9.2]{IomTod10}, for example, if there is an acip of positive entropy then $t^+\ge 1$.  On the other hand, even for continuous functions, if there is a wild attractor as in \cite{BruTod15} and \cite[Theorem 10.5]{AviLyu08}, or the map is infinitely renormalisable as in \cite{Dob09}, we can have $t^+\in (0,1)$. There are interesting results on maps for which $t^+ \in (1,\infty)$ by Coronel and Rivera-Letelier \cite{CorRL_low}. If we assume that $f$ has positive topological entropy, then $P(0) > P^0(0)$, so $t^- < 0$ and $t^+ > 0$. For $t \in (t^-, t^+)$, $P(-t\log|Df|) > P^0(-t\log|Df|)$ and any equilibrium measure necessarily has positive entropy. Moreover, whenever $E(f, -t\log|Df|, \mu) > P^0(-t\log|Df|)$ then $\mu$ has positive entropy. One can think of the closure of the zone $P^0 < E < P$ as the region of possible application of our results, see Figure~\ref{fig:pressure}.

        When considering the quadratic family, for example, maps with attracting periodic points are dense (\cite{lyu, GraSw}); without excluding measures with negative Lyapunov exponent from our consideration, $t^+$ would frequently be close to $0$, reducing the space of application of our results. For a single map, one could remove a neighbourhood of an attracting orbit from the domain of definition of the map but, for families, it makes more sense to define away the offending measures.

            \begin{figure}
    \centering
        \def\svgwidth{0.9\columnwidth}
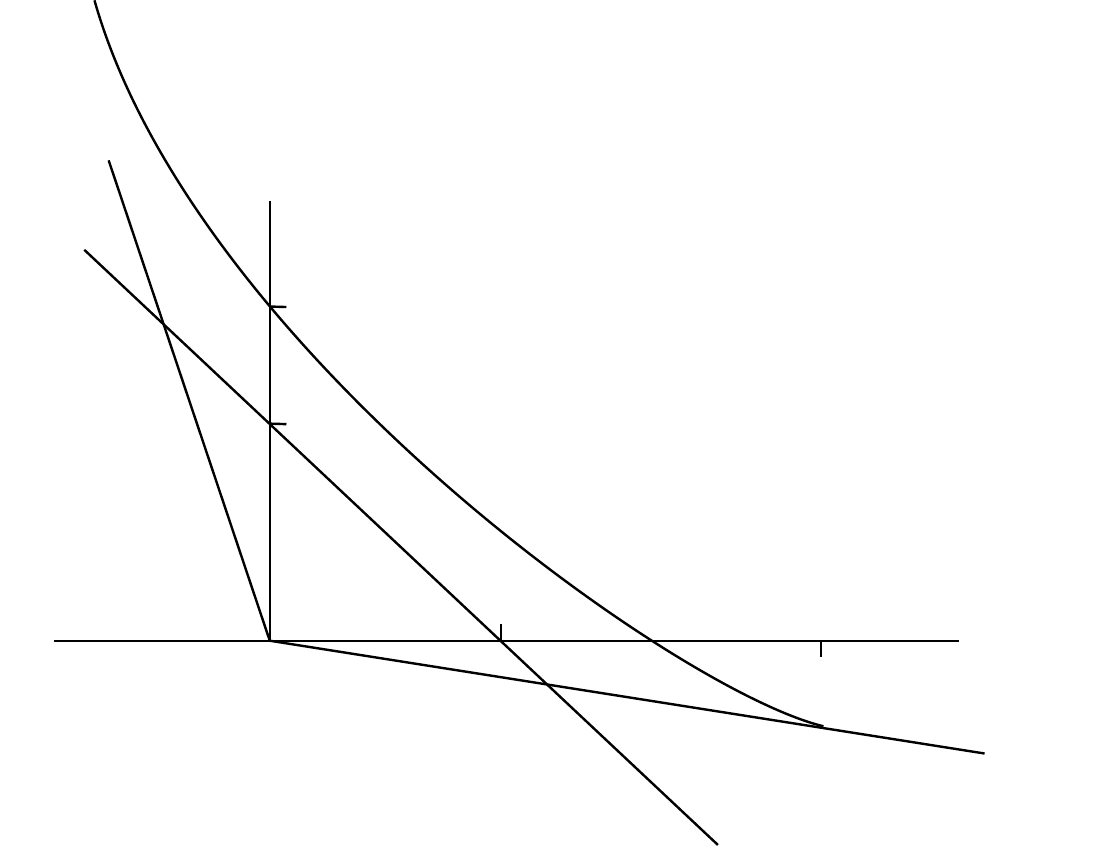
\caption{The zone of almost upper-semicontinuity of the free energy.  For a given $t$, if a sequence of measures $\mu_k$  have $E(f, -t\log|Df|, \mu_k)> P^0(-t\log|Df|)$, our  upper-semicontinuity results apply.}
    \label{fig:pressure}
\end{figure}

    Existence of equilibrium states is non-trivial. For a continuously differentiable interval map with finitely many critical points, for example, the entropy depends upper-semicontinuously on the measure (see for example \cite[Theorem 4.2.4]{Kellbook}), but $\mu \mapsto -\int \log |Df| ~d\mu$ is only lower-semicontinuous (and moreover, as in \cite{BrKel}, there are natural cases where this function is not upper semi-continuous), so if one naively takes a convergent sequence of measures whose free energies converge to the pressure, the limit measure may not maximise the free energy. Limit measures need not  be ergodic.  

\begin{definition} Given a sequence of $f$-invariant probability measures, we call an $f$-invariant measure $\mu$  a \emph{light limit measure} \index[def]{Light limit measure (for a fixed map)} provided some subsequence converges to a measure $\mu'$, $\mu \ll \mu'$ ($\mu$ is absolutely continuous with respect to $\mu$) and $\mu(I) =1$. 
\end{definition}

    In the above, if $\mu'$ were ergodic, then $\mu$ would be $\mu'$; if not, then one could decompose $\mu'$ into two non-trivial $f$-invariant measures; each one, when normalised, would be a light limit measure.  We do not require light limit measures to be ergodic. 

         There is a compatible extension of the definition of light limit measures to families of maps. 
         \begin{definition} \label{def:llm}
             Let $(f_k)_k$ be a convergent sequence of $d$-branched  piecewise-monotone maps (see Definition~\ref{def:dbranch}) with $d$-branched limit $f_0$ and with ergodic $f_k$-invariant probability measures $\mu_k$.
         We say an $f_0$-invariant measure $\mu$ is a \emph{light limit measure}  \index[def]{Light limit measure (for a sequence of maps)} provided some subsequence of the $\mu_k$ converges to a measure $\mu'$, $\mu \ll \mu'$ and $\mu(I) =1$. 
        \end{definition}
        \section{Applications to the quadratic family} \label{sec:introquad}
    The quadratic (or logistic) family \index[def]{Quadratic/logistic family}
    $$
    \F_Q := \{f_a : x \mapsto ax(1-x), \quad a \in [1,4]\}
    \index{FQ@$\F_Q $, the quadratic family}
    $$
    is the primordial test-bed for ideas in non-linear, non-hyperbolic dynamics and is the subject of vast literature, see \cite{MSbook} for example. 
    Phenomena from quadratic dynamics appear in the real world, for example period doubling \cite{Feigenbaum experiments}. Our results apply (and are new) in this context. We present a simplified version of one theorem, with the intention of conveying the concept and motivating the reader to delve further into the article. We first require some definitions.

    Given $\eps >0$ and a potential $\phi$, let us write
    $$
    \M^\eps_f(\phi) := \{ \mu \in \M_f : E(f,\mu,\phi) \geq E_+(f, \phi) \quad \mbox{and} \quad h(\mu)\geq \eps\}.
    \index{measuresgood@$\M^\eps_f(\phi)$, `good' measures}
    $$
    This is the class of measures from which one can hope for some good behaviour. 
    Limits of ergodic measures need not be ergodic. Let us call an $f$-invariant measure $\mu_*$ \emph{hyperbolic}\index[def]{Hyperbolic measure} if $\mu_*$-almost every point $x$ has positive (pointwise) Lyapunov exponent, that is $\lim_{n\to \infty} \frac1n \log |Df^n(x)| >0$. 
    \begin{definition}
        We say the free energy is \emph{almost upper-semicontinuous} \index[def]{Almost upper-semicontinuous} for the quadratic family $\F_Q$ and the potentials $-t\log |Df|$ ($f \in \F_Q$, $t \in \R$)
        if, 
        for every $\eps >0$,  for every sequence $(a_k, t_k, \mu_k)_k$ such that
        \begin{itemize}
            \item
                $(a_k)_k$ converges to $a_0 \in [1,4]$ as $k \to \infty$,
            \item
                 $(t_k)_{k}$ converges to $t_0 \in \R$ as $k \to \infty$
            \item
                and  the measures $(\mu_k)_k$ satisfy 
                $$\mu_k \in \M_{f_{a_k}}^\eps(-t_k \log|Df_{a_k}|),$$ 
        \end{itemize}
        there exists a hyperbolic light limit measure $\mu_*$ with
        $$E(f_{a_0}, \mu_*, -t_0\log |Df_{a_0}|) \geq \limsup_{k \to\infty} E(f_{a_k},\mu_k, -t_k \log |Df_{a_k}|).$$
    \end{definition}

    \begin{theorem}\label{thm:feusc}
        The free energy is almost upper-semicontinuous for the quadratic family and the potentials $-t\log|Df|$ with $t\in \R$.
    \end{theorem}
    The proof of existence of equilibrium states $\mu_t$, $t \in (t^-,t^+)$ for the quadratic family (and more general multimodal maps) has a long history, with partial results by Hofbauer, Bruin and Keller, Bruin and Todd, Pesin and Senti \cite{HofIntrinsic, Hof81, BrKel, BT08, BTeqnat, PeSe} before Iommi and Todd \cite{IomTeq} finally proved existence for all $t$ in this range. Recently Przytycki and Rivera-Letelier \cite{PrzRL} gave another proof and also showed analyticity of the pressure function and statistical properties, with all these results holding for fixed maps. 
    Another proof follows from our work: one can always take a sequence of measures with free energies converging to the pressure; 
    sometimes (in particular, if $t \in (t^-,t^+)$) all measures with large free energy have uniformly positive entropy; 
    from Theorem~\ref{thm:feusc}, existence falls out naturally. 
    \begin{corollary} \label{cor:eqstate1}
For $f \in \F_Q$ and $t \in (-\infty, t^+)$, there exists an equilibrium state $\mu_t$ for the potential $-t\log|Df|$. 
            \end{corollary}
    If statistical properties of a dynamical system are described by invariant measures, it is natural to ask how these measures (and properties) vary as the parameters of a dynamical system change. 
            
    If, for example, a quadratic map has a \emph{hyperbolic attracting} periodic orbit (\ie with negative Lyapunov exponent), then all nearby maps do too, and the periodic orbit moves smoothly with the parameter. The equidistributions on the periodic attractors are ergodic invariant probability measures which describe the typical long-term dynamics of almost every point, since the basins of attraction have full measure \cite{MSbook}. As these measures vary continuously, we say that statistical stability holds at (\emph{hyperbolic}) parameters with hyperbolic attracting\index[def]{attracting (hyperbolic) periodic points} periodic orbits. 

            At non-hyperbolic parameters, the dynamics is more interesting. 
    Absolutely continuous invariant probability measures, when they exist, describe the typical long-term dynamics of a positive-measure set of points. 
    The question then arises as to the  dependence of acips
    as a function of parameter.
    There is a natural measure on $\F_Q$ corresponding to Lebesgue measure on parameter space. 
    Let $\F$ be a subset of $\F_Q$ and suppose that for all $f \in \F$, $f$ has an acip (acips are always unique for maps in the quadratic family). We will say that acips are \emph{statistically stable} \index[def]{Statistically stable} in $\F$ if the acip depends continuously on the parameter. 

    Rychlik and Sorets \cite{RycSor} proved existence of such a set $\F$ having positive measure, where continuous dependence holds at Misiurewicz maps, a non-trivial but zero-measure subset of maps.  For a positive measure set $\F \subset \F_Q$, Tsujii used Benedicks-Carleson techniques to show statistical stability of acips in $\F$ \cite{Tsu_cty}. This was extended by Freitas and Todd to the more general case where $\F$ is any set of non-renormalisable parameters satisfying a uniform summability condition for the derivatives along the critical orbit \cite{FreT}. Recently, Baladi \emph{et al} showed H\"older continuity in Tsujii's setting \cite{BalBenSch13}. On the other hand, Tsujii and Thunberg \cite{Tsu_cty, Thun} showed how renormalisation can lead to instability; we shall present further obstructions to statistical stability, but first we shall give a positive result. 

   \begin{definition}[Statistical quasistability]
        For each $\eps >0$, denote by $\F^\eps_Q$ the subfamily of maps $f$ in $\F_Q$ for which $f$ has an acip with entropy  $\geq \eps$.  
        We shall say acips are \emph{statistically quasistable} \index[def]{Statistically quasistable}  in the quadratic family $\F_Q$ if, for each $\eps >0$,  for any $f_0 \in \F^\eps_Q$ and sequence $(f_k)_{k \geq 1}$ of maps in $\F^\eps_Q$ converging to $f_0$, some light limit measure of the corresponding sequence of acips is the acip for $f_0$. 
    \end{definition}
    Recalling that acips in the quadratic family are equilibrium states for the potential $-\log|Df|$, Theorem~\ref{thm:feusc} implies the following result, equivalent to Theorem~\ref{thm:epsquad}. Note that there is no assumption on the behaviour of the (derivatives along the) critical orbit. 
    \begin{corollary} \label{cor:acip stab}
        Acips are statistically quasistable in the quadratic family $\F_Q$.
    \end{corollary}

There was hope in the community that statistical stability (and more) might hold for non-renormalisable Collet-Eckmann parameters (those for which the derivative grows exponentially along the critical orbit), see \cite[Conjecture B]{Bal_pwexp}. However, even for the better-behaved Misiurewicz parameters, one has instability [even worse, the limit measure can be almost anything at all, see Theorem~\ref{THM:THUN2}:
    \begin{theorem} \label{thm:MisUnstable}
		Statistical stability does not hold for acips for the subfamily of non-renormalisable Misiurewicz quadratic maps. 
	\end{theorem}
        It follows from the work of Tsujii \cite{Tsu_PosLyapExp, Tsu_cty}, though justification is beyond the scope of this paragraph, that near (almost) every non-renormalisable Misiurewicz parameter there is a positive measure set of non-renormalisable Collet-Eckmann parameters whose acips are very close to that of the Misiurewicz parameter. Compare the discussion in \S\ref{sec:failure}.
        From this and Theorem~\ref{thm:MisUnstable},  one can deduce the following. 
        \begin{corollary} \label{cor:thun3}
		Statistical stability does not hold for acips for any (relatively) full-measure subset of non-renormalisable Collet-Eckmann parameters in the quadratic family. 
	\end{corollary}

        Note that statistical quasistability is weaker than statistical stability, in that only some component of the (not necessarily ergodic) limit measure need be the acip of the limit map. 
        While the examples of Theorem~\ref{THM:THUN2} have entropy decreasing to zero, uniform positive entropy does not imply stability, as we show in Theorem~\ref{THM:EXO1BIS}. One cannot expect better than quasistability --  for actual statistical stability one needs stronger hypotheses. Consider a sequence of (maps and) acips with uniform positive entropy. Theorem~\ref{thm:feusc} says that some light limit measure will be an acip. However, another component of the limit measure may be an atom at a repelling fixed point for the limit map, for example.

        \section[Piecewise-monotone families and metric entropy]{Piecewise-monotone families and upper-semicontinuity of metric entropy} \label{sec:introentropy}
    A quadratic map is a continuous $2$-branched piecewise-monotone map. We shall study a much broader class of maps, admitting discontinuities. 
    In Sections~\ref{sec:ausce}--\ref{sec:npesq} we present results concerning almost-semicontinuity and statistical stability  for general piecewise-monotone families. There are also results pertaining to the thermodynamic formalism for an individual map.

    \begin{definition} \label{def:dbranch}
    For $d \geq 2$, a map $f$ defined on a pairwise-disjoint collection $I_1, I_2, \ldots, I_d$ of non-degenerate subintervals of a bounded interval $I$ is called a \emph{($d$-branched) piecewise-monotone map} \index[def]{dbranchedmap@$d$-branched piecewise-monotone map} if $f$ maps each $I_j$ into $I$ and, on each $I_j$, $f$ is continuous and strictly monotone. 
    \end{definition}
    We shall assume throughout the article, without loss of generality, that $d \geq 2$. Otherwise the entropy of the map, and thus of any invariant measure, would be zero. 
    
    In the presence of more than one such map, we write $I_j(f)$ for $I_j$. A branch $I_j$ of $f$ is called \emph{full} if it is mapped by $f$ onto $I$.
        We do not specify here whether the subintervals $I_j$ are open or closed. Since the domain of $f$ need not be the whole interval $I$, iterates for some points may not be defined. However, for any invariant probability measure, almost every point has all iterates defined. Moreover, almost every point is recurrent. 
        
\begin{definition} 
We call $\F$ a \emph{$d$-branched piecewise-monotone family} \index[def]{dbranchedfam@$d$-branched piecewise-monotone family} on $I$ if each $f \in \F$ is a $d$-branched piecewise-monotone map.

        $\F$ is called a \emph{piecewise-monotone family} if it is a $d$-branched piecewise-monotone family for some $d \geq 2$. 
        \label{def:PMF}
    \end{definition}

    \begin{definition} \label{def:pmfconv}
Let $\F$ be a $d$-branched piecewise-monotone family.  
We say a sequence of $f_k \in \F$ \textit{converges to $f$ in $\F$} \index[def]{Convergence in a $d$-branched piecewise-monotone family} as $k \to \infty$ if $f \in \F$ and there is a sequence of homeomorphisms $h_k: I \to I$ such that, for $1 \leq j \leq d$, 
        \begin{itemize}
            \item $h_k(I_j(f_k)) = I_j(f)$, 
            \item $h_k\circ f_k \circ h_k^{-1}$  converges in $C^0$ to $f$ (on the domain of definition of $f$) as $k\to \infty$,
            \item $h_k$ converges in $C^0$ to the identity on $I$ as $k\to \infty$.
        \end{itemize}
    \end{definition}
    This definition is equivalent to uniform convergence on compact subsets of 
    $$ \bigcup_{j=1}^d \mathring{I}_j(f),$$
    where $\mathring{I}_j(f)$ denotes the interior of $I_j(f)$, provided the $f_k$ and $f$ are uniformly Lipschitz.  
            For $f$ as above, set $$
            \E(f) := \bigcup_{j=1}^d  \partial I_j.$$
        
            The Hofbauer extension, a Markovian extension of an original system, was introduced in  \cite{Hpwise}. This powerful idea lies behind a host of results in one-dimensional dynamics. The structure of the Hofbauer extension depends on the post-critical orbits. 
    In recent work on statistical stability \cite{FreT} where Hofbauer extensions were used, the limit maps studied were assumed not to have critical orbits which intersect or self-intersect. That unsatisfactory assumption excludes post-critically finite maps and many more; it was necessary to obtain convergence of Hofbauer extensions. We introduce the following definition to deal with this issue, eventually embedding Hofbauer extensions in some ambient space and examining convergence properties there. 
    \begin{definition} \label{def:deccr}
        Let $f$  be a piecewise-monotone map and let $c \in \E(f)$. 
        Let $\sigma \in \{-1, +1\}$. Suppose there is some minimal $n \geq 1$ for which 
        $$
        \lim_{\varepsilon \to 0^+} f^n(c+ \sigma \varepsilon) = c' \in \E(f).
        $$
        We call this a \emph{critical relation}\index[def]{Critical relation} of order $n$. 
        This is a purely topological condition, defined by $(c, c', \sigma, n)$.  We say that a sequence $(f_k)$ such that $f_k\to f_0$ as $k\to\infty$ has \emph{decreasing critical relations} if to each critical relation $(c_0, c_0', \sigma, n)$ of $f_0$ there is a corresponding  critical relation $(c_k, c_k', \sigma, n)$ of $f_k$, where $h_k^{-1}(c_0) = c_k$ and $h_k^{-1}(c_0') = c_k'$.  
    \end{definition}

	Of course, a sequence will have decreasing critical relations if all maps from the sequence  share the same critical relations.  
        For multimodal maps, there may be many critical relations. This is not true for unimodal maps. 
        \begin{remark}[Critical relations for unimodal maps] \label{rem:Crum}
            Consider a 
        quadratic  map $f_a : x \mapsto ax(1-x)$. It maps the interval $[0,1]$ into itself for $a \in (3,4]$, the parameters of interest. Extend the domain of definition to $[-\eps,1+\eps]$ for some $\eps>0$. Then the boundary points are mapped outside the domain so they contribute no critical relations. The orbit of the critical point $\frac12$ never leaves $[0,1]$, so it can only cause a critical relation if it is periodic and (thus) super-attracting.   
        These critical relations are negligible: 
    If one is only interested in positive-entropy measures,  no mass lies in the basin of attraction. Remove a small neighbourhood of $\frac12$ contained in the basin of attraction and the critical relation disappears. If the neighbourhood shrinks to nothing quickly, as one changes the parameter, the conditions for being a piecewise-monotone family remain verified. 
    We only used the $C^1$ nature of quadratic maps, so the same holds for general $C^1$ unimodal maps. 
\end{remark}

As discussed in~\S\ref{sec:Kel}, some assumption regarding critical relations is necessary in our work. In particular, 
there are examples showing metric entropy is not necessarily upper-semicontinuous as one varies both map and measure. On the other hand, Raith has shown that for unimodal maps $f$, under a condition on the orbits of points in $\E(f)$ which implies decreasing critical relations, the measure of maximal entropy varies continuously \cite{Raith_stab} (hence topological entropy is upper semi-continuous). This admits a strong generalisation, Theorem~\ref{thm:introusc} on upper semicontinuity of metric entropy, proven in Chapter~\ref{sec:llm}.

         \section{Almost upper-semicontinuity of free energy}
         \label{sec:ausce}
         Theorem~\ref{thm:SSSS}, below, is the principal technical result of the article; its proof occupies~Chapters~\ref{sec:cylchap}-\ref{sec:SSSS} and part of Chapter~\ref{chap:Kat}. Its statement requires the introduction of regularity of piecewise-monotone families.

        \begin{definition}
We that say a piecewise-monotone map $f$ has \emph{non-positive Schwarzian derivative} if, 
on the interior of each  $I_j$, $f$ is a $C^2$ diffeomorphism and $1/\sqrt{|Df|}$ is convex. If the convexity is strict, the map has \emph{negative Schwarzian derivative.} \index[def]{Schwarzian derivative}
        \end{definition}
        Non-positive \cite[p.24]{MisIHES1981} or negative Schwarzian derivative \cite[{\S}IV.1]{MSbook} is a standard property which gives good distortion control for iterates via the Koebe Lemma (see Lemma~\ref{lem:koebe}). It is not especially restrictive (\cite{Koz00}). 
Maps from the quadratic family $f_a : x \mapsto ax(1-x)$ with $a \in (0,4]$ are ($2$-branched) piecewise-monotone maps with negative Schwarzian derivative.  
One can readily check that maps with non-positive Schwarzian derivative have bounded derivative. 
The decreasing critical relations hypothesis allows one to obtain uniform Koebe space for some natural induced maps (see Lemma~\ref{lem:kappaPk}).  

        \begin{definition}
            Let $\F_\mathrm{NSD}$ \index{FNSD@$\F_\mathrm{NSD}$, piecewise monotone families with non-positive Schwarzian derivative} denote the class of piecewise-monotone families $\F$ such that for each $f \in \F \in \F_\mathrm{NSD}$, 
$f$ has non-positive Schwarzian derivative and $\sup_{f \in \F} \sup|Df| < +\infty$. 
        \end{definition}
        
By Definition~\ref{def:PMF},  each element of $\F_\mathrm{NSD}$ is a $d$-branched family for some $d$.  In our applications of this notion, we will often take some sequence $(f_k)_k\in \FNSD$.

\begin{theorem}[Free energy is almost upper-semicontinuous] \label{thm:SSSS}
    Let 
 $(f_k)_{k\ge 0} \in \FNSD$ be a sequence converging to $f_0$ as $k \to \infty$ with decreasing critical relations. 
 Suppose  that 
\begin{enumerate}[label=({\alph*}), itemsep=0.0mm, topsep=0.0mm, leftmargin=7mm]
\item $t_k \to t_0$ as $k \to \infty$; 
\item $(\mu_k)_{k\geq1}$ is a convergent sequence of measures $\mu_k\in \wtm_{f_k}$;
\item $E(f_k, \mu_k, -t_k \log |Df_k|)$ converges to a limit $E_L$ as ${k \to \infty}$;
\item $E_L \geq \liminf_{k \to \infty} E_+(f_k, -t_k\log|Df_k|) $; 
\item $\limsup_{k\to\infty} h(\mu_k) >0$.
\end{enumerate}
Then, writing $E(\mu_*)$ for $E(f_0,\mu_*, -t_0 \log |Df_0|)$, 
some light limit measure $\mu_*$ is hyperbolic and satisfies one of the following statements:
    \begin{itemize}
     \item
         $E(\mu_*) > E_L$;\\ 
        \item 
            $\mu_* = \lim_{k\to \infty} \mu_{k}$ and,
            for some strictly increasing subsequence $(k_n)_n$, \\
            $\int \log |Df_0| ~d\mu_* = \lim_{n \to \infty} \int \log |Df_{k_n}| ~d\mu_{k_n}$, \\ $h(\mu_*) \geq \limsup_{n\to\infty} h(\mu_{k_n}) $ and  $E(\mu_*) \geq E_L$; \\  
     \item
         $E(\mu_*) \geq E_L$ and $h(\mu_*) > \limsup_{k\to\infty} h(\mu_k)$.
    \end{itemize}
If $E_L= P(-t_0 \log |Df_0|)$,  one of the last two alternatives holds and some light limit measure is a positive entropy  equilibrium state for this potential. 
    \end{theorem}
    Consequent results, detailed below and proven (where not immediate corollaries) in~\S\ref{sec:ancillary}, follow with little extra work.

    Theorem~\ref{thm:SSSS} and 
    Remark~\ref{rem:Crum} 
    imply the following corollary, equivalent to 
    Theorem~\ref{thm:feusc}. 
  
    \begin{corollary} \label{cor:ssssquad}
        The statement of Theorem~\ref{thm:SSSS} holds for the quadratic family $\F_Q$, without the decreasing critical relations hypothesis. 
    \end{corollary}
    We shall later treat analyticity of the pressure function and statistical properties, but for now let us just consider existence of equilibrium states. Existence was shown in \cite[Theorem~A]{IomTeq} and again in \cite{PrzRL}. Compared with those settings, we allow discontinuities, parabolic points and holes (\cite{PrzRL} allows holes), while both \cite{IomTeq,PrzRL} impose a 
    \emph{non-flatness} condition on critical points (see also~\S\ref{sec:nsdweak}).  
    The following corollary  implies Corollary~\ref{cor:eqstate1}.

    \begin{corollary} \label{cor:eqstate2}
    Given a $d$-branched map $f$ with non-positive Schwarzian derivative,  for $t \in (-\infty, t^+)$, there exists an equilibrium state $\mu_t$ for the potential $-t\log|Df|$. 
                For $t \in (t^-, t^+)$, there exists such $\mu_t$ with positive entropy. 
            \end{corollary}
            From the pressure function, one can read off certain properties. 
                For example, the following result generalises \cite[Proposition~1.1]{IomTeq}. 
                \begin{proposition} \label{prop:IomTeq}
        The map $t \mapsto P(-t \log |Df|)$ is differentiable at $t^+$ if and only if there is no equilibrium state with positive entropy for the potential $-t^+ \log |Df|$. The same holds for $t^-$ in place of $t^+$.
    \end{proposition}
 
    Concerning statistical stability of equilibrium states, we obtain strong, new results. 

    \begin{theorem}[Smooth pressure and uniform pressure gap imply statistical stability] \label{thm:eqstatstab}
 Let 
 $(f_k)_k \in \FNSD$ be a sequence converging to $f_0$ as $k \to \infty$ with decreasing critical relations. 
 Take $\eps >0$ and assume that
    \begin{enumerate}[label=({\alph*}), itemsep=0.0mm, topsep=0.0mm, leftmargin=7mm]
        \item $t_k \to t_0$ as $k \to \infty$;
        \item for each $k \geq 0$, $P(f_k, -t_0\log|Df_k|) \geq P^0(f_k, -t_0\log|Df_k|)+\eps$.
        \end{enumerate}
        Then for each $k\geq0$, there is an  equilibrium state $\mu_k$ for $f_k$ and the potential $-t_k\log|Df_k|$. 
        If $t \mapsto P(f_0, -t\log|Df_0|)$ is differentiable at $t_0$, then any limit measure of the sequence $(\mu_k)_k$ is a convex combination of equilibrium measures for $f_0$ with the potential $-t_0\log|Df_0|$.
    \end{theorem}
    Jumping ahead a little, for non-renormalisable quadratic maps, $f$ is transitive on $J(f)$ so positive entropy equilibrium states are unique, see Theorem~\ref{thm:eqstates}. Moreover the pressure is analytic on $(t^-,t^+)$, see Theorem~\ref{thm:anal}. 
    In the quadratic family $\F_Q$, we immediately obtain:
    \begin{corollary} [Statistical stability in the quadratic family] Suppose that $f_0 \in \F_Q$ is  non-renormalisable, $t_0 \in (t^-,t^+)$ and $P(f_0, -t_0\log|Df|) >0$. For $(t,f)$ in a neighbourhood of $(t_0,f_0)$ in $\R\times \F_Q$, an equilibrium state $\mu_{t,f}$ (for $f$ and the potential $-t\log|Df|$) exists.  The map $(t,f)\mapsto \mu_{t,f}$ is continuous at $(t_0,f_0)$.
    \end{corollary}
    This significantly improves upon \cite{FreT}, where stability was shown for $t$ close to $1$ under multiple further hypotheses, including non-uniform expansion, and upon \cite{Raith_stab}, where stability was shown when $t=0$. The result would be false if we did not exclude measures supported on periodic attractors from being equilibrium states.
    
    \section{$J(f)$ for a piecewise-monotone map $f$}
    \label{sec:EandJ}
     
        \begin{definition}
            Given a piecewise-monotone map $f : \cup_{j=1}^d I_j \to I$, the set of \emph{critical points} is denoted 
            $$
            \E(f) := \bigcup_j \partial I_j, \index{Ecrit@$\E(f)$ Set of critical points} \index[def]{Critical points}
            $$
            while the set of \emph{critical values} is denoted
            $$
            \V(f) := \bigcup_j \partial f(I_j).  \index{Vcrit@$\V(f)$ Set of critical values}  \index[def]{Critical values}
            $$
        \end{definition}
        
        \begin{definition}
            Given a piecewise-monotone map $f$, we denote by $J(f)$ \index{Julia@$J(f)$, analogue of Julia set} the set of points $x$ such that
        \begin{itemize}
            \item
                $x$ is recurrent, so $x \in \overline{\{f^n(x)\}_{n\geq 1}}$;
            \item
                $x$ is accumulated on both sides by points from $\cup_{n\geq 0} f^{-n}(\E(f))$, that is, 
                for all $\varepsilon >0$, there exist $l, r \in \cup_{n\geq 0} f^{-n}(\E(f))$ with $x -\varepsilon < l < x < r< x+\varepsilon$; 
            \item
                $f^n(x) \notin \E(f)$ and $f^{-n}(f^m(x)) \cap \V(f) = \emptyset$ for each $n,m  \geq 0$; 
            \item
                $x$ is accumulated on both sides by points with the above three properties.
            \end{itemize}
        \end{definition}
        It is easy to check that $J(f)$ is forward-invariant, $f(J(f)) \subset J(f)$. For any ergodic, invariant, positive-entropy measure, the set $J(f)$ has full measure -- we shall justify this statement in Lemma~\ref{lem:SMB}.  On the other hand, some zero-entropy measures may give no mass to $J(f)$. 
        This definition of $J(f)$ is related to, but more restrictive than, the definition of Julia set used in \cite{Riv12}, 
        the main differences being recurrence and that we require our accumulation from both sides. Of course, $J(f)$ is not a closed set.  

\begin{remark}
    Quadratic maps $f_a : x \mapsto ax(1-x)$ with $a \in (3,4)$ are not transitive on the unit interval $[0,1]$. However, for certain values of $a$, in particular when $f_a$ is non-renormalisable, they may be transitive on the interval $[f_a^2(1/2), f_a(1/2)]$. In such cases, $\overline{J(f_a)}$ would coincide with this forward-invariant \emph{dynamical core} $[f_a^2(1/2), f_a(1/2)]$ and $f_a$ would be transitive on $J(f_a)$. The only ergodic invariant probability measure outside the core would be the Dirac mass at $0$. 
    If one extends the domain to $[-\eps, 1+\eps]$, exclusion of isolated recurrent points means $0 \notin J(f_a)$.
\end{remark}

\section[Equilibrium states, pressure and correlations]{Uniqueness of equilibrium states, analyticity of pressure, decay of correlations}

One could ask how many equilibrium states there are for a given potential, compare \cite{Pin11} (see \cite{Hof81} for the case of constant potentials). Depending on the potential and the map, the answer may be infinitely many (see the final paragraph of~\cite{Dob09}). Restricting to those with uniformly positive entropy, we shall provide a finite bound on the number; this bound only depends on the number of branches and the entropy. With a transitivity assumption (on the chaotic part of phase space), there is only one equilibrium state. 
The bound allows us to pass from the final statement of Proposition~\ref{prop:SSSS}, where one finds a convex combination of equilibrium states, to Theorem~\ref{thm:SSSS}, where one finds a light limit measure which is an equilibrium state.  
\begin{theorem}[Bounds on the number of equilibrium states] \label{thm:eqstates}
    Given $d \geq 2$ and  $\varepsilon>0$, there is a number $N \in \N$ such that the following holds. If $f$ is a $d$-branched piecewise-monotone map 
     with non-positive Schwarzian derivative and $t \in \R$, there are at most $N$ equilibrium states with entropy greater than $\varepsilon$ for the potential $-t \log |Df|$. 

     If $f$ is transitive on $J(f)$, then there is at most one equilibrium state with positive entropy for the potential $-t \log |Df|$. 
\end{theorem}

Understanding the pressure function $t\mapsto P_f(t\phi)$ 
    and its smoothness properties can give information on the statistical properties (large deviations and multifractal spectra, for example) of the system $(I,f)$. 
 We note that if there exists  an equilibrium state $\mu_{s}$  for $s\phi$, and $t\mapsto P(t\phi)$ is differentiable at $t =s$, then
\begin{equation}
\left.\frac{\partial P(t\phi)}{\partial t}\right|_{t=s} = \int\phi~d\mu_{s}.
\label{eq:pres der}
\end{equation}

The proofs of the following two results depend on  thermodynamic formalism for shifts on an infinite alphabet. They correspond to \cite[Theorem~A]{PrzRL}. Again recall, we do not exclude parabolic points and do not assume non-flatness of critical points. 
            \begin{theorem}\label{thm:anal}
        Let $f$ be a piecewise-monotone map with non-positive Schwarzian derivative.  
        If $f$ is transitive on $J(f)$, the pressure function $t \mapsto P(-t \log |Df|)$ is real-analytic on the interval $(t^-, t^+)$. 
    \end{theorem}

Suppose that $(I, f, \mu)$ is an ergodic dynamical system.
     Following \cite[Theorem~1.2]{MelNic05}, we say that  $\phi : I \to \R$ satisfies the Almost Sure Invariance Principle (\emph{ASIP}) \index[def]{ASIP} if there exists $\gamma >0$, a sequence of random variables $\{S_N\}_N$ and a Brownian motion $W$ with variance $\sigma^2 \geq 0$ such that
    $$
    \left\{ \sum_{j=0}^{N-1} \phi \circ f^j\right\}_N = \{S_N\}_N \quad \mbox{in distribution} 
    $$
    and, almost everywhere,
    $$ 
    S_N = W(N) + O(N^{\frac12 -\gamma}) \quad \mbox{as } N \to \infty.
    $$
    
    For a pair of function spaces $\C_1, \C_2$,  we say that we have \emph{decay of correlations}, \index[def]{Decay of correlations} against $(\C_1, \C_2)$ observables, if there exists a function $\rho:\N\to [0, \infty)$ such that $\rho(n)\to 0$ as $n\to \infty$ such that for each pair $\phi\in \C_1$, $\psi\in \C_2$ there exists a constant $C_{\phi, \psi}>0$ such that for any $n\in \N$,
    $$\left|\int (\phi\circ f^n)\psi~d\mu-\int\phi~d\mu\int\psi~d\mu\right|\le C_{\phi, \psi}\rho(n).$$
    If $\rho(n)=O(e^{-\alpha n})$ for some $\alpha>0$, we say that we have exponential decay of correlations.

    For $\beta>0$, let $\H_\beta$ denote the set of $\beta$-H\"older continuous observables
     $\phi : I \to \R$.  

    \begin{theorem}[Decay of correlations and the ASIP]        \label{thm:ASIP DCor}
        Let $\beta >0$ and let $f$ be a piecewise-monotone map with non-positive Schwarzian derivative.  
        Assume $f$ is  topologically mixing on $J(f)$. For each $t \in (t^-, t^+)$, the (unique) equilibrium state for the potential $-t \log |Df|$ has exponential decay of correlations against $(L^\infty, \H_\beta)$ observables. The ASIP holds for zero-mean observables in $\H_\beta$. 
    \end{theorem}

    \section{Necessity of positive entropy for statistical quasistability}
\label{sec:npesq}
    A $C^2$ map  $f  \in \F \in \F_\mathrm{NSD}$  is called \emph{Collet-Eckmann} if for each critical value $v\in \V$, all forward iterates are defined and $\liminf_{n \to \infty} \frac{1}{n} \log |Df^n(v)| >0$.  A $C^2$ map  $f \in \F\in \F_\mathrm{NSD}$  is called \emph{Misiurewicz}\index[def]{Misiurewicz maps} if all periodic points are hyperbolic repelling and if,  for all critical points $c$ (in the standard sense that $Df(c) =0$), the orbit of $c$ avoids a neighbourhood of the critical set.  
By Ma\~n\'e's theorem (see for example \cite[Theorem III.5.1]{MSbook}) or \cite{MisIHES1981},  any Misiurewicz map is also Collet-Eckmann. A $C^2$ Collet-Eckmann  map $f \in \F \in \F_\mathrm{NSD}$ with non-flat critical points has an acip \cite{ColEck83}.

 Sometimes, given a piecewise-monotone map of an interval, a subinterval will be forward-invariant under some iterate of the map.
One refers to this property as \emph{renormalisability}, see \cite[\S II.5a]{MSbook} for definitions and details. 
 One can study the dynamics  restricted to such forward-invariant subintervals independently from  the dynamics of points that never enter these subintervals. 
Repeated renormalisability entails a stratification of the phase space. As a potential varies, equilibrium states may jump from one level to another in a non-smooth fashion, as  considered in \cite{Dob09}. 
 This does not happen for non-renormalisable maps under reasonable transitivity assumptions. 
 
 Our next theorem goes in the other direction to Theorem~\ref{thm:epsquad} (equivalently, Corollary~\ref{cor:acip stab}): even if we restrict to the well-behaved class of non-renormalisable  Misiurewicz maps, we do not have statistical quasistability of acips without (for example) the uniform entropy condition, once we leave the topological class. Contrast with~\cite{Ru_Mis}, which gives linear response when one remains within the class and suggested it may hold if one remains tangential to the class. 
  
  \begin{theorem} [Varying entropy implies instability]\label{THM:THUN2}
      Let $\F\in \F_\mathrm{NSD}$ denote the class of $C^2$ unimodal maps $f$  with non-positive Schwarzian derivative, having $f(\partial I) \subset \partial I$ and $|Df_{|\partial I}| > 1$, endowed with the $C^0$ topology. 
      Consider a continuous one-parameter family of maps $(f_t)_{t \in [0, \theta)} \subset \F$, for some $\theta >0$. Suppose that $f_0$ is non-renormalisable and has all periodic points repelling, and that the topological entropy of $f_t$ is not locally constant at $t=0$. Let $C >0$ and suppose that for all $t \in [0, \theta)$, $\int_I \log |Df_t(x)| dx > -C$. 
    Let $p$ be a non-boundary periodic point of $f_0$.

    Then there is a sequence $t_k \to 0$ for which each $f_{t_k}$ is a non-renormalisable Misiurewicz map with acip $\mu_k$  and for which the measures $\mu_k$ converge to the equidistribution on the orbit of $p$. Moreover, $h(\mu_k) \to 0$. 
    \end{theorem}
    Since any $f_0$-invariant probability measure (not supported on $\partial I$) can be approximated by equidistributions on periodic orbits (noting $f_0$ is Misiurewicz, this is easy to show), there are $t_k$ such that the $\mu_k$ converge to whatever $f_0$-invariant probability measure we please.  That the entropy tends to zero follows from Theorem~\ref{thm:SSSS}. 

As a remark, the condition on the integral of $\log |Df_t(x)|$ just says that the critical points are uniformly not too flat. It is clearly satisfied by non-trivial quadratic maps of the interval. See \cite{BM:flat, Dob08} for results concerning existence of acips under this condition and non-existence when it fails. 
Theorem~\ref{THM:THUN2} immediately implies Theorem~\ref{thm:MisUnstable}.
    
Equilibrium states do not vary continuously (in general), even if their entropy is bounded away from zero, see Theorem~\ref{THM:EXO1BIS}.




\section{Background}
For many classes of uniformly hyperbolic dynamical systems, statistical stability is well-established (in this article we will think of a dynamical system as the dynamics $f:I \to I$, plus some associated measure $\mu$ on $I$).  In particular, stability holds for systems where the dynamics $f:I \to I$ has an underlying (finite) Markov structure and the measures $\mu_f$ are equilibrium states for H\"older potentials $\phi_f:I \to \R$.
Indeed, for many such systems, it is known that the measures do not
merely vary continuously with the map, but actually vary
differentiably. 
For example, if $f_t:M\to M$
are $C^3$ Axiom A diffeomorphisms of a manifold $M$, each with a unique
\emph{physical} measure $\mu_t$, and the family $t \mapsto f_t$ is $C^3$,
then the map $t \mapsto \int\psi~d\mu_t$ is differentiable at $t=0$
for any real-analytic observable $\psi:M\to \R$, see \cite{KatKPW, Contreras, Rue}.  This theory has been further developed, placing it inside a general theory of `linear response' for dynamical systems, see  example \cite{Rulin} and references therein, as well as comments on the work of Baladi and Smania below.  

\subsection{Physical measures} 
An invariant measure is said to be \emph{physical}\index[def]{Physical measure} if there is a positive (full-dimensional) Lebesgue measure set of points $x \in M$ for which the average of Dirac measures along the orbit of $x$ converges to the invariant measure.  Any ergodic acip, if it exists, is a physical measure.  Alternatively if there is a periodic attractor, i.e. a point $x$ with $f^p(x)=x$ and $|Df^p(x)|<1$, then  the equidistribution on the orbit of $x$ is a physical measure. 
For smooth unimodal maps in general, physical measures can be very strange.  For example, a quadratic map with a physical measure which is the Dirac mass on a repelling fixed point was constructed in \cite{BrKel}, developing a construction in \cite{Joh87}.  
Perhaps even more surprisingly,  topologically transitive, unimodal maps of high critical order can have their physical measure supported on an absorbing Cantor set \cite{BKNS}.   

\subsection{Acips for smooth interval maps}
For example, let $\F_Q$ be the class of quadratic interval maps $f_a : x \mapsto ax(1-x)$, for $0<a \leq 4$, and consider the potential $\phi(x):=-\log|Df|$ for $f\in \F_Q$. By \cite{BloLyu89}, Lebesgue measure is ergodic. Any acip is a physical measure. Moreover it is an equilibrium measure with respect to the potential $\phi$, this following from \cite{Ledrap} and \cite{Prz93}. 
 On the other hand, if the map is hyperbolic, so it has a (hyperbolic) attracting periodic orbit, then the physical measure is the equidistribution on the attracting orbit.  
Hyperbolic parameters form an open, dense subset of $(0,4]$ (\cite{lyu, GraSw}), yet 
the set of parameters for which $f_a$ has an acip, by \cite{Jak, BenCarl}, has positive Lebesgue measure. There are further parameters for which acips do not exist, see for example \cite{HofKel90, AlkBruJak08}. 
This gives an indication of the complexity of the possible behaviour within the quadratic family. 

There are positive results regarding statistical stability in $\F_\mathrm{NSD}$. In \cite{Tsu_cty, Frei, FreT} it was shown that for large  families $\F\in \F_\mathrm{NSD}$ of smooth maps with acips, the acips depend continuously on the map.  In the first two of these papers, the maps had to have some exponential growth along the critical orbit, while in the final one, subexponential growth was sufficient.  However in all cases it was essential that the growth constants which defined $\F$ had to be uniform in the family.  For the purposes of the current article, we should point out that these uniform constants lead to uniform tail estimates (which correspond in our setting to $T^R_\Delta=0$, see \S\ref{sec:SSSS}, and a uniform lower bound of the entropy of the acips for all maps in $\F$).

\subsection{Prior results for fixed maps}
For a positive measure set of quadratic parameters near $a=4$, Pesin and Senti showed existence of equilibrium states and analyticity of the pressure for $t$ in a neighbourhood of $[0,1]$ and the potential $-t\log|Df|$ \cite{PeSe}. In \cite{BTeqnat, IomTeq}, these results were extended to eventually give existence of equilibrium states for $t \in (t^-, t^+)$ and differentiability of the pressure function for transitive (multimodal) maps. Recently, in \cite{PrzRL}, existence of equilibrium states and analyticity of the pressure were proven by Przytycki and Rivera-Letelier. Ideas were further developed in \cite{GelfPrzRams}. The paper \cite{PrzRL} transferred some results of \cite{PrzRLSm, PrzRLRational} for rational maps to the interval setting. In the same paper, strong results concerning alternative definitions of pressure and concerning non-uniform hyperbolicity (Topological Collet-Eckmann condition, etc.) were proven. They also show statistical properties such as decay of correlations and the Central Limit Theorem. Compared with that paper, ours has a very different (and, we feel, more canonical) method of constructing induced maps. We focus more on convergence properties of measures and thermodynamic quantities, while they concentrate on the pressure. Our results on almost upper-semicontinuity of the free energy are surprising. We have slightly weaker (see~\S\ref{sec:nsdweak}) hypotheses too, not requiring critical points to be non-flat, and allowing discontinuities and indifferent points. Of course, our article is geared towards proving results about families of maps; that we obtain both new and (slight generalisations of) recent results for fixed maps is a bonus. 

\subsection{Non-positive Schwarzian derivative and bounded distortion}\label{sec:nsdweak}
For the proof of Theorem~\ref{thm:SSSS}, we need distortion bounds and lower bounds on the derivative, see Lemma~\ref{lem:distn}. Uniformity of the bounds for the sequence of maps is important and follows from non-positive Schwarzian derivative and the uniform extensibility (\emph{aka} Koebe space) for the induced maps. If we assume the conclusions of Lemma~\ref{lem:distn} hold for our sequences, we can prove Theorem~\ref{thm:SSSS}. For further results, we just need some iterate of our induced maps to be expanding---the uniformity of Lemma~\ref{lem:expandingscheme} is unnecessary ($N,K$ may depend on $f$) for our purposes. Expansion follows from uniform bounded distortion for iterates of the  induced maps. For fixed maps, therefore, we could just assume some distortion bounds as done in 
 \cite[Definition~1.10]{PrzRL}; for sequences of maps we would need to assume uniformity in these bounds. For simplicity we assume non-positive Schwarzian derivative. We choose non-positive rather than negative because 
if one considers an infinitely-renormalisable quadratic map, 
the limit maps in renormalisation theory \cite[Theorem~1]{Sullivan} automatically have non-positive Schwarzian derivative, and this will be useful in applications of this work, see \cite{DobMih19}. 

\subsection{Return maps in the Hofbauer extension}
    First return maps in the Hofbauer extension \cite{Hpwise} were introduced in \cite{Bru95} and used more recently to good effect in \cite{BTeqnat, IomTeq}, for example. However, those works just look at the first return to a single interval, or `extensible column' in the extension, which engenders full-branched induced Markov maps. The tail estimates for such maps, however, are quite weak and led to difficulties in applications. 

    An important technical advance in this work is to take return maps to a possibly large, but finite, union of intervals in the extension. Combinatorial estimates then give strong control on the tails of the corresponding induced maps, see Lemma~\ref{lem:inducingcyl}. 

    For convergent sequences of maps, Hofbauer extensions need not converge. This problem was defined away in \cite{FreT} by excluding some possible limit maps, including any map with a preperiodic critical point. We overcome this difficulty by embedding our return maps in an ambient space, where a subsequence of the return maps converges. This limit return map will be an induced map for the limit base map, but may not correspond to a return map in the limit map's Hofbauer extension. 

\subsection{Upper-semicontinuity of entropy}
In piecewise-monotone families, topological entropy is lower-semicontinuous (note the number of branches is fixed). 
Misiurewicz has produced examples which show one cannot do better than  lower-semicontinuous \cite{Mis_jump}, see also \cite[Section 4.5]{AlsedaLlibreMis}. Measures of maximal entropy exist and have the same entropy as the map \cite{HofIntrinsic, Hof81}. Therefore metric entropy is not upper-semicontinuous in general (free energy is neither, then). In \cite{Misiurewicz and shlyachkov}, conditions are given for topological entropy to be continuous at a map (in terms of periodic orbits containing singularities). They consider all convergent sequences; we limit ourselves to those with decreasing critical relations but obtain more limit points, so there is some but not complete overlap. We prove upper semi-continuity of the metric entropy. This additionally implies that topological entropy is continuous at the limit map (for that sequence of maps).

\subsection{Keller's example} \label{sec:Kel}
In
\cite{Kel82}, Keller considers a sequence of continuous expanding, piecewise linear maps of the interval with four branches in the form of a {\tt W}. The two external branches can be assumed to be full. 
The internal branches meet at a tip just above the diagonal, and have slope approximately $2-\eps$. The fixed point near the tip and its symmetric preimage define a restrictive interval on which the map is a tent map with slope $2-\eps$. The acip for the map is supported on the restrictive interval. It has entropy equal to log of the slope of the internal branches, thus bounded away from zero. In the limit (where $\eps\to 0$), the tip meets the diagonal and the restrictive interval vanishes. The limit of the acips is an atom on the tip; meanwhile the limit map has its acip. 
This shows lack of quasistability. It can occur because the limit map has an extra critical relation: the turning point at the tip becomes a fixed point. This example has been further developed in \cite{Concordia, EslMis12}, where all maps are transitive; the limit map still has an extra critical relation. 
On the other hand, the examples of Theorems~\ref{THM:THUN2},~\ref{THM:EXO1BIS}  have decreasing critical relations. 

\subsection{Failure of statistical stability for acips for smooth interval maps} \label{sec:failure} 
Tsujii showed that, in the quadratic family, the Chebyshev parameter $4$ is accumulated by a positive measure set of once-renormalisable Collet-Eckmann parameters whose acips converge, as the parameters converge to $4$, to the Dirac mass on $0$ \cite[Remark~1.2]{Tsu_cty}. 
    These ideas were further developed by Thunberg \cite{Thun}.
    In particular, statistical stability does 
    does not hold for any full measure set of Collet-Eckmann parameters in the quadratic family. This lack of stability was due to renormalisation. 

In Theorem~\ref{THM:THUN2}, we show that statistical stability of acips does not hold even if one restricts to the very well-behaved, non-renormalisable Misiurewicz parameters (or in Corollary~\ref{cor:thun3}, if one restricts to full-measure sets of non-renormalisable Collet-Eckmann parameters). 
    At the boundary of Tsujii's and Thunberg's  renormalisation intervals are parabolic maps. Developing ideas of Homburg and Young \cite{HomYoun}, we find non-renormalisable and almost-parabolic Misiurewicz maps whose measures approximate the equidistribution along the parabolic orbit. 
    Showing convergence of measures is somewhat technical, requiring us to keep track of constants in the proof of existence of the acips, unlike say in the related estimates of Benedicks and Misiurewicz in \cite{BM:flat}. The entropy of the acips along our sequences of maps tends to zero. The proof of the theorem occupies \S\ref{sec:low ent}.
    In Theorem~\ref{THM:EXO1BIS}, we show that a lower bound on entropy does not guarantee statistical stability.

\subsection{Stronger versions of statistical stability}
In many works, see for example \cite{RycSor, Alv04, AlvVia_sta, AlvSou13} the convergence in $L^1(m)$ of the densities $\frac{d\mu_f}{dm}$ of acips $\mu_f$ is called \emph{strong statistical stability}.  The notion of statistical stability of measures we use here is sometimes referred to as weak statistical stability.  One can also ask for even stronger results.
Baladi, in \cite[\S1]{Bal_pwexp}, see also \cite[\S3.2]{Balopen}, asks about the smoothness of measure-parameter dependence for Collet-Eckmann unimodal maps. Our Theorem \ref{THM:THUN2} implies that some sort of uniformity condition on the Collet-Eckmann constants is required for continuity (let alone smoothness). 
Baladi and Smania, in \cite{BalSma08},  showed that the physical measure $\mu_0$ (an acip) for a tent map $f_0$  actually depends differentiably if the family of tent maps $(f_t)_{t\in (-\eps, \eps)}$, with corresponding physical measures $(\mu_t)_{t\in (-\eps, \eps)}$, is chosen to be tangent to the topological class of $f_0$.  On the other hand, for unimodal maps with critical points, the corresponding result in \cite{BalSma09} (see also \cite{Ru_Mis}) requires that the maps stay in a fixed topological class.
That requirement implied (see Appendix A of \cite{BalSma09}) uniformity of the Collet-Eckmann constants, as well as constant topological entropy.  
While this manuscript was being prepared, a paper of Baladi, Benedicks and Schnellmann appeared proving sharp results on statistical stability for a subset of the Collet-Eckmann parameters \cite{BalBenSch13}.  These are stronger conclusions, but for a much less general class of maps than those considered here.

\subsection{Higher dimensions}
For results on statistical stability for non-uniformly hyperbolic maps in higher dimensions, see \cite{Vasq} and \cite{ACF_hen} where partially hyperbolic diffeomorphisms and  H\'enon maps are considered.   See also \cite{Alv04, Aru}.  The results in the current article are restricted to the one-dimensional case, but we expect
that some analogue of our results carries over to the higher-dimensional case, probably requiring \emph{sufficiently large} entropy as in \cite{Buz_ent_exp} --  the main difficulty is then to build families of Markov extensions with the right properties for this setting.

    \section{Method and structure}
    First we give an overview of how almost-semicontinuity of the free energy may hold. The structure of the  article is then presented. 

    \subsection{Method}
    Due to the presence of critical points, it is convenient to use the Hofbauer extension $(\hat I, \hat f)$ of the original system $(I, f)$. 
    In Definition~\ref{def:canonR} we introduce the important \emph{(canonical) level-$R$ induced map} $(\hat X(R), \hat F, \tau)$, where $\hat X(R) \subset \hat I$ is a certain finite union of $R$-cylinders in the Hofbauer extension, $\hat F$ is the (Markovian) first return map to $\hat X(R)$ with return time function $\tau$. Note that $\hat X(R)$ is a finite union of intervals and that branches of the first return map will not be surjective. 

    Positive-entropy measures in $\M_{f}$ \emph{lift} to measures on the extension $(\hat I, \hat f)$. Counting arguments on the extension imply that if $R \geq 8d$ and $d \geq 2$, 
    \begin{equation} \label{eq:epsdef}
        \eps(R) = \eps(R,d) := 8 \frac{ \log R }{R}, \index{epsRd@$\eps(R,d) = 8 \frac{ \log R }{R}$}
    \end{equation}
    and 
    \begin{equation} \label{eq:etadef}
        \eta(R, d)) :=  \frac{ \eps(R)^2 }{2R(\log d)^2}, \index{eztaRD@$\eta(R, d)) =  \frac{ \eps(R)^2 }{2R(\log d)^2}$}
    \end{equation}
    then any measure $\mu \in \M_f$ with entropy at least $2\eps(R)$ must lift to a measure $\hat \mu$ with $$\hat \mu(\hat X(R)) \geq \eta(R,d).$$ 
    The normalised restriction $\hat \nu$ of the measure $\hat \mu$ to $\hat X(R)$ is $\hat F$-invariant. Kac' Lemma implies that  $$\int \tau \,d\hat \nu \leq \frac1{\eta(R,d)}$$
    (see \cite[\S1.5]{Aar97}, for example, for a proof of both of these facts).
    The counting arguments of Lemma~\ref{lem:Rtower} are an important refinement of the ideas in \cite[Theorem~9]{Hpwise}, and better implemented than in \cite[Lemma~4]{BTeqnat}. 

    For convergent sequences of $d$-branched maps, level-$R$ induced maps need not converge to a level-$R$ induced map for the limit system. 
    Embedding the induced maps in some ambient space, some subsequence of the induced maps does converge to a well-defined limit map, with controlled properties. 
    
    The estimates hitherto depend only on $R$ and $d$, so they hold for sequences of maps. 
    We study convergence of induced measures $\hat \nu$ embedded into the ambient space. The limit induced measure will \emph{spread} (Definition~\ref{dfn:spread}) to a light limit measure for the limit map (which may or may not be the one desired). 

    Tail estimates on the first return times then come into play. Our counting arguments imply that the number of branches with return time equal to $n$ grows at most at a small exponential rate, bounded by $\exp(n \eps(R))$.  This implies in some vague sense that the branches with high return time make little contribution to the entropy. 

    Drops in the Lyapunov exponent in the limit occur when diminishing mass for the induced measures has non-diminishing integral of $\log |DF|$, where $F$ represents the induced maps. This bit of integral gets lost in the limit. This only happens if there is also a drop $T_\Delta^R$ in the integral of the return time.  
    
    If $T_\Delta^R =0$ for some $R$, then the Lyapunov exponent is continuous (for the original sequence of measures). This continuity  remains true, merely assuming $\lim_{R \to \infty} T_\Delta^R = 0$. As entropy is upper-semicontinuous (Theorem~\ref{thm:introusc}), for these cases free energy will also be upper-semicontinuous. 
    Otherwise, if the limit is positive, the spread of the limit induced measure will have entropy strictly greater than the limit of the entropies. Either this measure will have free energy greater than or equal to the limit or, to balance out, some bit of free energy gets lost in the limit. In this latter case, some other similarly-constructed measure must have free energy strictly greater than the limit free energy, or one can approximate the lost free energy (a quantity greater than $E_+$) by a periodic orbit, giving a contradiction. Almost upper-semicontinuity of the free energy is obtained. 

    We induce $F$ again to obtain a full-branched Markov map with (exponential) good tail estimates.  This useful result is stated as Theorem~\ref{thm:fullbranch}. Application of thermodynamic formalism and ergodic theory for countable state Markov shifts gives analyticity of the pressure and uniqueness of equilibrium states, under transitivity assumptions. The almost sure invariance principle and decay of correlations follows, since our full-branched induced map can be viewed as a Young tower with exponential tails.

    \subsection{Structure}

In  Section~\ref{sec:cyl}, we detail some properties of $J(f)$. 
We introduce, in Section~\ref{sec:combi}, notions concerning the Hofbauer extension and its topological structure; important counting arguments are carried out in Lemma~\ref{lem:Rtower}. 
Canonical level-$R$ induced maps are introduced in Definition~\ref{def:canonR}.  In the subsequent section, we examine convergence of cylinder sets in Hofbauer extensions. 
    In Chapter~\ref{sec:embed}, we embed the Hofbauer extensions (and thus $\hat X(R)$) into an ambient space, obtaining convergence of the induced maps there to some limit map. 
    Continuing on the topological side, we then show in Chapter~\ref{sec:rarereturns} that given some orbit, there is another orbit which mimics the part of the orbit corresponding to large return times. The error term is uniform for a given sequence of maps.  

    Up until this point, everything was topological. 
    Next, in Chapter~\ref{sec:MM}, we describe the correspondence between measures on the interval and measures for induced maps.   The lift of a positive entropy measure to the Hofbauer extension  cannot have much mass near the boundary of $\hat X(R)$, we show. From this tightness-type result, we obtain limit induced measures in Chapter~\ref{sec:llm}.   Any drop in (metric) entropy is small compared to the  (possible) drop $T_\Delta^R$ in the integral of the inducing time. Upper-semicontinuity of metric entropy is proven. 

We present some standard distortion and expansion estimates for maps with non-positive Schwarzian derivative in Chapter~\ref{sec:NSD}.
All except the final statement of 
Theorem~\ref{thm:SSSS} is proven in Chapter~\ref{sec:SSSS}, with a casewise analysis depending on the behaviour of $T_\Delta^R$. For the final statement, to pass from a convex combination of equilibrium states to some light limit measure actually being an equilibrium state, we need a bound on the number of equilibrium states. 

In Chapter~\ref{chap:Kat}, we present some Katok theory (with self-contained, simple proofs) and, in Section~\ref{sec:ancillary}, several results depending on Theorem~\ref{thm:SSSS} and Katok theory.  
In Section~\ref{sec:tdf2}, facts concerning ergodic theory and thermodynamic formalism for countable state Markov shifts are reviewed. We obtain a full-branched induced map with exponential tails (with respect to the lift of the equilibrium measure). Thanks to the estimates for our first return maps defined earlier,  thermodynamic formalism almost automatically gives analyticity of the pressure and uniqueness of equilibrium states, at least under a transitivity assumption. Without the transitivity assumption, we obtain a bound on the number of equilibrium states with a given entropy and complete the proof of Theorem~\ref{thm:SSSS}. We obtain the almost sure invariance principle and decay of correlations. 

    In Chapter~\ref{sec:low ent} we prove Theorem~\ref{THM:THUN2} which, we recall, shows that for any natural family of unimodal maps and corresponding measures,  uniform positive entropy is necessary for quasistability.  

    Finally, in Chapter~\ref{sec:secexo1}, we show, for a toy model, that uniform positive entropy \emph{does not imply} stability of acips.

    \section{Acknowledgements}  Some of this work came out of discussions between MT and J.M.\ Freitas and M.\ Holland at the University of Porto.  He would like to thank them for their comments and encouragement.  He would also like to thank M.\ Sambarino for his question about convergence of pressure.  We thank V.\ Baladi, F.\ Przytycki and J.\ Rivera-Letelier for several insightful conversations. The referee provided numerous helpful comments and questions.

\chapter{Topological structures}\label{sec:cylchap}

 \section{Piecewise preliminaries} \label{sec:cyl}

\begin{lemma}
 \label{lem:SMB}
 Given a piecewise-monotone map $f$, each  positive-entropy measure $\mu \in \M_f$ must have $\mu(J(f)) = 1$. 
 \end{lemma}
 \begin{proof}      To show this, note first that $\mu$ is ergodic with positive entropy. Hence it must be non-atomic, so any countable set  has measure zero. Moreover, $\mu$ is invariant so almost every point is recurrent. It suffices, therefore, to show that almost every point is accumulated on both sides by points from $\cup_{n\geq 0} f^{-n}(\E)$. 
     Recall that $J(f)$ and $\E$ were defined in \S\ref{sec:EandJ}.

 Consider the set $$E = I \setminus \overline{\bigcup_{n\geq 0} f^{-n}(\E \cup \partial I)}.$$ It is an open set, so its connected components consist of a countable number of intervals. Note that $E$ is forward-invariant, $f(E) \subset E$. Fix some connected component $E_j$ of $E$.  
 If $f^n(E_j) \cap E_j \ne \emptyset$ for some $n >0$, then $f^n(E_j) \subset E_j$. 
 Either some iterate of $f$ is not defined anywhere on $E_j$ or all iterates are defined and homeomorphic on $E_j$. 
 Either way, it follows from ergodicity, invariance and positive entropy that $\mu(E_j) = 0$.   
 The boundary points of the components $E_j$ form a countable set, therefore with measure zero.  Let $E_*$ be the union of the closures of the connected components of $E$. Thus $I\setminus (E_* \cup \partial I) $ has full measure. But this is exactly the set of points accumulated on both sides by points from $\cup_{n\geq 0} f^{-n}(\E)$. 
 We conclude that almost every point is in $J(f)$. 
 \end{proof}

 Uniqueness of equilibrium measures will rely on the following rather general lemma. Its proof is very similar to that of Lemma~5.2 of \cite{Dob08}. Here we assume transitivity; in the setting of \cite{Dob08}, ergodicity was assumed, giving transitivity on a set of full measure. 

    \begin{lemma} \label{lem:transitivity}
        Let $f$ be a piecewise-monotone map, transitive on $J(f)$. Suppose there is an open interval $W$ containing a point from $J(f)$ and $l \geq 1$ such that $W \subset f^l(W)$. Then there exists $N \geq 1$ such that $$J(f) \subset \bigcup_{j=1}^N f^j(W).$$ 
    \end{lemma}
    \begin{proof}
        We can assume, without loss of generality, that $l = 1$, since $f^l$ would also be a piecewise-monotone map. 
Let $$W_j = \bigcup_{i=0}^j f^i(W) = f^j(W).$$ Since $\E = \E(f)$ is finite, there is an $N$ such that for all $j > N$, $W_j \setminus W_N$ does not form a one-sided neighbourhood of any point of $\E$. 

Suppose $j \geq N$ and $V$ is a connected component of $W_j \setminus \E$. Suppose there is some  $k \geq 1$  such that $f^k(V) \cap W_N \ne \emptyset$ and let $l$ be the least such $k$. 
Then $f^l|_{V}$ is continuous. 

For each connected component $V$ of $W_N \setminus \E$, let $n_V$ be the minimal $k\geq 1$ such that $f^k(V) \cap W_N \ne \emptyset$ if such a $k$ exists; otherwise set $n_V = 0$. 
Let $M$ be the maximum of the $n_V$, noting that $W_N$ has a finite number of connected components. 
Let $X$ be a connected component of $W_j$ for some $j\geq M$. Suppose $X$ does not contain a connected component of $W_M$. Let $y \in X$ and let $l \geq 1$ be minimal such that $y \in f^l(V)$, where $V$ is some component of $W_N\setminus \E$. Then $l>M$. In particular, $n_V = 0$, so $f^i(y) \notin W_N$ for all $i\geq 0$. This holds for each $y \in X$. Therefore $f^i(X) \cap W_N = \emptyset$ for all $i \geq 0$. By transitivity, $X \cap J(f) = \emptyset$. 

Let $V_1, \ldots, V_r$ denote the connected components of $W_M$.
For $j\geq M$ and $1 \leq k \leq r$, let $V^j_k$ denote the connected component of $W_j$ containing $V_k$. These are the only connected components of $W_j$ which may contain points from $J(f)$.

Let $W_\infty = \bigcup_{j\geq 0}  W_j$. 
We shall say that a point $x$ \emph{has the one-sided property} if 
 $W_\infty$ contains nested,  one-sided neighbourhoods $Z_k$ of the $x$ such that: each $Z_k$ contains points from $J(f)$,  each $Z_k \not\subset W_j$ for any $j < \infty$, and  $|Z_k| \to 0$ as $k \to \infty$. 
 
 Note that if the conclusion of the lemma does not hold, there exists a point $x$ which has the one-sided property. Suppose this is so. We must arrive at a contradiction. 

 Points with the one-sided property belong to the finite set $$\bigcup_{k=1}^r \partial \bigcup_{j\geq M} V^j_k.$$ 
If $y$ is in the interior of $W_\infty$, then $f^k(y)$ does not have the one-sided property for any $k\geq 0$. 
Suppose $x$ has the one-sided property and let $Z_x$ be a corresponding one-sided neighbourhood of $x$. Then 
there is a (we can assume strictly monotone) sequence $(y_n)_n \subset J(f)$  for which $f(y_n) \in Z_x$ and $f^n(y) \to x$ as $n \to \infty$. Let $y$ be its limit. 
Then $f( (y_n, y) )$ contains $(f(y_n), x) \subset Z_x$. Since $x$ has the one-sided property, $(y_n,y) \not\subset W_j$ for any $j>0$, so $y$ has the one-sided property. Moreover, if $y_n \in U_i$ for all large $n$, then the continuous extension of $f$ to $\overline{U_i}$ maps $y$ to $x$.

It follows that (the finite collection of) points having the one-sided property are periodic, in the sense that if $x,Z_x$ are as before and $Z_x$ is sufficiently small then some iterate $f^k(Z_x)$ is a small one-sided neighbourhood of $x$. 
The point $x$ is a fixed point of the continuous extension $g$ of $f^k_{|Z_x}$ to the closure $\overline{Z_x}$.
If $ g(Z_x) \subset Z_x$ for arbitrarily small  $Z_x$ then $J(f)$ must just be the orbit of $x$, by transitivity, and thus be contained in $f^k(W)$. 
Otherwise,  $x$ is a (one-sided, topologically-) repelling fixed point for $g$.
For points of $J(f)$  to accumulate near $x$ in $Z_x$, they have to come from somewhere: there must be a point $y \ne x$ and arbitrarily small,  one-sided neighbourhoods $Z_y$ of $y$ mapped by $f^k$ onto one-sided neighbourhoods of $x$ in $Z_x$, with $Z_y$ containing points from $J(f)$. 
But $y, Z_y$ are not periodic in the above sense, so $y$ cannot have the one-sided property. In particular, some sufficiently small $J_y$ is contained in some $W_j$. Then $W_{j+k} \supset f^k(J_y)$ contains some $J_x$, contradicting the one-sided property. 
    \end{proof}

    \section{Combinatorics in the Hofbauer extension} \label{sec:combi}
    Let $f : \cup_{j=1}^d I_j \to I$ be a $d$-branched piecewise monotone map. 
    Let us describe the dynamically defined cylinders.  
    Recall that we denote by $\mathring{A}$ the interior of a set $A$\index{Aint@$\mathring{A}$ the interior of $A$}. Let $f_\E$ denote the restriction of $f$ to $\cup_{j=1}^d \mathring{I}_j$. 
    Let $\P_0=\P_0^f:=\mathring{I}$ and $\P_1=\P_1^f:=\{\mathring{I}_1, \ldots, \mathring{I}_d\}$.  For $n \geq 2$, let  
    $$\P_n=\P_n^f := \bigvee_{i=0}^{n-1} f_\E^{-i}(\P_1).
    \index{Pnf@$\P_n^f$, $n$-cylinders} \index[def]{Cylinders} 
    $$
    Each  $\cyl \in \P_n$ is \emph{an $n$-cylinder}: 
    $f^n:\cyl\to f^n(\cyl)$ is a well-defined homeomorphism, $f^j(\cyl) \cap \E = \emptyset$ for $j = 0,\ldots, n-1$ and $\cyl$ is a maximal interval with these properties. Cylinder sets are open. If $n \geq 1$, the boundary points of $\cyl$ are in the set $\cup_{j=0}^{n-1} f^{-j}(\E)$.
We let $\cyl_n[x]$ \index{cyln@ $\cyl_n[x]$, the element of $\P_n^f$ containing $x$} denote the member of $\P_n$ containing $x$.  

    \begin{lemma}
        \label{lem:cylshrinks}
        For $x \in J(f)$, $\cyl_n[x]$ is defined for all $n \geq 0$ and contains $x$ in its interior and $\cyl_n[x]$ shrinks to the point $x$ as $n \to \infty$.  
    \end{lemma}
    \begin{proof} This follows immediately from the definition of $J(f)$.
    \end{proof}

    One can code the cylinder sets as follows. For $n \geq 1$, let 
    $$
    \Sigma^n_d := \{1, 2, \ldots, d\}^n.
    \index{Sigman@$ \Sigma^n_d$, coding space for $n$-cylinders}$$
    Any $\omega \in \Sigma^n_d$ is of the form $\omega = (\omega_1 \omega_2 \ldots \omega_n)$. 
    Given such an $\omega$, there exists at most one $n$-cylinder, denoted $\cyl^\omega$, such that $f^{l-1}(\cyl^\omega) \subset I_{\omega_l}$ for $l = 1, \ldots n$. Of course, each cylinder is of this form for some $\omega$.

    Set 
    $$
    \D := \{f^k(\cyl) : k\geq 0, \cyl \in \P_{k}\}.
    \index{Domains@$\D$, domains of the Hofbauer extension}$$
    As $\D$ is a set, each element $D \in \D$ appears once. 
    As in \cite{Hpwise}, the \emph{Hofbauer extension}\footnote{This is the Hofbauer extension of $f$ restricted to $\cup_{j=1}^d \mathring{I_j}$. Restricting to open cylinder sets, as we do here, has advantages when considering convergence for families of maps, at the cost of a countable set of points.} is defined as 
$$\hat I:= \bigsqcup_{D\in \D} D.\index[def]{Hofbauer extension} \index{I@$\hat I$, the Hofbauer extension} $$ 
with the disjoint union topology on $\hat I$. Note that each $D \in \D$ is an open subinterval of $\R$. 
We call each $D$ a \emph{domain} of $\hat I$. 
There is a natural projection map 
$\pi:\hat
I \to I$.\index{pi@$\pi$ the projection map from $\hat I$ to $I$}  A point $\hat x\in \hat I$\index{xhat@$\hat x=(x, D)$, a point in the Hofbauer extension} is in some $D \in \D$;  $\hat x$ can
be represented by $(x,D)$ where 
$x=\pi(\hat x)$.  Given $\hat x\in \hat I$, we can denote the domain $D\in \D$ it belongs to by $D_{\hat x}$.

The map 
$$\hat f: \pi^{-1}\left( \bigcup_{j=1}^d I_j \right ) \to \hat I$$\index{fhat@$\hat f$, the dynamics on $\hat I$}
is defined by
$$\hat f(\hat x) = \hat f(x,D) = (f(x), D')$$
if there are cylinder sets $\cyl' \subset \cyl$, with $\cyl'\in \P_{k+1}$ and $\cyl \in \P_{k}$, \st $$x \in
f^k(\cyl') \subset f^k(\cyl) = D$$ and $D' = f^{k+1}
(\cyl')$.
Equivalently, there are $k, y$ such that $f^k(y) = x$,  $f^k(\cyl_k[y]) = D$ and $f^{k+1}(\cyl_{k+1}[y]) = D'$. 
In this case, we write $D \to D'$, giving $(\D, \to)$ the
structure of a directed graph. A path from $D_1$ to $D_2$ in $(\D,\to)$ is a sequence of arrows in the graph leading from $D_1$ to $D_2$.   The map $\pi$
acts as a semiconjugacy between $\hat f$ and $f$: $$\pi\circ \hat
f=f\circ \pi.$$  
The Hofbauer extension has the useful property of being Markovian: 
\begin{lemma} \label{lem:MarkovH}
    Let $\hat x \in D \in \D$, let $n \geq 1$ and suppose $D' \in \D$ contains $\hat f^n(\hat x)$. Then some neighbourhood $W \subset D$ of $\hat x$ is mapped homeomorphically by $\hat f^n$ onto $D'$. 
\end{lemma}
    \begin{proof} This follows by induction from the definition of $\hat f$. 
    \end{proof}

\begin{lemma} \label{lem:meetup}
 If $\hat x_1, \hat x_2\in \hat I$ have $\pi(\hat x_1)=\pi(\hat x_2)\in J(f)$ then there exists $n\in \N$ such that $\hat f^n(x_1)=\hat f^n(x_2)$.
\end{lemma}

\begin{proof}
    Suppose that $x=\pi(\hat x_1)=\pi(\hat x_2)$ and let $D_1$ and $D_2$ denote the domains in the Hofbauer extension containing $\hat x_1$ and $\hat x_2$ respectively. By Lemma~\ref{lem:cylshrinks}, there exists an $n \geq 1$ such that   $\cyl_n[x] \subset \pi (D_1 \cap D_2)$. 
The construction of the extension implies that $\hat f^n(\hat x_1),\hat f^n(\hat x_2)\in f^n(\cyl_n[x])\in \D$ and hence that $\hat f^n(\hat x_1)=\hat f^n(\hat x_2)$.  
\end{proof}

        \begin{lemma} \label{lem:path12}
            Suppose $f$ is transitive on $J(f)$. 
            Let $D_1, D_2 \in \D$ be domains for which the first return maps  each have at least two branches. Then there is a path from $D_1$ to $D_2$. 
        \end{lemma}
        \begin{proof}
            The first return map to any domain in $\D$ is full-branched, by Lemma~\ref{lem:MarkovH}, so $D_1$ and $D_2$ each contain infinitely many repelling periodic points. Meanwhile, there are only finitely many periodic points in the post-critical orbits, that is, contained in the set
            $\cup_{n\geq 0} f^n(\V)$. Therefore there are infinitely many repelling periodic points in each of $D_1$ and $D_2$ which project down to points in $J(f)$. 
            Let $\hat p_1 \in D_1$ and $\hat p_2 \in D_2$ be such points, with projections $p_1$ and $p_2$ respectively. 

            Let $W_1 \subset \pi D_1$ be an open interval containing $p_1$, small enough that $W_1 \subset f^l(W_1)$. By Lemma~\ref{lem:transitivity}, there is some $N$ for which $\cup_{j=0}^n f^j(W_1)$ contains $p_2$. Thus some $\hat y \in D_1$ is mapped by some $\hat f^j$ to $\pi^{-1} (p_2)$.    By Lemma~\ref{lem:meetup}, some further iterate of $\hat y$ by $\hat f^j$ is actually equal to $\hat p_2$.  
        \end{proof}

        We denote the \emph{base} of $\hat I$, the copy of $\mathring{I}$ in $\hat I$, by  $D_0$.\index{D0@$D_0$, the base of $\hat I$} The natural inclusion map sending $\mathring{I}$ to $D_0$ we denote by $\iota : \mathring{I} \hookrightarrow \hat I$.\index{izota@$\iota$ the inclusion map of $\mathring{I}$ into $D_0$}  For $D\in \D$, we define $\level(D)$ \index{level@$\level(D)$} to be the length of the shortest path $D_0 \to \dots \to D$ starting at the base $D_0$.  
        \begin{definition}          For each $R \in \N$, let $\hat I^R$ be the bounded
part of the Hofbauer extension defined by
$$
\hat I^R := \{ (x,D) \in \hat I : \level(D) \le R \}.\index{IR@$\hat I^R$, bounded part of $\hat I$}
$$
\end{definition}
Abusing notation somewhat, we can write $\hat I^R = \sqcup \{ D \in \D : \level(D) \le R \}.$

    We define $\hat \P_n$ \index{Pn@$\hat \P_n$, induced partition on $\hat I$} to be the set of  intervals $\{\pi|_D^{-1}\cyl:D\in \D, \cyl\in \P_n\}$. 
    For a domain $D\in \D$, let $D_\ell^R$ be the left-most element of $\hat \P_R$ in $D$ and let $D_r^R$ be the rightmost. We refer to these elements \emph{boundary components}. 
    Set $$\hat I_-^R:= \hat I^R \cap \left(\sqcup_{D\in \D}D\sm (D_\ell^R\cup D_r^R)\right).
    \index{Iminus@$\hat I_-^R$, trimmed version of $\hat I^R$} 
    $$
    As we will see later, excluding the boundary components will give uniform distortion bounds (thanks to the Koebe principle) for certain induced maps. Boundary components project to subsets of $R$-cylinders.  
   
    For the purposes of the following lemma, let us set $\P_n^0 := \iota^{-1} \P_n$, the same collection of $n$-cylinders except living on the base of the extension, and for a word $\omega$, we set $\cyl_0^\omega := \iota^{-1}(\cyl^\omega)$.

    Recall from~\eqref{eq:epsdef} that $\eps(R) = 8\log R / R.$
\begin{lemma} \label{lem:Rtower} 
    Let $R \geq 8d$.
    Let $$\delta := \frac{\eps(R) }{ 4 \log d} = \frac{2\log R}{R \log d}.$$  
    There exists $n_0 = n_0(R)$  such that for any piecewise-monotone $d$-branched map $f$ and any $n \geq n_0$, the following holds. 
    Let $\Q$\index{Q@$\Q$, elements of $\P_{nR}^0$ with rare visits to $\hat I^R_-$ } denote the collection of elements of $\P_{nR}^0$ which visit $\hat I^R_-$ at most $\delta n$ times in the first $nR$ iterates.
    Then
    $$
    \#\Q \leq R^{7n} < \exp(nR \eps(R)).
    $$
    Moreover, for $n \geq n_0$,\index{n0@$n_0=n_0(R)$, used to estimate $\#\Q$} the number of elements of $\hat P_{R+n}$ contained in $\hat I^R_-$ with first return time to $\hat I^R_-$ equal to $n$ is bounded by $\exp(n \eps(R))$.
\end{lemma}
\begin{proof}  
    From the definition of $\eps(R)$, $R^{7n} < \exp(nR \eps(R))$.

    By construction of the extension, if $\omega$ has length $n$,  $\hat f^j$ is homeomorphic on $\cyl_0^\omega$ for each $j \leq n$ (in particular, for each $j$, $\hat f^j(X)$ is contained in a single domain of the extension). 

 We shall count elements of $\P_{nR}^0$  for large $n$. 
 The following estimates are needed.
 \begin{enumerate}[label=(\roman*)]
        \item \label{en:high}
            if $X \in \P^0_l$ for some $l \geq 1$ then, for $1 \leq k \leq R$, there are at most two elements $X'$ of $\P^0_{l+k}$ for which both $X' \subset X$ and $\hat f^{l+k}(X') \subset \hat I \setminus \hat I^R$; 
        \item \label{en:external}
            there are at most $(2dR)^2$ domains $D$ with $\level(D) \leq R$; in particular, there are at most $8d^2R^2$ connected components of $\hat I^R \setminus \hat I^R_-$;
        \item \label{en:combi}
            $\binom{n}{\lfloor \delta n \rfloor} \leq 2^n.$
          
    \end{enumerate}
    Estimate \ref{en:high} is so because if $X'$ does not have a common boundary point with $X$, then $\pi \hat f^l(X') \in \P_k$, so $\hat f^{l+k}(X') \subset \hat I^k \subset \hat I^R$. 
    For estimate \ref{en:external},
    the projection of each boundary point of a domain $D$ of level at most $R$ is contained in the set $\cup_{0 \leq j \leq R-1} f^j(\V)$ of cardinality bounded by $2dR$, since the set $\V$ of critical values has cardinality at most $2d$.  A domain has two boundary points, giving $(2dR)^2$ possibilities. 
    The trivial estimate \ref{en:combi} on the binomial coefficient could be improved (for example using Stirling's formula), but is sufficient for our needs. 

 To each element $\cyl = \cyl^\omega_0$ corresponds a word $\omega$ of length $nR$. 
 For $j \geq 0$,
 denote by $\cyl^j$ the image $\hat f^j(\cyl)$, so $\pi \cyl^j$ is contained in the $1$-cylinder of $f$ corresponding to the symbol $\omega_{j+1}$.  
    We can divide the word $\omega$ into $n$ blocks of length $R$. 
    If $\cyl^{kR +j} \subset \hat I \setminus \hat I^R_-$ for $j = 0, \ldots, R-1$ then (counting from zero) we say the $k^\mathrm{th}$ block is \emph{external}.
    If it is external, there is a maximal $l \leq R$ for which $\cyl^{kR +j} \not \subset \hat I^R$ for $j = 0, \ldots, l -1$ (so if $\cyl^{kR} \not\subset \hat I^R$ then $l = 0$, and if $l < R$ then $\cyl^{kR + l} \subset \hat I^R \setminus \hat I^R_-$). 
    Given $\omega_1 \ldots \omega_{kR }$, by \ref{en:high} there are at most two possibilities for the sequence $\omega_{kR+1} \ldots \omega_{kR + l}$. 
    Given $\omega_1 \ldots \omega_{kR+l}$, there are by \ref{en:external} at most $8d^2R^2$ possibilities for $\omega_{kR+l +1}\ldots \omega_{(k+1)R}$ (one possibility for each boundary component). There are $R$ possible values for $l$. 

    Multiplying these estimates, given $\omega_1 \ldots \omega_{kR}$, there are at most $16 d^2 R^3$ possible words $\omega_{kR+1} \ldots \omega_{(k+1)R}$ for which the $k^{\mathrm{th}}$ block is external. For general blocks of length $R$ 
    we use a crude estimate:
    there are at most $d^R$ possible words. There are fewer than $2^n$ possible configurations of $\lfloor \delta n \rfloor$ general and $n - \lfloor \delta n \rfloor$ external blocks,  by \ref{en:combi}.

    Elements of $\Q$ visit $\hat I^R_-$ at most $\delta n$ times in the first $nR$ iterates.
    A word corresponding to an element of $\Q$ has at least $n - \lfloor \delta n \rfloor$ external blocks of length $R$. 
    Combining the estimates from the preceding paragraph, we get an upper bound of 
    \begin{equation} \label{eq:ent1}
     \# \Q \leq 2^n (16 d^2R^3)^n d^{R \delta n}.
 \end{equation}
 Replacing $\delta$ by $\frac{2\log R}{R\log d}$, we obtain
 $$
 \# \Q \leq (32 d^2R^3)^n R^{2n}.
     $$
     Since $R \geq 8d$,
 $$
 \# \Q \leq  R^{7n}.
     $$
     The first statement is now proven. 

       To finish, we claim that the number of elements of $\hat \P_{R+n}$ contained in $\hat I^R_-$ with first return time to $\hat I^R_-$ equal to $n$ is bounded by 
    $$
     (2dR)^2(16 d^2 R^3)^{\lfloor n/R \rfloor} d^{1+2R}.$$
    We assembled this bound from the following components. An element of $\hat \P_{R+n}$ is defined by a word of length $n+R$ plus the starting domain.
    There are, as before, at most $(2dR)^2$ domains of level at most $R$. 
    There are $d$ possible choices for the first symbol of the word. Then there follows 
    $\lfloor n/R \rfloor$ full  external blocks, each having  $16d^2R^3$ possibilities,   followed by a final $\leq 2R$ free symbols with at most $d^{2R}$ possibilities. 
    
    Taking $n$ large, we can ignore the terms not being raised to the $n$, and recalling $R \geq 8d$ we obtain a bound of
    $$ (R^5)^{n/R} = \exp( 5 n(\log R)/R) < \exp(n\eps(R)),
    $$
    completing the proof. 
\end{proof}

\begin{lemma} \label{lem:YY}
    If $\hat Y, \hat Y' \in \hat \P_n$ and $\hat f(\hat Y) \cap \hat Y' \ne \emptyset$, then $\hat f^{-1}(\hat Y') \cap \hat Y$ is mapped homeomorphically by $\hat f$ onto $\hat Y'$. 
    \end{lemma}
    \begin{proof} This follows from the cylinder structure and Lemma~\ref{lem:MarkovH}. \end{proof}

        \begin{definition}\label{def:canonR}
            The \emph{(canonical) level-$R$ induced map} is the system $(\hat X(R), \hat F, \tau)$ \index[def]{Induced map} and consists of:
            $\hat X = \hat X(R)$, the union of elements of $\hat \P_R$ contained in $\hat I^R_-$; the first return map $\hat F$ to $\hat X$; the first return time function $\tau$ on $\hat X(R)$. 
        \end{definition}

        The large range of $\hat F$, with multiple connected components,  allows one to obtain powerful counting  estimates. 
\begin{lemma} \label{lem:inducingcyl}
    The  level-$R$ induced map $(\hat X(R), \hat F, \tau)$  has the following properties. 
    \begin{enumerate}[label=(\roman*)]
        \item \label{enum:c1}
                $\hat F$ is defined on a countable union of pairwise-disjoint open subintervals $\hat X_i$ of $\hat X$;
            \item
                $\hat F$ maps each $\hat X_i$ homeomorphically onto a connected component of $\hat X$;
                \item
                    the return time $\tau$ is constant on each $\hat X_i$, thus we can denote it $\tau_i$;
                \item \label{enum:c5}
                    for each $i$,  there is some $\hat X_i' \supset \hat X_i$ mapped homeomorphically by $\hat f^{\tau_i}$  onto the domain of $\hat I$ containing $\hat F( \hat X_i) $. 
            \end{enumerate}
            If  $n_0=n_0(R)$ is given by Lemma~\ref{lem:Rtower}, then $$\# \{i : \tau_i =n \} \leq \exp({n\eps(R)}) $$ for all $n \geq n_0$. 
        \end{lemma}
        \begin{proof}
            Properties \ref{enum:c1}-\ref{enum:c5} follow from Lemma~\ref{lem:YY} and Lemma~\ref{lem:MarkovH}. 
            The counting statement follows immediately from Lemma~\ref{lem:Rtower}, since each $\hat X_i$ is an $(R+\tau_i)$-cylinder. 
        \end{proof}
        Let $(\hat X, \hat F, \tau)$ be a level-$R$ induced map.\index{XFT@$(\hat X, \hat F, \tau)$, level-$R$ induced map}
        \begin{lemma} \label{lem:tails}
            Let $N\geq 1$ and let $\hat Y_j$,\index{Yj@$\hat Y_j$, $j\geq 1$, domains of $\hat F^N$} $j\geq 1$, be the connected components of the domain of $\hat F^N$. 
            Let $n_0= n_0(R)$  be given by Lemma~\ref{lem:inducingcyl}. 
            Let $\rho_j$ \index{rho@$\rho_j$ an $N^\mathrm{th}$ return time} denote the value on $\hat Y_j$ of 
            the $N^\mathrm{th}$ return time to $\hat X$ for $\hat f$, so $\hat f^{\rho_j} = \hat F^N$ on $\hat Y_j$. 
            Then there is an $n_1$ (dependent on $N, d,  R, \eps(R), n_0$ but not on $f$) for which, for all $n \geq n_1$, 
            $$\# \{j : \rho_j =n \} \leq e^{3n\eps(R)}.$$
        \end{lemma}
        \begin{proof}
            Let us write $\eps$ for $\eps(R)$. 
            With the notation of Lemma~\ref{lem:inducingcyl}, set $K_k := \# \{i : \tau_i =k \}$ 
            for each $k \geq n_0$, 
            so $K_k \leq e^{k\eps}$. 
            There is a uniform bound $K_*$ (depending just on $d, R, n_0$) on $K_k$ for $k \leq n_0$. 
            Then $$\# \{j : \rho_j =n \} \leq \sum \prod_{i=1}^N K_{k_i}$$
            where the sum is taken over all $N$-tuples $k_1, \ldots, k_N$ which sum to $n$. The product is then bounded by $\exp(2n\eps)$, for $n\geq \log K_* /\eps$.
            There are fewer than $n^N$ summands, so if $n$ is big enough that $n\eps \geq N\log n$, then 
            $$\# \{j : \rho_j =n \} \leq n^N \exp(2n\eps) \leq \exp(3n\eps).$$
        \end{proof}

        \begin{lemma} \label{lem:boundarytime}
            Let $\rho$ be the inducing time (the $R^\mathrm{th}$ return time) corresponding to $\hat F^R$. Let $K \geq 1$ and suppose $\rho(\hat x) \leq K$ for some $\hat x \in \hat X$. Let $\hat y$ be in the boundary of the connected component of $\hat X$ containing $\hat x$. Then $(\hat x, \hat y)$ contains a $(K+R)$-cylinder for $\hat f$. 
        \end{lemma}
        \begin{proof}
            Let $k \geq R$ be minimal such that $\hat f^k(\hat x) \in \hat X$. Then $\rho(\hat x) \geq k$. If $\hat f^k \in D \in \D$, we can pull back the boundary cylinders $D^R_l, D^R_r$ to obtain $(k+R)$-cylinders separating $x$ from $\partial \hat X$.  
        \end{proof}
        
        The map $\hat F$ is typically not onto. Since $\hat X$ has only a finite number of connected components and $\hat F$ is Markovian, we can find subsets with good mixing properties. Let us say that connected components $\hat Y, \hat Y'$ are \emph{linked} if some iterate of $\hat F$ maps  $\hat Y$ onto $\hat Y'$ and some iterate also maps $\hat Y'$ onto $\hat Y$. Being linked is clearly a symmetric and transitive relation. 

        We denote by $\hat X_\T$ the collection of all connected components $\hat Y$ of $\hat X$ for which the first return map from $\hat Y$ to $\hat Y$ has at least two branches. Given $\hat Y \in \hat X_\T$, we can write $\Y_{\hat Y}$ for the union of all components linked to $\hat Y$.  
        \begin{definition}\label{def:primi}
        We call $\Y_{\hat Y}$ the \emph{primitive component} \index[def]{Primitive component} of $\hat X$ containing $\hat Y$. \index{YY@ $\Y_{\hat Y}$ the primitive component of $\hat X$ containing $\hat Y$}
    \end{definition}
        
    \begin{remark}\label{rem:prim}
        The number of connected components of $\hat X$ is bounded by $(2dR)^2d^R$ and this is trivially an upper bound on the number of primitive components of $\hat X$.  
    \end{remark}
        \begin{lemma} \label{lem:primi3}
            If $f$ is transitive on $J(f)$ then $\hat X$ has only one primitive component. 
        \end{lemma}
        \begin{proof}
            Let $\hat Y_1, \hat Y_2 \in \hat X_\T$, so their first return maps each have at least two branches. The  return maps are full-branched, by Lemma~\ref{lem:inducingcyl}, so $\hat Y_1$ and $\hat Y_2$ each contain infinitely many repelling periodic points. Meanwhile, there are only finitely many periodic points in the post-critical orbits, that is, contained in the set
            $\cup_{n\geq 0} f^n(\V)$. Therefore there are infinitely many repelling periodic points in each of $\hat Y_1$ and $\hat Y_2$ which project down to points in $J(f)$. 
            Let $\hat p_1 \in \hat Y_1$ and $\hat p_2 \in \hat Y_2$ be such points, with projections $p_1$ and $p_2$ respectively. 

            Let $W_1 \subset \pi \hat Y_1$ be an open interval containing $p_1$, small enough that $W_1 \subset f^l(W_1)$. By Lemma~\ref{lem:transitivity}, there is some $N$ for which $\cup_{j=0}^N f^j(W_1)$ contains $p_2$. Thus some $\hat y \in \hat Y_1$ is mapped by some $\hat f^j$ to $\pi^{-1} (p_2)$.    By Lemma~\ref{lem:meetup}, some further iterate of $\hat y$ by $\hat f^j$ is actually equal to $\hat p_2$. Repeating the argument switching twos and ones, we find a point in $\hat Y_2$ mapped to $\hat Y_1$. It follows that $\hat Y_1$ and $\hat Y_2$ are linked, so they are subsets of the same primitive component, as required.   
        \end{proof}

        \begin{lemma} \label{lem:primi4}
            Let $\Y$ be a primitive component of $\hat X$. Then the first return (under $\hat f$) map 
            $\hat F_\Y$ to $\Y$
            has the following properties:
            \begin{enumerate}[label=(\roman*)]
        \item \label{enum:cc1}
                $\hat F_\Y$ is defined on a countable union of pairwise-disjoint open subintervals $\hat X_i$ of $\Y$;
            \item \label{enum:cc2}
                $\hat F_\Y$ maps each $\hat X_i$ homeomorphically onto a connected component of $\Y$;
            \item \label{enum:cc3}
                    there exists $N_\Y$ such that, for each $\hat X_i$,
                    $$\bigcup_{j=1}^{N_\Y} \hat F_\Y^j(\hat X_i) \supset \Y;$$
                \item \label{enum:cc4}
                    the return time is constant on each $\hat X_i$ and coincides with the return time under $\hat F$ to $\hat X$: thus we can denote it $\tau_i$;
                \item \label{enum:cc5}
                    for each $i$,  there is some $\hat X_i' \supset \hat X_i$ mapped homeomorphically by $\hat f^{\tau_i}$  onto the domain of $\hat I$ containing $\hat F_\Y(X_i)$. 
            \end{enumerate}
            If $n_0=n_0(R)$ is given by Lemma~\ref{lem:Rtower}, then $$\# \{i : \tau_i =n \} \leq \exp({n\eps}) $$
            for all $n \geq n_0$. 
        \end{lemma}
        \begin{proof}
            We must prove~\ref{enum:cc3} and~\ref{enum:cc4}. The other three properties follow from Lemma~\ref{lem:inducingcyl}. 
           
            To show~\ref{enum:cc3}, note first that $\hat F_\Y(\hat X_i)$ is a connected component of $\Y$.  Each pair  $\hat Y_1, \hat Y_2 \subset \Y$ of connected components of $\Y$ is linked, so for some minimal $n = n(\hat Y_1, \hat Y_2)$, 
            $$\hat F_\Y^n(\hat Y_1) \supset \hat Y_2.$$ 
            There are only a finite number of such ordered pairs, so 
            $$N_\Y := 1 + \max_{\hat Y_1, \hat Y_2 \in \Y} n(\hat Y_1, \hat Y_2)$$ 
            is finite and has the desired property.   

            Now we show~\ref{enum:cc4}. Suppose that $\hat x$ is in the domain of $\hat F_\Y$, so
            $\hat x \in \hat Y_1$ and $\hat F_\Y(\hat x) \in \hat Y_2$ for some connected components $\hat Y_1, \hat Y_2$ of $\Y$. 
            Let $\hat Y'$ denote the connected component of $\hat X$ containing $\hat F(\hat x)$. In particular, $\hat F(\hat Y_1) \supset \hat Y'$. Some $\hat F^j$, $j\geq0$, maps $\hat Y'$ onto $\hat Y_2$.
            Since $\hat Y_2, \hat Y_1$ are linked, some iterate of $\hat F$ maps $\hat Y_2$ onto $\hat Y_1$. 
            Consequently $\hat Y_1$ and $\hat Y'$ are linked. Hence  $\hat Y' = \hat Y_2 \in \Y$ and 
            $$\hat F(\hat x) = \hat F_\Y(\hat x) = \hat f^{\tau_i}(\hat x),$$ if $\hat F = \hat f^{\tau_i}$ on the interval $\hat X_i \ni \hat x$. 
        \end{proof}
        \begin{definition} \label{def:primcomp}
        A \emph{level-$R$ primitive induced map} is such a triplet
        $(\Y, \hat F_\Y, \tau)$, \index{YFT@ $(\Y, \hat F_\Y, \tau)$, a level-$R$ primitive induced map}  where $\Y$ is a primitive component of a level-$R$ induced map, $\hat F_\Y$ is the first return map to $\Y$ and $\tau$ is the return time (w.r.t.\ $\hat f$). 
        \end{definition}

                \section{Coding convergence}
               Recall that the coding of cylinder sets $\cyl^\omega$, for $\omega \in \Sigma^n_d$, was defined near the start of~\S\ref{sec:combi}.
    Introducing $f$-dependence, we can write $\cyl^\omega(f)$ for $\cyl^\omega$. Let $\Sigma^n(f)$\index{Sigman@ $\Sigma^n(f)$, the set of codes for $\cyl^\omega(f)$} denote the set of $\omega \in \Sigma^n_d$ for which $\cyl^\omega(f)$ exists.  
    Sequences of maps $(f_k)_k$ in what follows will be sequences of $d$-branched piecewise-monotone maps. 
    \begin{lemma} \label{lem:cylconv}
        Let $(f_k)_k$ be a sequence converging to $f_0$ as $k \to \infty$ and having decreasing critical relations. Given $n \geq 1$, for all sufficiently large $k$, $\Sigma^n(f_0) = \Sigma^n(f_k)$. For $\omega \in \Sigma^n(f_0)$, 
    $$
        \cyl^\omega(f_k) \to \cyl^\omega(f_0)
    $$
        in the Hausdorff metric as $k \to \infty$. 
    \end{lemma}
    \begin{proof}
        Since the $f_k$ converge to $f_0$ and cylinder sets are open intervals, $\Sigma^n(f_0) \subset \Sigma^n(f_k)$ for large $k$ and the corresponding cylinder sets converge. To complete the proof, we prove by contradiction that $\Sigma^n(f_k) \subset \Sigma^n(f_0)$ for all large $k$. 
        
        The case $n=1$ is trivial, by definition of convergence of the $f_k$. 
        So suppose $n \geq 2$ is minimal such that for some subsequence $(k_i)_i$, tending to infinity with $i$,  there is some $\omega$ with 
    $$
        \omega \in \Sigma^n(f_{k_i}) \setminus \Sigma^n(f_{0})
    $$
    for all $i$. To simplify notation, we can suppose that this subsequence is the original sequence $(f_k)_k$.      
    Denote by $\omega'$ the word $\omega$ with its first symbol removed. Then $\omega'$ is a word of length $(n-1)$ and  $f_k(\cyl^\omega(f_k)) \subset \cyl^{\omega'}(f_k)$. Since $n$ is minimal, $\omega' \in \Sigma^{n-1}(f_0)$. 
    It follows that $f_0(I_{\omega_1}(f_0))$ shares a common boundary point, $x$ say, with $\cyl^{\omega'}(f_0)$, and that 
    $$f_0(I_{\omega_1}(f_0)) \cap \cyl^{\omega'}(f_0) = \emptyset.$$ 
    Let $c_0$ be the boundary point of $I_{\omega_1}(f_0)$ whose small one-sided neighbourhoods in $I_{\omega_1}(f_0)$ get mapped to a one-sided neighbourhood of $x$.  
    Let $p \geq 0$ be minimal such that $f_0^p(x) \in \E(f_0)$ and note that $f^p_0$ is a homeomorphism on some neighbourhood of $x$. Therefore
    on any small neighbourhood $W$ of $x$, 
    $f^p_k$ is a homeomorphism 
    for all $k$ large enough; $f^j_k(W) \cap \E(f_k) = \emptyset$ for $j=0,\ldots,p-1$; the corresponding boundary points $x_k \in \partial \cyl^{\omega'}$ to $x$ lie in $W$.  
    There is a critical relation of order $(p+1)$ between $c_0$ and $f_0^p(x)$ for $f_0$. By the decreasing critical relations hypothesis, this  critical relation also exists for $f_k$. In particular, $\{x_k\} = W \cap \partial f_k(I_{\omega_1}(f_k))$, for $W$ small and $k$ sufficiently large. But this implies that 
    $$f_k(I_{\omega_1}(f_k)) \cap \cyl^{\omega'}(f_k) = \emptyset $$ so $\cyl^\omega(f_k) = \emptyset$, a contradiction.
 \end{proof}

 \begin{lemma} \label{lem:bigdomains}
     Let $(f_k)_k$ be a sequence converging to $f_0$ as $k \to \infty$ and having decreasing critical relations. For each $j$ and $n$ with $0\leq j \leq n$, for each $\omega \in \Sigma^n(f_0)$ and for all $k$ large enough, the sets
     $$
     f_k^j(\cyl^\omega(f_k))
     $$
     exist and they converge in the Hausdorff metric to (the non-trivial set) $f_0^j(\cyl^\omega(f_0))$ as $k \to \infty$. 
 \end{lemma}
 \begin{proof}
     This follows from convergence of the maps and Lemma~\ref{lem:cylconv}.
 \end{proof}
	\begin{lemma}\label{lem:overlaps}
     Let $(f_k)_k$ be a sequence converging to $f_0$ as $k \to \infty$ and having decreasing critical relations. For each $j$ and $n$ with $1\leq j, n$, there is a $\kappa > 0$ such that the following holds. For each $\omega \in \Sigma^j(f_0)$ and $\theta \in \Sigma^n(f_0)$, for all $k$ large enough, either $\cyl^\omega(f_k) \cap f_k^n(\cyl^\theta(f_k)) = \emptyset$ or 
	$$\left|\cyl^\omega(f_k) \cap f_k^n(\cyl^\theta(f_k))\right| \geq \kappa.$$
	\end{lemma}
 \begin{proof}
     For the purposes of the proof, we may assume that $\cyl^\omega(f_k) \cap f^n_k(\cyl^\theta(f_k)$ is non-empty. Then the set $$W_k = \cyl^\theta(f_k) \cap f^{-n}_k(\cyl^\omega(f_k))$$ 
	is non-empty and is an $(n+j)$-cylinder. 
        By Lemma~\ref{lem:cylconv}, $W_k$ converges to an $(n+j)$ cylinder $W_0$ for $f_0$. From Lemma~\ref{lem:bigdomains}, $f^n_k(W_k)$ converges to $f_0^n(W_0)$. Taking $\kappa = | f_0^n(W_0)| /2$ completes the proof. 
 \end{proof}

    \textbf{Notation.} 
       Recall that for a map $f$, $\hat X = \hat X(R)$ is the union of elements of $\hat P_R$ contained in $\hat I^R_-$. The first return map to $\hat X$ we denote $\hat F$; its return time we denote $\tau$.
    Given maps $f_k$ and a number $R\geq 1$, we introduce dependence on $k$ to the notation $\P$, $\hat P_n$, $\hat I^R$, $\hat X(R)$, $\hat X$, $\hat F$, $\tau$, giving the notation 
    $\P^k$, $\hat P^k_n$, $\hat I^R(k)$, $\hat X^k(R)$, $\hat X^k$, $\hat F_k$, $\tau_k$. 
    \begin{lemma} \label{lem:kappaPk}
     Let $(f_k)_k$ be a sequence converging to $f_0$ as $k \to \infty$ and having decreasing critical relations. For each $R, K\geq 1$, there is a $\kappa > 0$ such that the following holds. 
     $$
     \inf_k \{|V| : V \in \hat \P_K^k, V \subset \hat I^R(k) \} \geq \kappa >0.$$
	\end{lemma}
 \begin{proof}
     This follows straightforwardly from 
            Lemma~\ref{lem:overlaps} with $j = K$ and $n = R$.
    \end{proof}

    \section{Extension embedding}\label{sec:embed}
    
    Given $d, R$, we can order the set $\cup_{j\leq R} \Sigma^j_d$. Then for a $d$-branched map $f$, we can embed $\hat I^R$ naturally into the compact space
    $$I_* := \overline{I} \times \bigcup_{j\leq R} \Sigma^j_d$$
    as follows. 
    To each $\hat x \in \hat I^R$, there is a minimal word $$\theta = \theta(\hat x) \in \bigcup_{j\leq R} \Sigma^j_d$$ for which projection of the domain of $\hat I^R$ containing $\hat x$ is 
    $f^{L(\theta)}(\cyl^\theta(f))$. Embed $\hat x$ as $(x, \theta(\hat x))$, and let us call the embedding map $\iota_R$. Recall we denote by $(\hat X(R), \hat F, \tau)$ the level-$R$ induced map. 
    \begin{definition}
        We call $(X(R), F, \tau)$ the  \emph{embedded level-$R$ induced map}, \index{XFT@ $(X(R), F, \tau)$, the embedded level-$R$ induced map } \index[def]{Induced map (embedded)} where
     $X=X(R)$ is the embedding of $\hat X = \hat X(R)$, and $F$ is the corresponding map, so 
    $$F = \iota_R \circ \hat F \circ \iota_R^{-1}.$$ 
    The domain of definition of $\tau$ is extended by  $\tau_{|X} := \tau \circ \iota_R^{-1}$. 
    \end{definition}
    We can write $\pi_1$ for projection onto the first coordinate. Then $\pi_1 \circ F = f^\tau \circ \pi_1$ on the domain of definition of $F$. 
    Of course, $F$ and $\tau$ are only defined on a subset of $I_*$. 
    One should think of $F$ as being Markovian, as it inherits the properties of $\hat F$ given in Lemma~\ref{lem:inducingcyl}. 

    To a branch $(Z, \theta)$ of $F$, there is a word $\omega$ corresponding to $Z$, such that $Z = \cyl^\omega(f)$, whose length $L(\omega)$ equals $\tau +R$ on $(Z, \theta)$. 
    There is another unique word $\theta_*$ corresponding to the domain containing $F((Z, \theta))$. The triplet $(\omega, \theta, \theta_*)$ entirely specifies the branch $(Z,\theta)$, in the sense that it tells us the interval $Z$ and the starting and image domains in $I_*$. 

    Let $\Omega$ denote the set of such triplets corresponding to branches of $F$. 
    Let $\Omega^N$ denote the set of $(\omega, \theta, \theta_*) \in \Omega$ for which $L(\omega) \leq N$. 
    Let $\Omega'$ be the (finite) set of pairs $(\omega', \theta)$, with $L(\omega')=R$, for which $(\cyl^{\omega'}, \theta)$ is a connected component of $X$. 
As such, $$X := \bigcup_{\omega' \in \Omega'} (\cyl^{\omega'}(f), \theta). $$ 
    For $k \geq 1$ and $f_k$ (in place of $f$), we introduce the notation $F_k, \tau_k, X^k, \Omega_k, \Omega^N_k$ and $\Omega'_k$. 

    For each $N\in \N$, we shall say that the sequences $(\Omega_k^N)_k$ and $(\Omega_k')_k$ are \emph{eventually constant} if, 
    for all $k,k'$ large enough, $\Omega_k^N = \Omega_{k'}^N$ and $\Omega_k'=\Omega_{k'}'$.  
    Suppose that for all $N \geq 1$, the sequences $(\Omega_k^N)_k$ and $(\Omega_k')_k$  are eventually constant.  We denote by $\Omega^N_0$ and $\Omega_0'$ their respective limits. Set $\Omega_0 := \cup_N \Omega^N_0$. 
We define 
$$Y^0 := \bigcup_{\omega' \in \Omega'_0} (\cyl^{\omega'}(f_0), \theta).  
$$
For  $(\omega, \theta, \theta^*) \in \Omega_0$ and  $(x, \theta) \in (\cyl^\omega(f_0),\theta)$, we define $F_0$ of $(x,\theta)$ by
    $$F_0(x, \theta) := (f_0^{L(\omega) - R}(x), \theta_*)$$
    and $\tau_0$  by  
    $$\tau_0 (x, \theta) =  L(\omega) -R.$$ 
    The domain and range of $F_0$ are subsets of $Y^0$. 
    Convergence of $f_k$ to $f_0$ gives  $F_0(x, \theta) = \lim_{k \to \infty} F_k(x, \theta)$ (noting that the branches are open sets). 
    It makes sense to say that the sequence $((X^k, F_k, \tau_k))_k$ is convergent and we can write (still omitting the $R$-dependence)
    \begin{equation} \label{eq:indconv}
        (X^k, F_k, \tau_k) \to (Y^0, F_0, \tau_0).
    \end{equation}
    We call $(Y^0, F_0, \tau_0)$ a (embedded) \emph{level-$R$ limit induced map}. \index{Y0@$(Y^0, F_0, \tau_0)$, (embedded) level-$R$ limit induced map}

In general, for each $N$ (and $d,R$), there are only a finite number of possibilities for $\Omega^N$.
Thus there is always a (strictly increasing) sequence $(n_k)_k$ for which for each $N \geq 1$, the subsequence $(\Omega_{n_k}^N)_{n_k}$ is eventually constant. Similarly, there is a subsequence for which $(\Omega'_{n_k})_k$ is eventually constant.

    \begin{remark} If the map $f_0$ has finite post-critical set $\cup_{n\geq 0} f^n(\V)$, then the corresponding Hofbauer extension has only a finite number of domains. This need not be true for maps $f_k$ converging to $f_0$, so there is no particular reason
        the level-$R$ limit induced map should coincide with the level-$R$ induced map for $f_0$. 
        This is why we need to consider convergence of the embeddings. 
    \end{remark}

    \section{Dynamics of rare returns} \label{sec:rarereturns}
    In this chapter we construct orbits with large inducing time, which corresponds to only returning infrequently to the bounded part $\hat X$ in the Hofbauer extension. 

    Suppose we have a convergent sequence of embedded level-$R$ induced maps 
    $$(X^k, F_k,  \tau_k) \to   (Y^0, F_0, \tau_0)$$  as per~\eqref{eq:indconv}.
     For $U,V$ connected components of $X^k$ (or of $Y^0$ if $k=0$), we can define $\tau_k(U,V)$ as follows. Let  $A$ be a subinterval of $U$ for which there is an $n_A$ with $F_k^{n_A }(A)=V$. Let $p_A = \tau + \tau\circ F_k + \cdots + \tau \circ F_k^{(n_A -1) }$. Define $\tau_k(U,V)$ as the infimum over all such $A$ of $p_A$. 
    If $\tau_k(U,V) < +\infty$, we can choose an $A$ which minimises $p_A$. Then we can define a \emph{quick word}\index[def]{Quick words}
    $$b^k(UV) = Z_0Z_1\ldots Z_{n_A - 1},$$
    where $Z_j$ is the connected component of $X^k$ containing $F_k^{j}(A)$. 

    Since $Y^0$ has a finite number of connected components, 
    $$M := \max\{\tau_0(U,V) : \tau_0(U,V) < +\infty\}$$
    is finite. 
    \begin{lemma} \label{lem:UVW}
	For each $K$, for all large $k$ the following holds. If $j\geq 1$, $U, V$ are connected components of $X^k$, 
        if $\tau_k\left(F_k^{i}(A)\right) \leq K$
	for each $i = 0,1, \ldots, j-1$ 
        and if $F_k^j$ maps $A \subset U$  homeomorphically onto $V$,
        then $\tau_k(U,V) \leq M$. 
	\end{lemma}
\begin{proof}
    By convergence, for each pair of connected components $U,V$,
    \begin{equation}\label{eq:rhok}
        \lim_{k\to \infty} \inf\{\tau_k(U,V) : \tau_k(U,V) > M\} = +\infty.
    \end{equation}
    Consequently, for all $k$ large enough and all connected components $U,V,W$ of $X^k$, provided $\tau_k(U,V), \tau_k(V,W) \leq M$, one has $\tau_k(U,W) \leq 2M$ and so
    \begin{equation}\label{eq:rhok2}
    \tau_k(U,W) \leq M. 
    \end{equation}

    Now let $U,V,A$ be as per the hypotheses. 
	For $i = 0,1,\ldots, j$, let $V_i$ be the connected component of $X^k$ containing $F_k^{i}(A)$. Let $l \leq j$ be  the maximal integer for which for each $i =0, \ldots, l$, $\tau_k(U,V_i) \leq M$. 
	If $l < j$ then $\tau_k(V_l, V_{l+1}) > M$ by (\ref{eq:rhok2}) and so 
	$\tau_k(V_l, V_{l+1}) > K$ by
	(\ref{eq:rhok}). 
        This contradicts the assumption $\tau_k\left(F_k^{l}(A)\right) \leq K$, completing the proof of the lemma.
	\end{proof}

Let $\xi_k$ denote the set of connected components of the domain of $F_k$; let $\xi_k(K)$ denote those components on which $\tau_k \geq K$. Let 
    $$ V^k_K := \bigcup_{Z \in \xi_k(K)}Z \subset X^k.$$ 
    Let (without indicating the dependence on $k$) $e_n$ be the $n^\mathrm{th}$ return time (with respect to $F_k$) to $V^k_K$ and set
$$g(y) := \sum_{i=0}^{e_1(y) -1} \tau_k(F_k^i(y)).$$

\begin{lemma} \label{lem:shadow}
        Suppose that $x \in V^k_K$ is recurrent. 

        For arbitrarily large $n$, there is a periodic point $w \in V^k_K$ of period $e_n(w)$ such that $F_k^{e_j(w)}(w)$ and $F_k^{e_j(x)}(x)$ are in the same element of $\xi_k(K)$ for $j = 0,\ldots, n$ 
        and such that if $\tau_k(F_k^{e_j(w)+1}(w)) <K$ then   $g(F_k^{e_j(w)+1}(w))  \leq M$.

        There is $y \in V^k_K$ such that $F_k^{e_n(y)}(y),$ and $F_k^{e_n(x)}(x)$ are in the same element of $\xi_k(K)$ for each $n \geq 0$ and such that if $\tau_k(F_k^{e_n(y) +1}(y)) <K$ then   $g(F_k^{e_n(y)+1}(y))  \leq M$. 
    \end{lemma}
    \begin{proof}
 If $F_k^{j}(x) \in Z_j \in \xi_k$, we obtain a sequence $\underline{x} = Z_0Z_1Z_2\ldots$. Write $V_j$ for the connected component of $X^k$ containing $Z_j$. If $n,n+1, \ldots, n+l$ satisfy 
        $$Z_j \in \xi_k \setminus \xi_k(K)$$ 
        for $j = n, \ldots, n+l$, and $Z_{n-1}, Z_{n+l+1} \notin \xi_k \setminus \xi_k(K)$,  then we call $Z_nZ_{n+1}\ldots Z_{n+l}$ a \emph{low string}. Replacing each low string by the corresponding quick word $b^k(V_nV_{n+l})$, we obtain a new sequence $\underline{w} = W_0W_1W_2\ldots$. 
        By construction, $F_k(W_j) \supset W_{j+1}$, while  $W_0 \in \xi_k(K)$. Since $x$ is recurrent, $W_0$ and $W_n$ lie in the same connected component of $X^k$ for infinitely many $n \to \infty$.
        Therefore, for a sequence of arbitrarily large $n$, there exists an $n$-periodic point $w = w_n$ such that $F_k^{j}(w) \in W_j$ for $j=0,\ldots, n$. 

        Taking a convergent sequence of $w_n$ gives a limit point $y$ which shadows $x$ forever. 

        The estimates for $g$ hold by construction. 
    \end{proof}

\chapter{Measures and entropy} \label{sec:MM}

\section{Lifts of measures}
    
Let $f : \cup_{j=1}^d I_j \to I$ be a $d$-branched piecewise monotone map. 
Given $\mu\in \M_f$, we say that $\mu$ \emph{lifts to $\hat I$} \index[def]{Lift of a measure} if there exists an ergodic $\hat f$-invariant probability measure $\hat\mu$ on $\hat I$ such that $\hat\mu\circ\pi^{-1}=\mu$.    The process, given in \cite{Kellift}, of constructing the lifted measure is as follows.  
First set  $\hat\mu_0:=\mu\circ \iota^{-1}$. Now define $\hat \mu_n$ on $\hat I$ by  
    \begin{equation}\label{eqn:Hoflift}
        \hat\mu_n:=\frac1n\sum_{k=0}^{n-1} \hat\mu_0\circ\hat f^{-k}.  
    \end{equation}
Any vague limit of $(\hat\mu_n)_n$ is an $\hat f$-invariant measure.  However, this limit might be the null measure.  Keller showed that positive entropy is enough to guarantee existence of the lifted measure:

\begin{lemma}[{\cite[Theorem~3]{Kellift}, \cite[Theorem~1]{HofIntrinsic}}] \label{lem:entropybranches}
    Every positive entropy measure $\mu \in \M_f$ lifts to a measure $\hat \mu \in \M_{\hat f}$ with $\hat \mu(\hat I) = 1$ and the same entropy, $h(\hat \mu) = h(\mu)$. The sequence of measures $\hat \mu_n, n \geq 1$, defined in~\eqref{eqn:Hoflift} converges weakly to $\hat \mu$. 
        \end{lemma}

    Recall that $\hat X = \hat X(R)$ denotes the union of elements of $\hat \P_R$ contained in $\hat I^R_-$. Note that $\hat I^R_- \setminus \hat X$ is a finite set which therefore has zero mass for any positive-entropy ergodic invariant probability measure.
    From~\eqref{eq:epsdef},~\eqref{eq:etadef}, $\eps(R) = 8 \log R/R$, while 
    $$
    \eta(R,d) = \frac{\eps(R)^2}{2R (\log d)^2}.$$
\begin{lemma} \label{lem:Rtower2} 
    For any piecewise-monotone $d$-branched map $f$ and any  $\mu \in \M_f$ with entropy $h(\mu) \geq 2\varepsilon(R)$, the lift $\hat\mu$ of $\mu$ satisfies
    $$
    \hat\mu (\hat X) \geq\eta(R,d).
    $$
\end{lemma}
\begin{proof} 
The  following is a more detailed proof of the somewhat sketchy \cite[Lemma~4]{BTeqnat} which, moreover, yields more information: in particular uniform constants for families of maps together with a tail estimate. 
  
 As in Lemma~\ref{lem:Rtower}, let $\Q$ denote the collection of elements of $\P_{nR}^0$ which visit $\hat I^R_-$ at most $\delta n$ times in the first $nR$ iterates, where $\delta = \eps(R)/4\log d$.
 By Lemma~\ref{lem:Rtower}, for all large $n$,  $\# \Q \leq \exp(nR\eps(R)).$ Of course, there is a trivial bound $\# \P_{nR} \leq d^{nR}$. 

    Now we can use these two estimates on the numbers of $nR$-cylinders to estimate entropy. 
    For any finite sets of non-negative numbers $\{a_k\}_k, \{b_k\}_k$ satisfying the relations $\sum_k a_k=a$, $\sum_k b_k = b$ and $a+b =1$, Jensen's inequality implies that 
    $$-\sum_ka_k\log a_k -\sum_k b_k \log b_k \le a\log\#\{a_k\}_k + b\log \#\{b_k\}_k + \log 2,$$  which we apply in the following entropy estimate.  
    Since the partition $\P_1$ is generating for any ergodic positive entropy measure $\mu$ and since $\iota \P^0_{nR} = \P_{nR}$, we deduce the bound
    \begin{align*}
        2\eps(R) \leq h(\mu)&\leq -\frac1{nR} \sum_{\cyl\in \P_{nR}}\mu(\cyl)\log\mu(\cyl)\\
              &\leq -\frac1{nR} \left( \sum_{\cyl\in \P_{nR} \setminus \iota \Q}\mu(\cyl)\log\mu(\cyl) 
    +\sum_{\cyl\in \iota \Q}\mu(\cyl)\log\mu(\cyl)\right)\\
    &\leq \mu(\cup_{\cyl \in \P_{nR} \setminus \iota \Q} \cyl) \log d + \eps(R) + \frac{\log 2}{nR},
    \end{align*}
    for all large $n$. 
    Therefore
    $$ \hat\mu^0(\cup_{\cyl \in \P^0_{nR}\setminus \Q} \cyl) = \mu(\cup_{\cyl \in \P_{nR} \setminus \iota \Q} \cyl)\log d \geq \varepsilon(R)  - (\log2) / nR.$$ 
    In particular,  for each large $n$, there is a set of $\hat \mu_0$-measure  at least $\varepsilon(R)/\log d  - 1/nR = 8\delta - 1/nR$ which visits 
    $\hat I^R_-$ 
    at least $\delta n$ times in the first $nR$ iterates. From the definition of $\hat \mu$, we get 
    $\hat \mu(\hat I^R_-) = \hat \mu(\hat X) \geq 8\delta^2/R$. Substituting the definition of $\delta$ gives the required result.
\end{proof}

    Recall primitive components of $\hat X$ were introduced in Definition~\ref{def:primi}.
    \begin{corollary}\label{cor:Rtowerprim}
        For any  $\mu \in \M_f$ with entropy $h(\mu) \geq 2\varepsilon(R)$, there is a primitive component $\Y$ with
    $$
    \hat\mu (\Y) \geq\eta(R,d).
    $$
\end{corollary}
\begin{proof} From ergodicity and Definition~\ref{def:primi}, $\hat \mu$ can only give mass to one primitive component. By the lemma, this component has mass at least $\eta(R,d)$. 
\end{proof}

    Given $R$ and a measure $\mu \in \M_f$ with $h(\mu) \geq 2\eps(R)$, then $\hat \mu(\hat X) \geq \eta(R,d)$ as in Lemma~\ref{lem:Rtower2}. We obtain a measure $\hat \nu =\hat  \nu^R \in \M_{\hat F}$ defined by 
    \begin{equation}\label{eq:nudef}
        \hat \nu := \frac1{\hat \mu(\hat X)} \hat \mu_{|\hat X},
    \end{equation}
        since $\hat F$ is a first return map. 
        Recall the embedding $\iota_R$ from~\S\ref{sec:embed}, giving $(X, F, \tau)$, the embedded level-$R$ induced map. 
        The measure $\hat \nu \in \M_{\hat F}$ gives an \emph{embedded induced measure} \index[def]{Induced measure (embedded)} $\nu \in \M_F$ via $\nu = \hat \nu \circ \iota_R^{-1}$. 
        In this setting, we have the following. 
        \begin{lemma} \label{lem:toight}
            Given $\gamma >0$, there exists $K \geq 1$ ($K$ depending on $d, R, \gamma$) such that, if   
            $\kappa := \min\{|V| : V \in \hat \P_K, V \subset \hat I^R \}$ then 
        $$ \nu (B(\partial X, \kappa)) < \gamma.$$ 
    \end{lemma}
    \begin{proof} 
        It suffices to show that $$ \hat \nu (B(\partial \hat X, \kappa)) < \gamma.$$ 
            By Kac' Lemma, if $\rho$ is the inducing time corresponding to $\hat F^R$, then    
            \begin{equation} \label{eq:itb}
                \int_{ B(\partial \hat X,\kappa) } \rho ~d\hat \nu \leq 
                \int_{\hat X} \rho ~d\hat \nu = 
                \frac{R}{\hat \mu(\hat X)} \leq \frac{R}{\eta(R,d)}.
        \end{equation} 
        Let $K =  R + \lceil R/(\eta(R,d) \gamma) \rceil$, and let $\kappa$ be given by the statement. 
        By Lemma~\ref{lem:boundarytime}, 
        \[\rho \geq \frac{R}{\eta(R,d)\gamma} \]
        on $B(\partial \hat X, \kappa)$. 
        Thus by (\ref{eq:itb}), 
        $
            \hat \nu (B(\partial \hat X, \kappa))  <\gamma.
            $
    \end{proof}

    \section{Spreading measures and Abramov's formula} 
    Suppose we have some induced map $(X,F, \tau)$, where $(X,F,\tau)$ is a level-$R$ induced map, an embedded level-$R$ induced map or a level-$R$ limit induced map for an initial piecewise-monotone $d$-branched map $f$. Denote by $\pi$ the projection from $X$ to $I$. 
    Let $\nu$ be an $F$-invariant probability measure, possibly non-ergodic. The pushforward $\pi_* \nu$ is a measure on $I$. Write $\nu^n$ for the restriction of $\nu$ to the set $\tau^{-1}(n)$ and put
    $$T = T_\nu := \int \tau\, d\nu.\index{Tnu@ $T_\nu = \int \tau\, d\nu.$}
    $$
    \begin{definition} \label{dfn:spread}
        The \emph{spread} of $\nu$ is the measure \index[def]{Spread of a measure}
        $$
        \frac{1}{T} \sum_n \sum_{j = 0}^{n-1} f^j_*(\pi_*\nu^n),$$
        which is well-defined if $T < \infty.$ 
    \end{definition}
    Writing $\mu$ for the spread of $\nu$, Abramov's formula (\cite{Abramov} or \cite{Zwe05}) states that
    \begin{equation} \label{eq:Abe}
        h(\mu) = \frac{1}{T} h(\nu).
    \end{equation}
    If $\nu$ is ergodic and $f$ is piecewise differentiable, then straightforward application of Birkhoff's ergodic theorem (or \cite[Theorem~2.3]{PeSe}) gives
    \begin{equation} \label{eq:AbeDf}
        \int \log |Df|\, d\mu  = \frac{1}{T} \int \log |DF|\, d\nu.
    \end{equation}
    If $\nu$ is not ergodic, use ergodic decomposition: one can write $\mu = \int_{\M_f} \mu' \,  dm_1(\mu')$ and 
    $$\nu = \int_{\M_f} \frac{T}{T_{\nu'}} \nu' \, dm_1(\mu')$$ for some probability measure $m_1$ and with $\mu'$ the spread of $\nu'$. Then 
    \begin{eqnarray*}
        \int \log |Df| \, d\mu &=& \int_{\M_f} \int \log|Df| \, d\mu' \, dm_1(\mu') \\
                               &=& \int_{\M_f} \frac1{T_{\nu'}} \int \log|DF| \, d\nu' \, dm_1(\mu') \\
                               &=& \frac1T\int \log |DF| d\nu,
    \end{eqnarray*}
    so~\eqref{eq:AbeDf} still holds. 

    Naturally, given $(\hat X(R), F, \tau)$, if $\mu \in \M_f$ has entropy at least $2\eps(R)$ with lift $\hat \mu \in \M_{\hat f}$, and $\hat \nu \in \M_{\hat F}$ is the normalised restriction to $\hat X(R)$ (which by Lemma~\ref{lem:Rtower2} is well-defined), then $\hat \nu$ spreads to $\mu$. 

    \chapter{Light limit measures and upper-semicontinuity of metric entropy} \label{sec:llm}

     Let $(f_k)_k$ be a sequence  of $d$-branched piecewise-monotone maps converging to $f_0$ as $k \to \infty$ and having decreasing critical relations. 
     Let $R \geq 8d$. 
     For each $k \geq 1$, let $\mu_k \in \M_{f_k}$ have $h(\mu_k) \geq 2\eps(R)$ and suppose the sequence $(\mu_k)_k$ converges to some limit measure. 
     Let 
     $$((X^k, F_k, \tau_k))_k \to (Y^0, F_0, \tau_0)$$ 
     be a convergent sequence of level-$R$ induced maps, as in~\S\ref{sec:embed}.
     In particular,  the sequences $(\Omega_k^N)_k$ are eventually constant for each $N$. 
     We obtain the measures $\hat \nu_k \in \M_{\hat F_k}$ from~\eqref{eq:nudef}, together with their embeddings $\nu_k \in \M_{F_k}$. 
     Kac' Lemma says that 
     $$
     \int_{\hat X^k} \tau_k ~d\hat \mu_k = 1$$
     so,
     writing $\eta = \eta(R,d)$, we have the following bound on the inducing time
     \begin{equation}\label{eq:Kac1}
         \int \tau_k ~d\nu_k  = \frac{1}{\hat \mu_k(\hat X^k)} \leq \frac1\eta.
     \end{equation}
     Since $Y_0$ is a subset of the compact space $I_*$, there is a convergent subsequence of the measures $\nu_k$ whose limit is a probability measure supported on $\overline{Y^0}$. Let us suppose the sequence itself actually converges; the limit measure we shall denote $\nu_0 = \nu_0^R$. Its spread will be a light limit measure, as shown later in the chapter.
     \begin{lemma} \label{lem:nu0inv}
        The measure $\nu_0$ is $F_0$-invariant. 
    \end{lemma}
    \begin{proof}
        Let $V$ be a connected component of $Y^0 \subset I_*$, so $\pi_1 V$ is an $R$-cylinder.
        Let $U$ be an interval compactly contained in $V$ such that 
        $$\nu_0(\partial U) = \nu_0(\partial(F_0^{-1} (U)))= 0,$$ that is, the boundaries of $U$ and $F_0^{-1}(U)$ have zero measure. Since $\nu_0(\partial U) = 0$, 
        \begin{equation}\label{eq:nuconv}
        \lim_{k \to \infty} \nu_k(U) = \nu_0(U).
    \end{equation}
        Now each connected component of $F_0^{-1}(U)$ is compactly contained in a branch of $F_0$ and so $\partial(F_0^{-1}(U) \cap \{\tau_0 \leq K\}) \subset \partial(F_0^{-1}(U))$ has zero measure for each $K$. 
        For any $K$, since $\Omega^K_k = \Omega^K_0$ for all large $k$, the branches with inducing time bounded by $K$ converge. Therefore, 
        $$\lim_{k \to \infty} \nu_k \left(F_k^{-1}(U) \cap \{\tau_k \leq K \} \right) = \nu_0\left(F_0^{-1}(U) \cap \{\tau_0 \leq K \} \right). $$
        Meanwhile,~\eqref{eq:Kac1} implies  $\nu_k(\{\tau_k > K\}) < 1/K\eta$. Consequently (and again using that the boundary has zero measure with respect to $\nu_0$),
        \begin{equation}\label{eq:nuconv2}
        \lim_{k \to \infty} \nu_k \left(F_k^{-1}(U)  \right) = \nu_0\left(F_0^{-1}(U)  \right). 
    \end{equation}
    Since $\nu_k(U) = \nu_k(F^{-1}(U))$,~\eqref{eq:nuconv}-\eqref{eq:nuconv2} imply that $\nu_0(U) = \nu_0( F^{-1}(U))$. 
        Approximating general sets compactly contained in $V$ by open covers with massless boundary, one can then show that $\nu_0 = \nu_0 \circ F_0^{-1}$ on compact subsets of $V$ and thus on $V$. It remains to treat boundaries of sets like $V$, that is, of $R$-cylinders. 

        So let $Z^0$ be an $R$-cylinder for $f_0$ and let $(Z^0, \theta) \subset I_*$. Let $Z^k$ denote the corresponding cylinder for $f_k$, with a boundary point $p_k$ such that $p_k$ converges to a boundary point $p_0$ of $Z^0$. 
        Set $p_k' := (p_k, \theta)$. 
        Let $\gamma >0$ and let $K> R$ be given by Lemma~\ref{lem:toight}. Let $\kappa>0$ be given by Lemma~\ref{lem:kappaPk}. [Note that if $p_k' \notin \partial \iota_R \hat X_k$ then $\nu_k(B(p_k', \kappa)) = 0$.]
        Then Lemma~\ref{lem:toight} says that $$\hat \nu_k(B(\partial \hat X^k, \kappa)) < \gamma$$ for all large $k$. 
        In particular, $\nu_k (B(p_k', \kappa)) < \gamma$. Thus $\nu_0(B(p_0', \kappa/2)) \leq \gamma$. Letting $\gamma \to 0$, we conclude that the points projecting to boundaries of cylinders have zero $\nu_0$-measure, so $\nu_0 = \nu_0 \circ F^{-1}$ everywhere. 
    \end{proof}

    Let $T_k := \int \tau_k d \nu_k$;\index{TkR@ $T_k^R = \int \tau_k d \nu_k$} we will write $T_k^R$ when we need to indicate dependence on $R$. Suppose that $T_k$ converges to a finite quantity $T_L = T^R_L$,\index{TLR@ $T_L^R = \lim_{k\to \infty}T_k^R$} as $k \to \infty$. 
   Let $T_0  = T^R_0:= \int \tau_0 d \nu_0$ and set $T^R_\Delta := T^R_L - T^R_0$.\index{TRD@$T^R_\Delta := T^R_L - T^R_0$} Then
   \begin{equation}\label{eq:t00}
       T_0 = \lim_{K\to \infty} \lim_{k \to \infty} \int_{\tau_k \leq K} \tau_k ~d\nu_k \leq T_L \leq \frac1\eta.
   \end{equation}
   In particular, $T_\Delta^R \geq 0$. Since $\hat \mu_k(\hat X^k(R))$ is increasing with $R$,   \eqref{eq:Kac1} implies that $T_k^R$ is decreasing as $R$ increases and
   \begin{equation}\label{eq:tRR}
       0 \leq T_\Delta^{R'} \leq T_L^{R'} \leq T_L^R \quad \text{ for all } R' \geq R.
   \end{equation}

   \begin{lemma} \label{lem:spreadabscns}
       For some measure $\mu' \in \M_{f_0}$,
           $$
           \lim_{k\to\infty} \mu_k = \frac{T^R_0}{T^R_L} \mu_0^R + \left(1- \frac{T^R_0}{T^R_L}\right) \mu',$$
           where $\mu_0^R$ denotes the spread of $\nu_0$ ($\nu_0 = \nu_0^R$).
       \end{lemma}
       \begin{proof}
           Since the limit measure $\nu_0$ is $F_0$-invariant, it gives no mass to boundaries of cylinder sets. Denote by $\nu_k^n$ the restriction of $\nu_k$ to $\tau^{-1}(n)$. 
           It follows that
           $$
           \mu_{k,N} := \frac{1}{T_k} \sum_{n=1}^N \sum_{j = 0}^{n-1} (f^j_k)_*(\pi_*\nu_k^n)
           \to 
           \frac{1}{T_L} \sum_{n=1}^N \sum_{j = 0}^{n-1} (f^j_0)_*(\pi_*\nu_0^n)
           $$
           as $k \to \infty$. The difference between $\mu_{k,N}$ and $\mu_k$ tends to zero uniformly as $N\to \infty$.  The right-hand side converges to  $\frac{T^R_0}{T^R_L} \mu_0^R$ as $N\to \infty$, and the result follows.
       \end{proof}

     Let us suppose in addition that $h(\nu^R_k)$ converges to a limit $h_L = h^R_L$ as $k \to \infty$.\index{hnu@$h(\nu^R_L)= \lim_{k\to\infty}h(\nu^R_k)$} 
     Then set 
     \begin{equation} \label{eq:hdeltadef}
         h^R_\Delta : = h_L^R - h(\nu^R_0).
     \index{hrd@$ h^R_\Delta  = h_L^R - h(\nu^R_0)$}
     \end{equation}

    \begin{lemma}\label{lem:hnusemi}
        For all  $R$ with $\eps(R) \leq \inf_k h(\mu_k)$,  
        $$h^R_\Delta \leq 5\eps(R)T^R_\Delta.$$ 
        In particular,
        $\lim_{R\to\infty} (h^R_L - h(\nu^R_0)  ) \leq 0$. 
    \end{lemma}
    \begin{proof} 
        We shall fix $R$ and drop dependence on $R$ from the notation.
        Let $\xi_k$  denote the collection of branches of $F_k$.  It is a generating partition for $\nu_k$ with entropy $h(\nu_k) < \infty$. 
   Let $$\xi_k^N : = \bigvee_{j=0}^{N-1} F_k^{-j} \xi_k.$$ 
   Suppose $h_\Delta >0$, for otherwise there is nothing to prove, since $T_\Delta \geq 0$. 
   Let $\gamma := h_\Delta/15$. Then there is an $N$ for which  
   $$
   h(\nu_0) + \gamma > 
   \frac1N H(\nu_0, \xi^N_0) 
   := -\frac1N \sum_{Z \in \xi^N_0} \nu_0(Z) \log \nu_0(Z). $$
   Let $\rho = \rho_k := \tau + \tau \circ F_k + \cdots + \tau \circ F_k^{N-1}$, the $N^\mathrm{th}$ return time for $F_k$. 
   Let $\xi^N_k(K)$ denote the elements of $\xi_k^N$ on which $\rho \geq K$. 
   Since $H(\nu_0, \xi^N_0)$ is finite, for each  $K$ large enough  $$-\frac 1N \sum_{Z \in \xi^N_0(K)} \nu_0(Z) \log \nu_0(Z) < \gamma.$$ 
   Then for large $k$, 
   \begin{equation}
       \begin{split}
           h(\nu_0) + \gamma 
           & > -\frac{1}{N} \sum_{Z \in \xi^N_0\setminus \xi^N_0(K)} \nu_0(Z) \log \nu_0(Z) \\
           & > -\frac{1}{N} \sum_{Z \in \xi^N_k\setminus \xi^N_k(K)} \nu_k(Z) \log \nu_k(Z) - \gamma \\
           & \geq -\frac{1}{N}  H(\nu_k, \xi^N_k)  
           - \frac{-1}N \sum_{Z \in  \xi^N_k(K)} \nu_k(Z) \log \nu_k(Z)   - \gamma \\
       & \geq h(\nu_k) 
       - \frac{-1}N \sum_{Z \in  \xi^N_k(K)} \nu_k(Z) \log \nu_k(Z) - \gamma.  \\
   \end{split}
   \end{equation}
   Passing to the second line required convergence of the 
    measures on $\{\rho \leq K\}$.
   Rearranging and replacing $h(\nu_k)$ by (a lower bound) $h_L - \gamma$ produces
   \begin{equation} \label{eq:3g}
        \frac{-1}N \sum_{Z \in  \xi^N_k(K)} \nu_k(Z) \log \nu_k(Z) 
           > h_\Delta - 3\gamma 
\end{equation}
for all sufficiently large $k$. 

Thus we have a first  upper bound for $h_\Delta$. The exponential tails estimate of Lemma~\ref{lem:tails} will permit us to rewrite the upper bound in terms of $\eps$ and $T_\Delta$. 

    Set $$ \beta_{n,k} := \nu_k(\{\rho=n\}).$$ 
    The $F_k$-invariance of $\nu_k$ implies $\int\rho_k~d\nu_k=NT_k$, so 
    $$\lim_{k \to \infty} \sum_{n\geq 1} n\beta_{n,k} = \lim_{k\to \infty} NT_k =  N T_L. $$
    Hence 
    $0 \leq \sum_{n\geq K} \beta_{n,k} \leq NT_k/K \leq (NT_L + 1)/K$ (for large $k$) which goes to $0$ as $K \to \infty$. 
    As a consequence,
    \begin{equation}\label{eq:NTD}
        \lim_{K \to \infty} \lim_{k \to \infty} \sum_{n\geq K} n\beta_{n,k} = N T_\Delta. 
\end{equation}
For $\eps = \eps(R)$, Lemma~\ref{lem:tails} says that for some $n_1$ independent of $k$, for all $n \geq n_1$, 
    $$\#\{Z \in \xi^N_k : \rho(Z) = n\} \leq \exp(3n \eps).$$
    Thus Jensen's inequality gives 
    \begin{equation}\label{eq:Jensen1}
        -\frac1N \sum_{Z \in \xi^N_k : \rho(Z) = n} \nu_k(Z) \log \nu_k(Z)  \leq \frac{\beta_{n,k} 3n \eps}{N} -\frac1N \beta_{n,k}\log \beta_{n,k}.
\end{equation}
Splitting the following sum over $n \geq K \geq n_1$ in two, depending on whether $\beta_{n,k}$ is greater or smaller than $e^{-n\eps}$, one deduces that 
    \begin{equation}\label{eq:NTD2}
    \begin{split}
        \sum_{n\geq K} -\frac1N \beta_{n,k} \log \beta_{n,k} & \leq \frac1N \sum_{n \geq K} n\eps \beta_{n,k} + \frac1N \sum_{n\geq K} e^{-n\eps}n\eps \\
        & 
        \leq \frac1N \sum_{n \geq K} n\eps \beta_{n,k} + \frac1K,
    \end{split}
\end{equation}
    say, for all large $K$.

    From~\eqref{eq:3g},~\eqref{eq:Jensen1} and~\eqref{eq:NTD2}, for all $K$ large enough, 
    \begin{align*}
        h_\Delta - 3\gamma & < \sum_{n\geq K} \frac{\beta_{n,k} 3n \eps}{N} -\frac1N \beta_{n,k}\log \beta_{n,k} \\
                            & \leq \frac1{K}  + \frac{4\eps}{N} \sum_{n \geq K}n\beta_{n,k}.
    \end{align*}
    Applying~\eqref{eq:NTD} and letting $K \to \infty$, $h_\Delta - 3\gamma \leq 4\eps T_\Delta$. Since $\gamma = \eps h_\Delta/15$,  
    we obtain 
    \begin{equation}\label{eq:hT1}
        h_\Delta \leq 5 \eps T_\Delta,
    \end{equation}
    as required.
\end{proof}
   
\begin{proof}[Proof of Theorem~\ref{thm:introusc}]
    Now let us prove upper semi-continuity of entropy for convergent sequences of measures $\mu_k \in \M_{f_k}$ with limit $\mu_\infty$, where $(f_k)$ is a convergent sequence of $d$-branched maps with decreasing critical relations. By passing to a subsequence, we may assume that $h(\mu_k)$ converges to a limit, and moreover that this limit is strictly positive, for otherwise the result is trivial. Take $R_0$ large enough that $\eps(R_0)$ is less than this limit. 
    Passing to a subsequence if necessary, construct a convergent sequence of embedded level-$R$ induced maps for each $R \geq R_0$.  
    Obtain $\nu^R_k, T^R_k$ as before. Assume without loss of generality that $\nu^R_k$ converges to $\nu^R_0$, $h(\nu^R_k)$ to $h^R_L$ and $T^R_k$ to $T^R_L$, for all $R\geq R_0$. Define $h^R_\Delta, T^R_\Delta$ as before.  Let $\mu_0^R$ denote the spread of $\nu_0$. 
    By Lemma~\ref{lem:spreadabscns},~\eqref{eq:Abe} and Lemma~\ref{lem:hnusemi} consecutively,
    \begin{eqnarray*}
        h(\mu_\infty) &\geq& \frac{T^R_0}{T_L^R} h(\mu^R_0) \\
                      & =& \frac{1}{T_L^R} h(\nu^R_0)\\
        &\geq& 
        \frac{1}{T_L^R} (h^R_L -5\eps(R) T^R_\Delta)\\
        & =& \lim_{k\to \infty} h(\mu_k)   -5\eps(R) \frac{T^R_\Delta}{T_L^R}.
\end{eqnarray*}
This holds for all large $R$. As  $R\to \infty$, $\eps(R)\to0$ while, by definition,   ${T_\Delta^R} \leq {T_L^R}$, so 
    $$
    h(\mu_\infty) \geq \lim_{k\to \infty} h(\mu_k)   
    $$
    as required. 
\end{proof}
    
        \chapter{Non-positive Schwarzian derivative} \label{sec:NSD}

In this section we derive and use the distortion results required for our main theorems.

\begin{definition}
    If $g:Z\to g(Z)$ is a diffeomorphism on an interval $Z$, we say that the \emph{Minimum Principle} \index[def]{Minimum Principle} holds  if, for any $a,b \in Z$ and $x \in (a,b)$,  $$|Dg(x)| \geq \min\{|Dg(a)|, |Dg(b)|\}.$$ 
\end{definition} 
For diffeomorphisms $g$ with non-positive Schwarzian derivative, $|Dg|^{-1/2}$ is convex, so the Minimum Principle holds. 
                   To see a nice treatment of the following in our $C^2$ setting, see \cite[Chapter~3]{Ced06}. 
        
\begin{lemma}[{\cite[p18, p24]{MisIHES1981}}] \label{lem:koebe}
            Let $Z \subset Y$ be intervals and let $g$ be a $C^2$ diffeomorphism from $Y$ to $g(Y)$ with non-positive Schwarzian derivative. Extensibility implies \emph{bounded distortion}: \index[def]{Bounded distortion} if $\kappa>0$ and $\dist(g(Z), \partial g(Y)) \geq \kappa |g(Z)|$, then
            for all $x, x' \in Z$,
                   $$ 
                   \frac{|Dg(x)|}{|Dg(x')|} \leq  
                \left(\frac{1+\kappa}\kappa\right)^2.
                   $$ 
               \end{lemma}

    \begin{lemma} \label{lem:distn}
        Let $(f_k)_k$ be a sequence converging to $f_0$ as $k \to \infty$, having decreasing critical relations and non-positive Schwarzian derivative. Let $(X^k,\hat F_k,\tau_k)$ be the corresponding level-$R$ induced maps. For each $R$, there is a $C > 1$ such that 
     $$
     \left|\frac{D\hat F^n_k(\hat x)} {D\hat F^n_k(\hat x')}\right| \leq C
     $$
     for each $n,k$ and each $\hat x, \hat x'$ in the same branch of $\hat F_k^n$. In particular, there is a uniform lower bound on $|D\hat F_k|$ independent of $k$. 
 \end{lemma}
 \begin{proof}
     Between any point of $\hat X^k(R)$ and the boundary of the element of $\D$ containing the point, there lies an element of $\hat P^k_R$ of length at least $\kappa >0$, by Lemma~\ref{lem:kappaPk}. 
      Lemma~\ref{lem:MarkovH} and  the standard Koebe distortion bound Lemma~\ref{lem:koebe} will then imply the first statement. The second follows from the first since the components of $\hat X^k$ have length at least $\kappa$.
    \end{proof}
    To state the next lemma, we define
    $$\lambda_{\min}:= \max\left(0, \inf\left\{\frac1n\log|Df^n(x)|: f^n(x)=x, \ n\in \N\right\}\right),$$
      and $$\lambda_{\max}:= \sup\left\{\frac1n\log|Df^n(x)|: f^n(x)=x, \ n\in \N\right\}.$$
    
    \begin{lemma}\label{lem:nsdbounds}
        Let $(\Y,\hat F_\Y,  \tau)$ be a primitive level-$R$ induced map. There is a $C > 0$ such that 
        $$
        \lambda_{\min} \tau - C \leq \log |D\hat F_\Y| \leq \lambda_{\max} \tau +C.$$
    \end{lemma}
    \begin{proof}
        Let $Y_1, Y_2$ be  connected components of $\Y$. Since $\Y$ is primitive, $Y_1,Y_2$ are linked. There is a minimal $C_{1,2}\geq1$ for which there exists a subinterval $V$ of $Y_1$ mapped by some iterate $\hat f^n_\Y$ onto $Y_2$ with derivative bounded inside $[C_{1,2}^{-1}, C_{1,2}]$ and with $n\leq C_{1,2}$.  
        Let $C'$ be the maximum of $C_{1,2}$ over all pairs $Y_1,Y_2$ of connected components of $\Y$. 

        Now let $\hat X_i$ be a branch of $\hat F_\Y$, with $\tau = \tau_i$ on $\hat X_i$. For some $n$ with $0 \leq n \leq C'$,  $\hat f^{n+\tau_i}$ maps a subinterval $W\subset \hat X_i$ onto $\hat X_i$ with 
        $$
        \log|D \hat f^{n+\tau_i}|_{|W} - \log C' -C'' \leq \log |D\hat F_\Y|_{|\hat X_i} \leq \log|D \hat f^{n+\tau_i}|_{|W} +\log C' +C'', 
        $$
        where $C''$ comes from bounded distortion, and $\tau \leq n \leq C +\tau$. There is a repelling periodic point in $W$, necessarily with exponent in $[\lambda_{\min}, \lambda_{\max}]$. Taking $C = C'\lambda_{\max} + \log C' + C''$, the  result follows. 
    \end{proof}

        \begin{lemma} \label{lem:MPexp}
            Let $\kappa >0$. 
            Let $Z \subset Y \subset g(Z)$ be intervals, where $g$ is a $C^2$ diffeomorphism from $Y$ to $g(Y)$ with non-positive Schwarzian derivative. Suppose $\dist(g(Z), \partial g(Y)) \geq \kappa |g(Z)|$. 
            Then at the fixed point $p$, $|Dg(p)| \geq \sqrt{1+\kappa}$. 
        \end{lemma}
        \begin{proof} Suppose first that $g$ is orientation-preserving. There is, by the Mean Value Theorem, a point in each component of $Y \setminus \{p\}$ with derivative $\geq 1 + \kappa$. By the Minimum Principle, the derivative at $p$ is $\geq 1+\kappa$. If $g$ is orientation-reversing, apply this argument to $g^2$. 
        \end{proof}
        \begin{lemma} \label{lem:lbd}
            Given $\kappa >0$, there are $\gamma, \alpha > 0$ such that the following holds. 
            Let $Z \subset Y \subset g(Z)$ be intervals, where $g$ is a $C^2$ diffeomorphism from $Y$ to $g(Y)$ with non-positive Schwarzian derivative. Suppose $$\dist(g(Z), \partial g(Y)) \geq \kappa |g(Z)|.$$ 
	    If $|Z| \leq \gamma |g(Z)|$ then $|Dg_{|Z}| > 2$.
            If $|Z| \geq \gamma|g(Z)|$ then $g^{-2}(Z)$ and the components of $Z\setminus g^{-2}(Z)$ all have length bounded by $|Z|(1-\alpha)$. 
        \end{lemma}
        \begin{proof} 
	The first statement is obvious, by bounded distortion.
		For the second statement, since $Z \geq \gamma |g(Z)|$, by bounded distortion there is a uniform lower bound on $|g^{-2}(Z)|/ |Z|$, giving the upper bound on the components of $Z \setminus g^{-2}(Z)$. By Lemma~\ref{lem:MPexp}, at the fixed point $p$ of $g$, $Dg^2 \geq 1+\kappa$, so there is a uniform neighbourhood of $p$ with $Dg > \sqrt{1+\kappa}$. Using the Minimum Principle, a uniform upper bound on $|g^{-2}(Z)|$ follows easily. 
        \end{proof}

        \begin{lemma} \label{lem:expandingscheme}
            Given $\kappa >0$ and integers $d, R \geq 1$, there is a $K\geq 1$ such that the following holds. Let $(\hat{X}, \hat F, \tau)$ be a level-$R$ induced map for a $d$-branched piecewise-monotone map $f$ with non-positive Schwarzian derivative. 
            Suppose that $\kappa \leq \min \{|Y|/|I| : Y \in \hat \P_R, \quad Y \subset \hat I^R\}$. 
            Then $|D\hat F^{K}|>2$.
        \end{lemma}
        \begin{proof} 
        For each domain $D$ of $\hat I^R$ containing a point of $\hat X$, $\dist(\hat X, \partial D) > \kappa |I|$, while $|D|\leq  |I|$.
            Let $n \geq 1$ and let $Z \subset \hat X$ be a branch of $\hat F^n$, so $\hat F^n$ maps $Z$ homeomorphically onto a connected component of $\hat X$ contained in a domain $D$ of $\hat I^R$, with $\hat F^n|_{Z} = \hat f^{j}$ say. 
	There is, by Lemma~\ref{lem:MarkovH}, an interval $Y \supset Z$ mapped homeomorphically by $\hat f^{j}$ onto $D$.  
            Since $j \geq n$, $Y \subset \hat X$. 
            Lemma~\ref{lem:koebe} then gives a positive lower bound, depending only on $\kappa$, for $|D\hat F^n|$. 

	Let $\hat G$ denote the first return map to a connected component $\hat X_G$ of $\hat X$. Since there are at most a finite number (depending on $d,R$) of connected components, it suffices to show that there is an $N$ for which $|D\hat G^N| >2$. 
	The general result follows because high iterates can be decomposed into a sequence of repeated first return maps bookended by two maps with derivative bounded below.

	So let $Z$ be a branch of $\hat G^n$, for some $n \geq 1$, with $\hat G^n|_{Z} = \hat f^{j}$ say.  As before, there is an interval $Y \supset Z$ mapped homeomorphically by $\hat f^j$ onto $D$ the domain of $\hat I^R$, now containing $Z$. 
            
	Let $\gamma, \alpha >0$ be given by Lemma~\ref{lem:lbd}. If $|Z| \leq \gamma |\hat X_G|$ then $|D \hat G^n| >2$. Otherwise, by the second part of Lemma~\ref{lem:lbd}, the branches of $\hat G^{3n}$ contained in $Z$ have length bounded by $|Z|(1-\alpha)$. An inductive argument gives some $N$ for which all branches of $\hat G^N$ have length at most $\gamma |\hat X_G|$, and thus $|D \hat G^N| >2$, the required derivative estimate.  
        \end{proof}
        From the uniformity of the preceding lemma ensues:
        \begin{corollary} \label{cor:expsch}
            Let $(Y^0, F_0, \tau_0)$ be a level-$R$ limit induced map. For some $K\geq 1$, $|DF_0^K| >2$. 
        \end{corollary}
        Denote by $\hat \K$ the set of periodic points $\hat p \in \hat I$ with multiplier $\pm 1$, that is, for which, for some $n\geq1$,  $\hat f^n(\hat p) = \hat p$ and $D\hat f^n(\hat p) = 1$. It follows from  the Minimum Principle that if $\hat p \in D \in \D$, then on at least one component of $D\setminus \{\hat p\}$, $\hat f^n$ is the identity and the component lies in $\hat \K$. It similarly follows from the Minimum Principle that no periodic point in $\hat I$ can be attracting. 

        \begin{lemma} \label{lem:recplle}
            If   $\hat x = (x, D_{\hat x}) \in \hat I$ is non-periodic and recurrent, then all bar at most two points of 
            $\{\hat f^n(\hat x): n \geq 0\} \cap D_{\hat x}$ 
            lie in the domain of a level-$R$ induced map for all sufficiently large $R$. 
    \end{lemma}
    \begin{proof}
        Consider $q > j \geq 0$ such that 
         $\hat y = \hat f^j(\hat x) \in D_{\hat x}$ and 
         such that 
         $\hat z = f^q(\hat x) \in D_{\hat x}$. 
    It suffices to show that $\hat y$ and $\hat z$ are in different elements of $\hat \P_R$ for all large $R$, 
    for then any point in the set
            $\{\hat f^n(x): n \geq 0\} \cap D_{\hat x}$. 
         which is neither its infimum nor its supremum will be in a non-boundary $R$-cylinder for large $R$. 
         We prove by contradiction, assuming that, for each $R$, $\hat y$ and $\hat z$ are in the same element of $\hat \P_R$. Let $k = q-j$, so $\hat z = \hat f^k(\hat y)$. 
         Then, for each $R$, all points $\hat f^{sk}(\hat y)$ with $s\geq0$ lie in the same element of $\hat \P_R$. Hence the sequence $(\hat f^{2sk}(\hat y))_{s\geq 0}$ is strictly monotone. From this, one deduces that, for each large $S$ and for $n \geq S$, 
         $$\hat f^n(\hat y) \in \bigcup_{i=0}^{2k-1} \hat f^i \left([\hat f^{2Sk}(\hat y), \lim_{s\to \infty} \hat f^{2sk}(\hat y)]\right),$$
         an obstacle to recurrence. 
    \end{proof}
    We remark that we do not have exponential growth rates (necessarily), but at least $\limsup_{n\to\infty}|D\hat f^n(\hat x)| = \infty$ for recurrent points $\hat x \notin \hat \K$. 
    By Lemmas~\ref{lem:expandingscheme}-\ref{lem:recplle} and Birkhoff's theorem, one obtains:
        \begin{corollary} \label{cor:ple}
            Every measure $\hat \mu \in \M_{\hat f}$, not supported on $\hat \K$, has  positive Lyapunov exponent. 
        \end{corollary}
    Given $(\hat X, \hat F, \tau)$ or its embedding $(X,F,\tau)$, we say that $\log |D \hat F|$ (or equivalently $\log|DF|$) is \emph{locally H\"older} if there are $C_0, \alpha>0$ such that, for each $n$ and on every branch $Z$ of $F^n$, 
    $$
    \left| \log|DF(x)| - \log |DF(y)| \right| \leq C_0 \exp(-\alpha n).
    $$

    \begin{lemma}
        Let $f\in \F \in \F_M$ and let $(\hat X, \hat F, \tau)$ be the level-$R$ induced map.
        Then $\log|DF|$ is locally H\"older.
        \label{lem:locHol}
    \end{lemma}

\begin{proof}
    For ease of notation, we drop the hats and work with the embedded induced map $F$. 
    We have 
    $$\kappa := \min\{|Y| : Y \in \hat P_R, \quad Y \subset \hat I^R\} >0.$$ 
     By Lemma~\ref{lem:expandingscheme}, there is some $N_0$ for which $|DF^{N_0}| >2$.  
    If $Z$ is a branch of  $F^n$ with $n \geq kN_0+1$ say, then $|F(Z)| < |I| \exp(-2k)$, since $2 < |DF^{N_0}|$.
    The map $F$ on $Z$ extends  with non-positive Schwarzian derivative to map a larger interval homeomorphically onto an interval $X'$ with 
    $$\dist(F(Z), \partial X') \geq \kappa \geq \frac{\kappa}{|I|}   |F(Z)| \exp ( 2k). $$
    Applying the Koebe Lemma~\ref{lem:koebe} gives $\big|\log|DF(x)| -\log|DF(y)|\big| \leq C_0 \exp(-2k)$ for all $x,y \in Z$, for some constant $C_0$ depending just on $\kappa$ and $|I|$.  
    Now replacing $k$ by $(n-1)/N_0$, we obtain
    $$
    \left| \log|DF(x)| - \log |DF(y)| \right| \leq C_1 \exp(- n/N_0),
    $$
    for some constant $C_1$, as required. 
\end{proof}

 
    \chapter{Almost upper-semicontinuity of the free energy} 
    \label{sec:SSSS}
    The proof of the following proposition occupies this chapter. It implies most of
    Theorem~\ref{thm:SSSS}, the remaining ingredient being Theorem~\ref{thm:eqstates}. 
   
\begin{proposition} \label{prop:SSSS}
    Let 
 $(f_k)_{k\ge 0} \in \FNSD$ be a sequence converging to $f_0$ as $k \to \infty$ with decreasing critical relations. 
 Suppose  that, as $k \to \infty$,
\begin{enumerate}[label=({\alph*}), itemsep=0.0mm, topsep=0.0mm, leftmargin=7mm]
\item $t_k \to t_0$;
\item $(\mu_k)_{k\geq1}$ is a convergent sequence of measures $\mu_k\in \M_{f_k}$;
\item $E(f_k, \mu_k, -t_k \log |Df_k|)$ converges to a limit $E_L$ as ${k \to \infty}$;
\item $E_L \geq \liminf_{k \to \infty}E_+(f_k, -t_k\log|Df_k|) $; 
\item $h(\mu_k)$ converges to a strictly positive limit.
\end{enumerate}
Then, writing $E(\mu_*)$ for $E(f_0,\mu_*, -t_0 \log |Df_0|)$, 
some light limit measure $\mu_*$ is hyperbolic and satisfies one of the following statements:
    \begin{itemize}
     \item
         $E(\mu_*) > E_L$;\\ 
        \item 
            $\mu_* = \lim_{k\to \infty} \mu_{k}$ and,
            for some strictly increasing subsequence $(k_n)_n$, \\
            $\int \log |Df_0| ~d\mu_* = \lim_{n \to \infty} \int \log |Df_{k_n}| ~d\mu_{k_n}$, \\ $h(\mu_*) \geq \lim_{n\to\infty} h(\mu_{k_n}) $ and  $E(\mu_*) \geq E_L$; \\  
     \item
         $E(\mu_*) \geq E_L$ and $h(\mu_*) > \lim_{k\to\infty} h(\mu_k)$.
    \end{itemize}
If $E_L= P(-t_0 \log |Df_0|)$,  one of the last two alternatives holds and some light limit measure is a convex combination of positive entropy  equilibrium states for this potential. 
    \end{proposition}
    
    \begin{proof}
    Let $\eps' = \lim_{k\to \infty}h(\mu_k)/2$. Without loss of generality, we assume that $h(\mu_k) \geq \eps'$ for all $k\geq 1$. 
    We note that, by Ruelle's Inequality, each $\mu_k \in \wtm_{f_k}.$
    Let $R_0 \geq 8d$ be big enough that $2\eps(R_0) < \eps'$,  where $\eps(R_0)$ is defined in~\eqref{eq:epsdef}. For each $R \geq R_0$, we could construct in~\S\ref{sec:embed} sets $\Omega_k, \Omega^N_k, \Omega_k'$. Let $\mu_\infty = \lim_{k\to \infty} \mu_k$. Taking subsequences, we can assume that, as $k \to \infty$,
           for each $R\geq R_0$ and $N\geq 1$, the sequences $(\Omega_k^N)_k$ and $(\Omega_k')_k$ are eventually constant.
   This allows us, for each $R$,  to obtain convergent sequences of embedded level-$R$ induced maps
   $$
   (X^k(R) = X^k, F_k, \tau_k) \to (Y^0, F_0, \tau_0).$$ 
   For $k \geq 1$, we obtain the embedded induced measures $\nu^R_k \in \M_{F_k}$ via~\eqref{eq:nudef}. Note we do not indicate through notation the dependence of $F_k$ on $R$, but this should be understood. 
    Again taking subsequences, we assume that for each $R$, as $k \to \infty$,
   \begin{itemize}
       \item
            $\nu_k^R$ converges to a limit measure $\nu_0^R$;
       \item
           $T_k^R$, the integral of the inducing time, converges to some $T_L^R$;
       \item
           $h(\nu^R_k)$ converges to a limit $h_L^R$;
       \item
           $\lambda(\nu^R_k) := \int \log|DF_k| ~d\nu^R_k$ converges to some limit $\lambda_L^R$. 
   \end{itemize}
   By Lemma~\ref{lem:nu0inv}, $\nu_0^R \in \M_{F_0}$. 
   Set $T^R_\Delta:=T^R_L-T^R_0 \geq 0$ (by~\eqref{eq:t00}).   Note that $1 \leq T^R_k \leq 1/\eta(R_0,d)$ for each $k\geq0, R \geq R_0$, where $\eta$ is given by Lemma~\ref{lem:Rtower2}. As $T^R_0 \geq 1$, 
   \begin{equation} \label{eq:deltabound}
       T_\Delta^R \leq 1/\eta(R_0,d).
   \end{equation}

We can spread the measure $\nu_0 = \nu_0^R$ to get an  $f_0$-invariant probability measure $\mu_0^R$, defined by
   $$
   \mu_0^R = \frac{1}{T_0}\sum_n \sum_{j=0}^{n-1} f_{0*}^j (\pi_{1 *}\nu_{0|\{\tau_0 =n\}}).
   $$
   By Corollary~\ref{cor:expsch}, $\nu_0$ is hyperbolic. Hence so is $\mu_0^R$. 
   By Lemma~\ref{lem:spreadabscns},
   $\mu_0^R$ is a light limit measure (see Definition~\ref{def:llm}) for the sequence $(\mu_k)_k$. 
   Upon taking further subsequences, we similarly obtain measures $\mu_0^{R'}$ for each $R' > R$, and we can assume 
    furthermore that $E(\mu_k)$ converges to a limit $E_L$.

   Set $\lambda_\Delta^R := \lambda^R_L - \lambda_0^R$. 
   Due to non-positive Schwarzian derivative, Lemma~\ref{lem:distn} gives us a uniform lower bound on $|DF_k|$. There is a uniform upper bound on $|Df_k|$, so $\log |DF_k| \leq \tau_k C$ for some uniform constant $C >1$ (independent of $R$). 
       Hence
   \begin{equation} \label{eq:lambdabound}
      0 \leq \lambda^R_\Delta = \lim_{K\to \infty} \lim_{k\to \infty} \int_{\{\tau_k \geq K\}} \log |DF_k| ~d\nu^R_k \leq T^R_\Delta \log C.
   \end{equation}
    We shall use \eqref{eq:Abe} and~\eqref{eq:AbeDf} repeatedly in the following lemmas.

    \begin{lemma}\label{lem:lim0}
   If $\lim_{R\to \infty}T^R_\Delta = 0$, then $\mu_0^R$ converges strongly 
   to $\mu_\infty$ as $R \to \infty$. 
   Moreover,
        $\lambda(\mu_\infty) = \lim_{k \to \infty} \lambda(\mu_k)$ and
       $$
       E(\mu_\infty)  \geq E_L.$$
   \end{lemma}
   \begin{proof}
       The first statement follows immediately from Lemma~\ref{lem:spreadabscns}.
       From~\eqref{eq:lambdabound}, if $\lim_{R\to\infty}T_\Delta^R = 0$ then $\lim_{R \to \infty} \lambda^R_\Delta = 0$ and  
     \begin{align*} 
         \int \log|Df_0| ~d\mu_\infty  & = \lim_{R\to\infty} \int \log |Df_0| ~d\mu_0^R \\
                                       & = \lim_{R\to\infty} \frac{\lambda^R_0}{T_0^R} \\
         & = \lim_{R\to \infty}\lim_{k\to \infty} \frac{\lambda_k^R}{T^R_L} \\
         & = \lim_{k\to\infty} \int \log|Df_k| ~d\mu_k.
 \end{align*}
 Thus 
 $E(\mu_\infty) - E_L = h(\mu_\infty) - \lim_{k \to \infty} h(\mu_k)$. 
 By 
       Theorem~\ref{thm:introusc}, 
 entropy is upper-semicontinuous, we deduce that $E(\mu_\infty) \geq E_L$.
    \end{proof}
    
    Recall from~\eqref{eq:hdeltadef} that $h_{\Delta}^R = h^R_L - h(\nu_0^R).$ 
   If $T_\Delta = T_\Delta^R \ne 0$ then any jump in free energy can be written in a suggestive form. We suppress $R$ from notation, where appropriate. 
   \begin{lemma} \label{lem:TDne0}
       Suppose $T_\Delta \ne 0$. 
       Then
       $$
       \frac{T_0}{T_\Delta} \left(E_L - E(\mu_0)\right) = \frac{h_\Delta - t\lambda_\Delta}{T_\Delta} - E_L.$$
       \end{lemma}
       \begin{proof}
     From
     $$E_L = \lim_{k\to \infty} \frac{h(\nu_k) - t\lambda(\nu_k) }{T_k} = \frac{h_L - t \lambda_L }{T_L},$$
     one deduces (to verify, substitute $T_L-T_0$ for $T_\Delta$) that 
     $$E_L = \frac{h_L - t\lambda_L - T_\Delta E_L}{T_0}.$$ 
     Consequently, knowing $E(\mu_0) = \frac{h(\nu_0) - t\lambda(\nu_0)}{T_0}$,
    \begin{equation}\label{eq:limits23}
    E_L - E(\mu_0) 
    = \frac {1}{T_0}\left( (h_\Delta -t \lambda_\Delta) - T_\Delta E_L \right).
    \end{equation}
    Multiply both sides by $T_0/T_\Delta$ to conclude. 
    \end{proof}

       Let $T^\infty_\Delta := \liminf_{R\to\infty} T^R_\Delta \geq0$.\index{Tinfinity@ $T^\infty_\Delta = \liminf_{R\to\infty} T^R_\Delta$}  Note this is finite by \eqref{eq:deltabound}.
    \begin{lemma}  \label{lem:limplus}
       If $T^\infty_\Delta > 0$ then  there exists $R_0, \theta >0$ such that for all $R \geq R_0$, 
       the light limit measure $\mu_0^R$ has  entropy  $h(\mu_0^{R})  \geq \lim_{k\to \infty} h(\mu_k) + \theta$.
   \end{lemma}
   \begin{proof}
       For all  $R$ large enough,  $T^R_\Delta >  T^\infty_\Delta/2$ and 
       $$
       h(\mu_k) \frac{T_\Delta^\infty}{2T^R_0} - 6\eps(R) \frac{T_k^R}{T^R_0} > \theta
       $$ 
       for all large $k$ and some $\theta >0$. 

       Then for all large $k$, 
       \begin{align*}
           h(\mu_0^R) & = \frac{h(\nu_0^R)}{T^R_0} \\
                      & \geq \frac{h(\nu_k) - 5\eps(R) T_\Delta^R}{T_0^R} \\
                    & 
               = h(\mu_k) \frac{T_k^R}{T_0^R} - 5\eps(R) \frac{T^R_\Delta}{T_0^R} \\
               & \geq
               h(\mu_k) + h(\mu_k)\frac{T_k^R-T^R_0}{T^R_0} -5\eps(R) \frac{T^R_\Delta}{T_0^R} \\ 
               & > h(\mu_k) + h(\mu_k)\frac{T^\infty_\Delta}{2T^R_0}  -6\eps(R) \frac{T^\infty_\Delta}{T^R_0} \\
               & \geq h(\mu_k) + \theta,
           \end{align*}
           completing the proof.
       \end{proof}

    \begin{lemma} \label{lem:gammaE}
        Suppose $T_\Delta^\infty >0$ and
        let $R_0$ be given by Lemma~\ref{lem:limplus}. 
        If $E(\mu^{R_0}_0) < E_L$, either there is a hyperbolic light limit measure with free energy strictly greater than $E_L$ or there is a $\gamma >0$ such that $E(\mu_0^R) \leq E_L - \gamma$ for all $R \geq R_0$. 
   \end{lemma}
    \begin{proof}
        Suppose $E(\mu^{R_0}_0) = E_L - \gamma'$ for some $\gamma' >0$. 
        Let $R \geq R_0$ and consider the measures $\nu_k^{R}$. 
        Note that $\hat \nu_k^{R_0}, \hat \nu_k^{R}$ are constructed in the Hofbauer extension as normalised restrictions of $\hat \mu_k$. Thus for subsets of $X(R_0)$, 
        $$\nu_k^{R} = \frac{T_k^{R}}{T_k^{R_0}} \nu_k^{R_0},$$ so on subsets of $Y_0$, 
   $$
   \nu_0^{R} = \frac{T_L^{R}}{T_L^{R_0}} \nu_0^{R_0}.$$ 
   Thus we can write $\mu_0^{R} = (1/T_L^{R_0})\mu_0^{R_0} + (1-1/T_L^{R_0})\mu'$, say, where $\mu'$ is a hyperbolic (positive)
   $f_0$-invariant probability measure (not necessarily different from $\mu_0^{R_0}$). In particular, $\mu^{R}$ and $\mu'$ are light limit measures. Either   
    $E(\mu') > E_L$ or, since $E$ is linear, 
    \begin{align*}
        E(\mu^{R}_0) & \leq \frac{E_L - \gamma'}{T_L^{R_0}} + \left(1-\frac{1}{T_L^{R_0}}\right) E_L\\
                     & \leq E_L -  \gamma'/T_L^{R_0}.
    \end{align*}
    Setting $\gamma:= \gamma'/T_L^{R_0}$ completes the proof.
   \end{proof}

   \begin{lemma} \label{lem:contra}
       Suppose $E(\mu^R_0) \leq E_L - \gamma$ for all $R\geq R_0$, for some $R_0, \gamma >0$. Then for sufficiently large $R$ and $k$, there exists   $\mu \in \M_{f_k}$ supported on a periodic orbit, with $E(\mu) > 
       E_L + \eta(R_0,d) \gamma/3$.
   \end{lemma}
   \begin{proof}
       From Lemma~\ref{lem:lim0} it follows that $T_\Delta^\infty >0$. 
   Let us fix $R \geq R_0$ large enough that $\eps :=\eps(R) < \gamma \eta(R_0, d)/10$.  
   For this $R$, we have $X^k, Y^0, F_k, \tau_k, \nu_k, T_k$ for each $k$. We write $T_L, T_\Delta, h_L, \lambda_L$ for $T_L^R, T_\Delta^R, h_L^R, \lambda_L^R$. 
   Set $h_\Delta := h_L - h(\nu_0)$ and 
     $\lambda_\Delta := \lambda_L - \lambda(\nu_0)$. Let $\mu_0$ denote the projection of $\nu_0$.

     By Lemma~\ref{lem:TDne0}, 
    $$ E_L - E(\mu_0) 
    = \frac {1}{T_0}\left( (h_\Delta -t \lambda_\Delta) - T_\Delta E_L \right) \geq \gamma.
    $$
    Noting $T_0 \geq 1$,
    $$
        -t\lambda_\Delta /T_\Delta \geq E_L + \gamma/T_\Delta - h_\Delta/T_\Delta.
        $$
    By~\eqref{eq:Kac1}, $1/T_\Delta \geq \eta(R_0,d)$. 
    By Lemma~\ref{lem:hnusemi}, $h_\Delta \leq 5\eps T_\Delta$. These estimates, combined with  choice of $R$, imply
    \begin{equation}\label{eq:limits}
        -t\lambda_\Delta /T_\Delta \geq E_L + \eta(R_0,d) \gamma/2.
    \end{equation}

   As previously, let $\xi_k$ denote the collection of branches of $F_k$.  
   Let $$\xi_k^N : = \bigvee_{j=0}^{N-1} F_k^{-j} \xi_k.$$ 
   Let $\xi_k(K)$ denote the elements of $\xi_k$ on which $\tau_k \geq K$ 
        and set $$V^k_K := \bigcup_{Z \in \xi_k(K)} Z.$$
        Let $x \in V^k_K$ be a typical point for the measure $\nu_k$ (we know $T_\Delta^R$ is positive, so $V^k_K$ has positive measure), so $x$ is recurrent. 
        Now we can speed up the map $F_k$ a little, and define $G$, the first return map (for $F_k$) to $V_K^k$. 
        Let $n, w$ be given by Lemma~\ref{lem:shadow}, so for some arbitrarily large $n$, $w$ is an $n$-periodic point (for $G$) and $G^j(w), G^j(x)$ lie in the same element of $\xi_k(K)$ for $j=0, \ldots, n$. 
        We wish to estimate the Lyapunov exponent of $w$. 
        Let $e(n)$ denote the $n^\mathrm{th}$ return time of $x$ to $V^k_K$ under $F$, so $F^{e(j)}(x)$ and $G^j(w)$ are in the same element of $\xi_k(K)$, so $\tau_k(G^j(w)) = \tau_k(F^{e(j)}(x))$. 
        By Birkhoff's theorem, for large $n$, $$\frac{n}{e(n)} \leq  2\nu_k(V^k_K) \leq 2(T_L + 1)/K$$ 
        and the last term tends to $0$ as $K \to \infty$. 
        Let  $$g(y) := \sum_{i=0}^{e_1(y) -1} \tau_k(F_k^i(y)).$$
        If $\tau_k(F_k(G^j(w))) \geq K$, then $\tau_k$ and $g$ coincide at $G^j(w)$. 
        If $\tau_k(F_k(G^j(w))) < K$, then $g(G^j(w)) = \tau_k(G^j(w)) + g(F_k(G^j(w)))$. Consequently, by Lemma~\ref{lem:shadow}, 
        $$  
        K \leq 
        \tau_k(G^j(w)) \leq 
        g(G^j(w)) \leq 
        \tau_k(G^j(w)) + M.$$ 

        We then obtain the following estimates, using $\chi$ to denote the indicator function of a set. Recall that $\lim_{K\to \infty}\lim_{k\to\infty} \int_{V^k_K} \tau_k ~d\nu_k =  T_\Delta.$ 
        In what follows, $\kappa(K,k,n)$ may vary from line to line, but it will satisfy $$\lim_{K\to \infty} \lim_{k \to \infty} \lim_{n\to \infty} \kappa(K,k,n) = 0.$$  
        \begin{equation*}
            \begin{split}
            \sum_{j=0}^{n-1} g(G^j(w))
            & = (1 + \kappa(K,k,n)) \sum_{j=0}^{n-1} \tau_k(G^j(w)) \\
            & =  (1 + \kappa(K,k,n)) \sum_{j=0}^{e(n)-1} \chi_{V^k_K}(F_k^{j}(x)) \tau_k(F_k^j(x))\\
            & = e(n) \left( \kappa(K,k,n) + \int_{V^k_K} \tau_k ~d\nu_k \right) \\
            & = e(n)\left( \kappa(K,k,n) + T_\Delta \right).
        \end{split}
    \end{equation*}
                
    To estimate the derivative, note that distortion is uniformly bounded and that the derivative  corresponding to quick words is uniformly bounded below away from zero (by Lemma~\ref{lem:distn}), while it is bounded above by $\sup_k\sup|Df_k|^M$. Thus
        \begin{equation*}
            \begin{split}
            \log |DG^n(w)|  
            & = (1 + \kappa(K,k,n)) \sum_{j=0}^{e(n)-1} \chi_{V^k_K}(F_k^j(x)) \log |DF_k(F_k^{j}(x))| \\
            & =  e(n) \left( \kappa(K,k,n) + \int_{V^k_K} \log |DF_k| ~d\nu_k \right) \\
            & =  e(n)\left( \kappa(K,k,n) + \lambda_\Delta \right).
        \end{split}
    \end{equation*}
        
    Dividing, the Lyapunov exponent for the $f_k$-invariant equidistribution, denoted $\mu_w$ say, along the periodic orbit of $\pi_1 w \in I$ is $\lambda_\Delta /T_\Delta + \kappa(K,k,n)$. 
               The free energy of the equidistribution satisfies,
               using \eqref{eq:limits} and~\eqref{eq:hT1},
               \begin{align*}
                   E(\mu_w) & = -t\lambda_\Delta /T_\Delta + \kappa(K,k,n) \\
                            & \geq E_L + \eta(R_0,d) \gamma/2 + \kappa(K,k,n) \\
                            & > E_L + \eta(R_0,d) \gamma/3.
               \end{align*}
               Thus for each $k$ large enough, there is a $f_k$-invariant measure $\mu^k$ ($=\mu_w$, for some $K,k,n$), supported on a periodic orbit, with $E(f_k, \mu^k) > E_L + \eta(R_0,d) \gamma/3$. 
               This completes the proof of Lemma~\ref{lem:contra}.
           \end{proof}
          Let us now complete the proof of Proposition~\ref{prop:SSSS}.
               In the case $T^\infty_\Delta = 0$, just apply Lemma~\ref{lem:lim0}: the second alternative of Theorem~\ref{thm:SSSS} holds. 
               If $T^\infty_\Delta > 0$, by Lemma~\ref{lem:limplus}, the light limit measure $\mu^{R_0}_0$ has entropy strictly greater than $\lim_{k\to \infty} h(\mu_k)$. If this measure has free energy at least $E_L$, the third alternative of the proposition holds, or if some other hyperbolic light limit measure has free energy $> E_L$ the first alternative of the proposition holds. 
               If these alternatives do not hold, Lemma~\ref{lem:gammaE} gives a $\gamma > 0$ for which $E(\mu^R_0) \leq E_L - \gamma$ for all $R \geq R_0$. We then derive a contradiction via 
               Lemma~\ref{lem:contra} and 
               the hypothesis  $\liminf_{k\to \infty} E_+(f_k) \leq E_L$.  
               \end{proof}

           \chapter{Katok theory, pressure and exponential tails} \label{chap:Kat}
           \section{Katok theory} \label{sec:katok}
               An idea due, we believe, to Katok is that one can approximate measures with non-zero Lyapunov exponents by measures on hyperbolic sets with similar entropy and exponents, see~\cite[\S11.6]{PrzUrb10} for this in the complex dynamics setting.
               
               \begin{lemma}\label{lem:Katok1}
                   Given $f\in\F\in\F_\mathrm{NSD}$, suppose $\mu \in \M_f$ has positive entropy and (therefore) positive Lyapunov exponent. Then there exists a sequence $(\Lambda_n)_{n\geq1}$ with the following properties. 
                   Each $\Lambda_n$ is a forward-invariant, compact, hyperbolic subset of $I$ containing periodic points. 
                   If $(\mu_n)_{n\geq1}$ is any sequence of measures $\mu_n \in \M_f$ with support $\Lambda_n$, then 
                       as $n \to \infty$, 
                       \[
                           \mu_n  \to \mu \quad \text{ and } 
                       \lambda(\mu_n)  \to \lambda(\mu). 
                       \]
                   Moreover, if one takes as $\mu_n$ the measure of maximal entropy for $f_{|\Lambda_n}$, then $h(\mu_n) \to h(\mu)$. 
               \end{lemma}
               
                   \begin{proof}
               To show convergence of measures, let $\psi_1,\psi_2, \ldots$ be a sequence of Lipschitz continuous real-valued functions on $I$. It suffices to show that, given $N\geq 1$ and $\alpha>0$, Birkhoff averages of $\psi_i$ along orbits in $\Lambda_n$ approximate $\int \psi_i \,d\mu$ to within $\alpha$ for all $i \leq N$ and all large $n$. 

                       For some $R\geq 1$, there exists, in the Hofbauer extension, a non-boundary element  $\hat Z$ of $\hat\P_R$  with $\theta := \hat \mu(\hat Z) > 0$. 

                       Let $\hat G$ denote the first return map to $\hat Z$ with $n^\mathrm{th}$ return time $e_n$. By Lemma~\ref{lem:expandingscheme}, $|D\hat G^n| > 2$ for all large $n$. The map $\hat G$ has inducing time $\tau$ with $\int_{\hat Z} \tau ~d\hat \mu = 1$,  
               while the probability measure $\hat \nu := \theta^{-1}\hat \mu  $ 
               is $\hat G$-invariant with entropy $h(\hat \nu) = h(\mu)/\theta$ (recall~\eqref{eq:AbeDf}). 

               Let $\Q = \Q_1$ denote the branch partition of $\hat G$, that is, the collection of connected components of the domain of $\hat G$.  It is generating and has entropy equal to $h(\hat \nu)$. Denote by $\Q_n$ the join $\bigvee_{j=1}^n \hat G^{-j+1} \Q_1$ and by $\Q_n[\hat x]$ the element of $\Q_n$ containing $\hat x$. 

               For almost every $\hat x$ in $\hat Z$, as $n \to \infty$, by Shannon-MacMillan-Breiman and Birkhoff's theorems, for $1 \leq i \leq N$,
               \begin{align*}
                   \frac{-1}{n} \log \hat \nu(\Q_n[\hat x]) & \to h(\hat \nu),\\
                   e_n(\hat x)/n & \to \theta^{-1},\\
               \frac{1}{n} \log |D\hat G^n(\hat x)| & \to \lambda(\mu)/\theta,\\
               \lim_{k\to\infty} \frac1k \sum_{j=0}^{k-1} \psi_i \circ \pi \circ \hat f^j(\hat x) &= \int \psi_i \, d\mu.
               \end{align*}
           Since a set of measure $\alpha$, say, cannot be covered by fewer than  $\alpha\eps^{-1}$ sets of measure $\leq \eps$, we obtain the following, using little-$o$ notation. 

           There is a finite collection $\Q_n'$ of $\exp(n h(\hat \nu)(1 + o(1)))$ elements of $\Q_n$ on which
               \begin{eqnarray}
                   \label{eqn:Qn1}
                   e_n &  =& \frac n\theta (1+o(1)),\\ 
                   \label{eqn:Qn2}
                   |D\hat G^n| & =& \exp\left(\frac{n}{\theta} \lambda(\mu)(1+o(1)) \right),\\ 
                  \frac1{e_n} \sum_{j=0}^{e_n-1} \psi_i \circ \pi \circ \hat f^j &=& (1+o(1))\int \psi_i \, d\mu. \label{eqn:Qn3}
               \end{eqnarray}
               Note that to obtain uniformity,~\eqref{eqn:Qn3} relies on some iterate of $G$ being expanding and on Lipschitz continuity of the functions $\psi_i$. 

               Let $\hat Y_n := \cup_{\hat V \in \Q_n'} \hat V$.
               Set 
               $$\hat \Lambda'_n := \bigcap_{k\geq0} \overline{\hat G^{-kn}(\hat Y_n)} = \bigcap_{k\geq 0} \hat G^{-kn}(\hat Y_n) ,$$
               the latter equality holding because $\overline{\hat Y_n} \subset \hat Z$ and $\hat G^n$ maps components of $\hat Y_n$ onto $\hat Z$. 
               Then all iterates of $\hat G^n$ are defined on $\hat \Lambda'_n$ and $\hat \Lambda'_n$ is a forward-invariant set under $\hat G_n$, while $e_n$ is bounded. Since we only selected a finite number of branches, $\hat \Lambda'_n$ is compact. Thus
               $$\Lambda_n := \bigcup_{k\geq 0} f^k(\pi (\Lambda'_n))$$
               is a forward-invariant, compact hyperbolic  subset of $I$ 
               (also known as a  \emph{conformal expanding repeller}).
               The fixed points of $\hat G^n$ in $\hat \Lambda_n'$ correspond to periodic points for $f$ in $\Lambda_n$. 

               Every point in $\Lambda_n$ enters $\pi (\Lambda'_n)$ infinitely often. Thus for all $x \in \Lambda_n$, by~\eqref{eqn:Qn2}-\eqref{eqn:Qn3}, for $k \geq n^2$, say,
               \begin{eqnarray*}
                    \frac1k \log |Df^k(x)| &= (1+o(1)) \lambda(\mu), \\
                  \frac1k \sum_{j=0}^{k-1} \psi_i \circ f^j(x) &= (1+o(1))\int \psi_i \, d\mu. 
               \end{eqnarray*}
               In particular, any measure 
                $\mu_n$  supported on $\Lambda_n$  has Lyapunov exponent $$\lambda(\mu)(1+o(1)),$$ and for any sequence  $(\Lambda_n, \mu_n)_n$, $\mu_n$ converges to $\mu$. 

               The topological entropy of  $\hat G^n : \hat Y_n \to \hat Z$ is 
               $$\log \# Q_n' = n h(\hat \nu)(1 + o(1)).$$
               The average inducing time for any measure $\nu_n$ on $\hat \Lambda'_n$ is 
               $$T_n := \int_{\hat Y_n} e_n ~d\nu_n = \frac{n}{\theta}(1+o(1)). $$
               Consequently (use \eqref{eq:Abe} in both directions),  the measure of maximal entropy for $f$ restricted to $\Lambda_n$ has entropy
               $$n h(\hat \nu)(1+o(1)) / T_n = \theta h(\hat \nu)(1+o(1))= h(\mu)(1+o(1))$$
               as required. 
           \end{proof}

               \begin{corollary}\label{cor:Katok2}
                   Suppose $f\in\F\in\F_\mathrm{NSD}$,  $\eps >0$, and $\mu \in \M_f$ has positive entropy. For each $g$ sufficiently close to $f$ in $\F$, there is a measure $\mu_g \in \M_g$ with 
                   $$|h(\mu_g) - h(\mu)| < \eps$$ 
                   and 
                   $$|\lambda(\mu_g) - \lambda(\mu)| < \eps.$$ 
                   In particular, $\liminf_{g\to f} P(g, -t\log|Dg|) \geq E(f,\mu, -t\log|Df|)$. 
               \end{corollary}
               \begin{proof}
 By Lemma~\ref{lem:Katok1}, there is a measure $\nu \in \M_{f}$  supported on a compact hyperbolic set $\Lambda$ with 
                   $|h(\nu) - h(\mu)| < \eps/2$ 
                   and 
                   $|\lambda(\nu) - \lambda(\mu)| < \eps/2.$ 
                   For $g$ sufficiently $C^1$-close to $f_0$ (on a neighbourhood of $\Lambda$),  there is a continuous conjugacy $h_g$ (close to the identity) on a neighbourhood of $\Lambda$ with $g = h_g \circ f \circ h_g^{-1}$, so $h_{g*}\nu \in \M_g$, and the entropy and Lyapunov exponent of $h_{g*}\nu$ approximate those of $\nu$ to within $\eps/2$.  To give more details for the Lyapunov exponent, 
                    \begin{align*}
\lambda(h_{g*}\nu) &= \int \log |Dg| d(h_{g*}\nu) \sim \int\log|Dg\circ h| d(h_{g*}\nu)\\
&= \int\log |Dg|d\nu \sim\int \log |Df|d\nu = \lambda(\nu).
\end{align*}
where the first $\sim$ is because $h_g$ can be taken to be arbitrary $C^0$ close
to identity by taking $g$ $C^1$ close to $f$.  A similar argument follows for entropy so the result follows. 
    \end{proof}

           \section{Ancillary proofs} \label{sec:ancillary}
                      \begin{proof}[Proof of Corollary~\ref{cor:eqstate2}]
                                For $t \in (t^-, t^+)$, 
  consider a sequence of measures $\mu_k \in \wtm_f$ with $E(\mu_k) \to P(-t\log|Df|)$. By hypothesis, for large $k$, 
  $$ h(\mu_k) - t\lambda(\mu_k) \geq \frac{P(-t \log |Df|)- P^0(-t \log |Df|)}2 
  + P^0(-t\log|Df|).$$
  In Lemma~\ref{lem:Katok1}, is shown that positive entropy measures can be approximated by equidistributions on periodic orbits with similar Lyapunov exponents. 
  Hence 
  $$
  P^0(-t\log |Df|) - 
  \left(-t \lambda(\mu) \right) \geq 0
  $$ 
  for every  $\mu \in \wtm_f$ (not just those with zero entropy, as in the definition). 
  Therefore,
   $$h(\mu_k)  \geq \frac{P(-t \log |Df|)- P^0(-t \log |Df|)}2>0,$$ 
   so the measures have uniformly positive entropy. Applying Proposition~\ref{prop:SSSS},  we obtain existence of  an equilibrium measure with positive entropy. 
\end{proof}

           \begin{proof}[Proof of Proposition~\ref{prop:IomTeq}]
    Let 
    $$\underline{\lambda} := \inf_{\mu \in \wtm_f} \lambda(\mu) = \inf_{\mu \in \wtm_f, h(\mu)=0} \lambda(\mu),$$
    the equality holding by  Lemma~\ref{lem:Katok1}. For $t \geq 0$, $P^0(-t\log|Df|) = -t\underline{\lambda}$. 
    Suppose $\mu$ is an equilibrium state for $-t^+ \log |Df|$ with $h(\mu) > 0$. Then the lines $t \mapsto -t \underline{\lambda}$ and $t \mapsto h(\mu) -t \lambda(\mu)$ intersect at $(t^+, P(-t^+ \log |Df|))$ and subtend the graph of $t \mapsto P(-t \log |Df|)$, showing lack of differentiability. In the other direction, take a strictly increasing sequence of $t_k$ converging to $t^+$ and let $\mu_k$  be the corresponding equilibrium states. Non-differentiability together with convexity implies that for some $\delta >0$, $\lambda(\mu_k) \geq \underline\lambda + \delta$ for all $k$. Simple geometry then implies that $\inf_k h(\mu_k) >0$ and that $E(f, \mu_k, -t^+ \log |Df|) \to P(-t^+ \log |DF|)$. Applying Proposition~\ref{prop:SSSS} gives existence of the equilibrium state with positive entropy. The proof for $t^-$ is the same (with $t_k \searrow t^-$ and $\overline{\lambda}$, appropriately defined, in place of $\underline{\lambda}$). 
    \end{proof}

\begin{lemma} \label{lem:prescon}
    Let 
 $\varepsilon >0$ and $t \in \R$. Let 
 $(f_k)_k \in \FNSD$ be a sequence converging to $f_0$ as $k \to \infty$ with decreasing critical relations. 
 Suppose that 
    \begin{enumerate}[label=({\alph*}), itemsep=0.0mm, topsep=0.0mm, leftmargin=7mm]
        \item for each $k \geq 0$ there is an equilibrium state $\mu_k$ for the map $f_k$ and the potential $-t \log |Df_k|$;
        \item each $\mu_k$ has entropy at least $\varepsilon$.
        \end{enumerate}
        Then 
        \begin{equation} \label{eqn:pcon2}
            \lim_{k \to \infty} P(f_k, -t \log|Df_k|)  = P(f_0, -t \log |Df_0|).
        \end{equation}
    \end{lemma}
    
           \begin{proof}
               From Corollary~\ref{cor:Katok2}, 
            $$\liminf_{k \to \infty} P(f_k, -t \log|Df_k|)  \geq P(f_0, -t \log |Df_0|).$$
        
        By  Proposition~\ref{prop:SSSS}, there is a hyperbolic positive entropy light limit measure $\mu_0' \in \M_{f_0}$ with 
        $$E(\mu_0') \geq \limsup_{k \to \infty} P_{f_k}(-t \log |Df_k|).$$ By definition of pressure, 
        $$ P(f_0, -t\log|Df_0|) \geq E(\mu_0'), $$ 
        so the limit of $P(f_k, -t\log|Df_k|)$ exists and coincides with $P(f_0, -t\log|Df_0|)$. 
    \end{proof}

\begin{lemma} \label{lem:SSSScon}
    Let 
    $(f_k)_{k\ge 0} \in \FNSD$ be a sequence converging to $f_0$ as $k \to \infty$ with decreasing critical relations. Let $t_0 \in \R$. 
 Suppose  that 
\begin{enumerate}[label=({\alph*}), itemsep=0.0mm, topsep=0.0mm, leftmargin=7mm]
\item $\mu_k \to \mu_\infty$ as $k \to \infty$, where $\mu_k \in \wtm_{f_k}$ for $k \geq 1$;
\item $E_L := \lim_{k\to \infty} E(f_k, \mu_k, -t_0 \log |Df_k|) = P(-t_0 \log |Df_0|)$;
\item $E_L \geq \liminf_{k \to \infty} E_+(f_k, -t_0\log|Df_k|) $; 
\item $\limsup_{k\to\infty} h(\mu_k) >0$;
\item  the pressure functions $t \mapsto P(f_k, -t \log|Df_k|)$ converge on a neighbourhood of $t_0$ to a limit function $t \mapsto p(t)$. 
\end{enumerate}
If $\mu_\infty$ is not a convex combination of equilibrium states for the potential $-t_0 \log|Df_0|$, then some light limit measure $\mu_*$ is an equilibrium state with $$h(\mu_*) >\limsup_{k\to\infty} h(\mu_k),$$ and $p$ is not differentiable at $t_0$. 
    \end{lemma}
    \begin{proof}
        Considering the free energies of $\mu_k$, the graph of $p$ is subtended by the line passing through the points $(0,\liminf_{k\to\infty}h(\mu_k)), (t_0, p(t_0))$.

        One of the last two alternatives of Proposition~\ref{prop:SSSS} hold. If $\mu_\infty$ is not a convex combination of equilibrium states, the second alternative does not hold and the third (final) alternative must hold. Thus there is a light limit measure $\mu_*$ with $E(\mu_*) = P(f, -t_0,\log|Df_0|)$ and with entropy $h(\mu_*) >\limsup_{k\to\infty} h(\mu_k)$. 
        Via Corollary~\ref{cor:Katok2}, considering free energies of the given measures for maps close to $f_0$, one obtains that 
         the graph of $p$ is subtended by the line passing through the points $(0,h(\mu_*)), (t_0, p(t_0))$. 
        
         As a result, $p$ is not differentiable at $t_0$. 
    \end{proof}

           \begin{proof}[Proof of Theorem~\ref{thm:eqstatstab}]
               By hypothesis, 
               $$P(f_k, -t_0\log|Df_k|) \geq P^0(f_k, -t_0\log|Df_k|) + \eps$$ 
               for each $k$. Since the pressure functions are uniformly Lipschitz, we immediately deduce the existence of a neighbourhood $U$ of $t_0$ for which, for all large $k$, 
               \begin{equation}\label{eqn:peps2}
                   P(f_k, -t\log|Df_k|) \geq P^0(f_k, -t\log|Df_k|) + \eps/2
           \end{equation}
           for all $t \in U$. By Corollary~\ref{cor:eqstate2}, there is an equilibrium state $\mu_{k,t}$ for $f_k$ and the potential $-t\log|Df_k|$ (for large $k$ and for $k=0$). By~\eqref{eqn:peps2}, $h(\mu_{k,t}) \geq \eps/2$. 
           Applying Lemma~\ref{lem:prescon}, the pressures converge for $t \in U$: 
           $$\lim_{k\to\infty} P(f_k, -t\log|Df_k|) = P(f_0, -t\log|Df_0|).$$
           Since $t \mapsto P(f_0, -t\log|Df_0|)$ is assumed differentiable at $t_0$, Lemma~\ref{lem:SSSScon} implies that any limit measure of the sequence $(\mu_k)_k$ is a convex combination of equilibrium measures. 
    \end{proof}

\section{Thermodynamics for induced maps}
\label{sec:tdf2}

A strategy used to study interval maps which lack Markov structure and expansiveness is to consider induced maps, a generalisation of the first return map.  
    We shall use level-$R$ primitive induced maps $(\Y, \hat F_\Y, \tau)$ as introduced in Definition~\ref{def:primcomp}. These maps will be expanding and  Markovian, although over a countable alphabet and not full-branched.  The idea is to study (thermodynamic) properties of the induced maps and then to translate the results into the original system.  
We shall make use of thermodynamic formalism for topological Markov shifts (in Sarig's terminology) as developed by Sarig and by Mauldin and Urba\'nski.  
We then prove Theorems~\ref{thm:eqstates}--\ref{thm:ASIP DCor}.
    
\subsection{Primitive induced maps and topological Markov shifts}
    Let $(\Y, \hat F_\Y, \tau)$ be a level-$R$ primitive induced map of a  piecewise-monotone map $f$ with non-positive Schwarzian derivative. 
    The connected components of the domain of $\hat F_\Y$ are countable and can be listed $\{\hat X_i\}_{i \in \A}$, for some countable alphabet $\A$. This gives a \emph{natural coding} $H$ \index{H@$H$, the coding map on $\Lambda$} defined on the invariant set
    $$ \Lambda := \bigcap_{n \geq 0} \hat F_\Y^{-n}(\Y)$$
    by 
    $H(\hat x) = i_0i_1i_2\ldots$, where $\hat F^k_\Y(\hat x) \in \hat X_{i_k}$ for all $k \geq 0$. 
    From Lemma~\ref{lem:expandingscheme} (and Lemma~\ref{lem:primi4}\ref{enum:cc5}), it follows that $\Lambda$ is totally disconnected and the map $H$ is injective. 
    Then Lemma~\ref{lem:primi4}\ref{enum:cc3} 
   implies that $\hat F_\Y$ is transitive on $\Lambda$. 
    Set $\X := H(\Lambda)$ and denote by $\sigma$ the left-shift on $\X$. 
    On $\X$, 
    \begin{equation}\label{eq:Hconj}
     H \circ \hat F_\Y \circ H^{-1} = \sigma. 
    \end{equation}
    There is an \emph{incidence matrix} associated to the shift; in Mauldin and Urba\'nski's terminology, this matrix is irreducible, because ($\hat F_\Y$ and therefore) $\sigma$ is transitive, while verification that the matrix is \emph{finitely irreducible}
    is straightforward thanks to 
    Lemma~\ref{lem:primi4}\ref{enum:cc1}-\ref{enum:cc2}, 
    noting that $\Y$ comprises of finitely many intervals and each pair of these intervals is linked. 
    Summarising: 
    \begin{lemma}
        On the invariant set $\Lambda$, the coding map conjugates $\hat F_\Y$  to a Markov shift over a countable alphabet with a finitely irreducible incidence matrix (a \emph{finitely irreducible Markov shift}).
   \end{lemma}
  
   \subsection{Induced Markov maps} \label{sec:imm}
   The preceding construction can be made more general. Since we will need to deal with a second induced map, it is worthwhile to do so.  Let $(I,f)$ be some dynamical system defined in some metric space $I$ equipped with a Borel $\sigma$-algebra, where $f$ is a map with domain and range contained in $I$.
   We will call $(Y,F,\tau, \pi)$ an \emph{expanding induced Markov map} \index[def]{Induced Markov map} for $(I,f)$ if
   \begin{itemize}
       \item
           $Y$ is a  metric space with a finite number of connected components $Y_1, \ldots, Y_p$;
       \item
           on each connected component of $Y$, $\pi$ is a metric isomorphism onto a subset of $I$;
       \item
           the domain of definition of $F$ can be written as $\bigcup_{i\in \A} X_i$ for some  countable alphabet $\A$, where $X_i$ are pairwise-disjoint subsets of $Y$, each compactly contained in $Y$ and mapped homeomorphically by $F$ onto some $Y_j$;
       \item
           $\tau$ is constant on each $X_i$ and $F(x) = f^{\tau(x)}(\pi(x))$;
       \item
           the maximal diameter of a connected component of $F^{-n}(Y)$ tends to $0$ as $n\to \infty$.
   \end{itemize}
   We say $(Y, F, \tau, \pi)$ is \emph{full-branched} \index[def]{Full-branched} if $Y$ only has one connected component; in this case one can identify $Y$ with $\pi(Y)$. 
           Set $\Lambda(F) :=\bigcup_{n\geq0} F^{-n}(Y)$; this set is compact and totally disconnected.
   We say that $(Y, F, \tau)$ is \emph{irreducible} \index[def]{Irreducible} if $F$ is transitive on $\Lambda(F)$. 
   If $(Y,F,\tau)$ is irreducible, then the natural coding map $H$  is a conjugacy between $F$ on $\Lambda(F)$ and a finitely irreducible Markov shift $(\X,\sigma)$ over the alphabet $\A$.

   \subsection{Finitely irreducible Markov shifts}
   Let $(\X, \sigma)$ be a finitely irreducible Markov shift over a countable alphabet $\A$.
   Denote by $\P^\sigma$ the set of $1$-cylinders (a $1$-cylinder is a set of points of $\X$ with the same first symbol) and by $\P_n^\sigma$ the set of $n$-cylinders. Equip $\X$ with the standard distance $d$, so the diameter of an  $n$-cylinder is $e^{-n}$. Given $(\X,\sigma)$   and potential $\Phi:\X\to \R$, we define the  \emph{$n$-th variation} by 
\begin{equation}\label{eq:var}
V_n(\Phi)=V_n^\sigma(\Phi) := \sup_{\indcyl \in \P^\sigma_n} \sup_{x,y \in \indcyl} |\Phi(x) -
\Phi(y)|,
\index{VnP@$V_n(\Phi)$, the $n$-th variation of $\Phi$}
\end{equation}
We say that $\Phi$ is $(C,\alpha)$-\emph{locally H\"older continuous}\index[def]{Locally H\"older continuous}  (or just \emph{H\"older continuous}) of order $\alpha$ if $V_n(\Phi) \leq Ce^{-\alpha n}$ for some $C,\alpha>0$, for all $n \geq 1$. `Locally' indicates that one only compares values of the potential within $1$-cylinders, which allows for unbounded potentials. We assume throughout this section that $\Phi$ is locally H\"older continuous. 

For $A \subset \A$, let $$\X_A := \{x \in \X : x_k \in A \mbox{ for all } k \geq 0\},$$ 
which is just the subset of $\X$ consisting of symbolic sequences all of whose symbols belong to $A$. 
The notion of pressure has been extended to Markov shifts over countable alphabets.  In the current setting of finitely irreducible Markov shifts with a locally H\"older continuous potential, 
\begin{equation} \label{eq:PressPhi} P_\sigma(\Phi) = \sup \left\{h(\mu) + \int \Phi \, d\mu\right\}.
\end{equation}
Here the supremum can be taken over all measures $\mu \in \M_\sigma$ or, alternatively, over finite alphabets $A$ and  measures $\mu \in \M_\sigma$ supported on $\X_A$, see~\cite[Theorem~2.1.5]{MauUrb03} and the top of page~10 in the same (recall $\M_\sigma$ denotes the collection of ergodic $\sigma$-invariant probability measures). 
An equilibrium measure is an ergodic $\sigma$-invariant measure $\nu$ with $\int \Phi \, d\nu > -\infty$ for which $P_\sigma(\Phi) = h(\nu) + \int \Phi \, d\nu.$ 

If $$\sum_{\indcyl \in \P^\sigma} \sup_{x \in \indcyl} \exp(\Phi(x)) < \infty,$$
then we say that $\Phi$ is \emph{summable}.\index[def]{Summable (for a potential)} Note that a summable potential is necessarily bounded above (and unbounded below, if the alphabet is infinite).\footnote{In~\cite[Chapter~2]{MauUrb03} the recurring assumption that the potential $\Phi$ (there denoted $f$) is bounded ought to read that $\exp(\Phi)$ is bounded. }  

\begin{remark}\label{rem:summ}
    By~\cite[Proposition~2.1.9]{MauUrb03}, the potential is summable if and only if the pressure is finite (summability is the same as $Z_1 < \infty$, by definition). 
\end{remark}

The Birkhoff sum of $\Phi$ is 
$$
S^\sigma_n\Phi(x) := \sum_{j=0}^{n-1} \Phi \circ \sigma^j(x).$$

A measure  $\nu$ is said to be a \emph{Gibbs measure} \index[def]{Gibbs measure} for $\Phi$ 
if there exist
$K<\infty$ and an associated constant $p\in \R$ \st for all $\indcyl\in \P^F_n$, 
$$\frac1{K} \le \frac{\nu(\indcyl)}{e^{S^\sigma_n\Phi(x)-np}}\le K$$
for any $x\in \indcyl$.  Given an expanding induced Markov map $F$ and the natural coding map, if the push-forward of an $F$-invariant measure $\hat \nu$ under the coding map is a Gibbs measure, we also call $\hat \nu$ a Gibbs measure. 

\begin{lemma} \label{lem:GibbsUniq}
    If $\Phi$ is locally H\"older continuous and is summable, then there exists a unique $\sigma$-invariant Gibbs measure $\nu_\Phi$ for $\Phi$. It has constant $p = P_\sigma(\Phi) := P(\sigma, \Phi)$ equal to the pressure. If $\int \Phi\, d\nu_\phi > -\infty$ then $\nu_\Phi$ is the unique equilibrium state for the potential $\Phi$. 
\end{lemma}
 \begin{proof} The first statements follow from {\cite[Corollary~2.7.5, Proposition 2.2.2]{MauUrb03}}, while for the final one we invoke {\cite[Theorem~2.2.9]{MauUrb03}}. 
 \end{proof}

 Denote by $\K_\beta$ the set of functions $\Psi : \X \to \mathbb{C}$ which are locally H\"older continuous  (defined in the same way as for real-valued functions) of order $\beta >0$. Denote by $\K_\beta^s$ the set of functions in $\K_\beta$ whose real parts are summable, and by $L(\K_\beta)$ the space of bounded continuous operators on $\K_\beta$. 

    We define the transfer operator for  $\Phi$ 
 as
\[
(\L_{\Phi}g)(x) := \sum_{\sigma(y) = x} e^{\Phi(y)} g(y),
\]
for $g$  in $\K_\beta$, so $\L_\Phi \in L(\K_\beta)$. 
\index{Lp@$\L_{\Phi}$, transfer operator}

\begin{lemma} 
    Let $U$ be an open ball in $\mathbb{C}^2$ and let $\Psi_j \in \K_\beta$, $j = 0,1,2$. Suppose that for each $(u,v)\in U$, the potential $\Phi_{u,v} := \Psi_0 + u \Psi_1 +v\Psi_2 \in \K_\beta^s$, that is, each potential is summable. 
    Then $(u,v) \mapsto \L_{\Phi_{u,v}}$ is analytic.
\end{lemma}
\begin{proof}
    From~\cite[Corollary~2.6.10]{MauUrb03}, $(u,v) \mapsto \L_{\Phi_{u,v}}$ is separately analytic. Then a vector-valued version of Hartogs' theorem \cite[3.3.4]{ColombeauDifferentialCalc} implies analyticity on $U$. 
\end{proof}
The following proposition is a two-variable version of \cite[Theorem~2.6.12]{MauUrb03}, whose proof we follow. 
\begin{proposition} \label{prop:anal}
    Let $U$ be an open ball in $\mathbb{C}^2$ and let $\Psi_j \in \K_\beta$, $j = 0,1,2$. Suppose that for each $(u,v)\in U$, the potential $\Psi_{u,v} := \Psi_0 + u \Psi_1 +v\Psi_2 \in \K_\beta^s$. 
    If $A$ is (in the terminology of \cite{MauUrb03}) finitely primitive (in particular, if $\sigma$ is a full shift, or finitely irreducible and topologically mixing), then 
    $$(u,v) \mapsto P_\sigma(\Psi_{u,v})$$
    is analytic. 
\end{proposition}
    \begin{proof}
        By~\cite[Theorem~2.4.6]{MauUrb03} (which has `finitely primitive' as a hypothesis), $\L_{\Psi_{u,v}}$ has a simple isolated eigenvalue equal to $\exp(P(\Psi_{u,v}))$, and analytic perturbation theory (see \cite{Kat66}) then implies that this eigenvalue moves analytically with $(u,v)$, from which the result follows. 
    \end{proof}

    \subsection{Pressure above and below, I} \label{sec:pupdown1}
    As in~\S\ref{sec:imm}, let $(Y, F, \tau)$ be an induced Markov map for some dynamical system $(I,f)$. Let the coding map $H$ denote the conjugacy with $(\X, \sigma)$, the corresponding finitely irreducible Markov shift. 
    Suppose $\hat \nu \in \M_{F}$ spreads to $\mu \in \M_f$ and let $\nu$ denote $H_*\hat \nu \in \M_\sigma$. 
    Given $\phi :I \to \R$  an integrable potential function, its corresponding \emph{induced potential} $\Phi : \X \to \R$ is
   $$\Phi := S_\tau \phi \circ \pi \circ H^{-1}:= \sum_{j=0}^{n-1} \phi \circ f^j \circ \pi \circ H^{-1}.$$
   Write $T:= \int_Y \tau \, d\hat \nu$. The entropies satisfy Abramov's formula~\eqref{eq:Abe}, so $h(\mu) = T^{-1}h(\nu)$. 
    By ergodicity, $\Phi$ satisfies (similarly to~\eqref{eq:AbeDf}) 
    \begin{equation} \label{eq:AbePhi}
        \int_I \phi\, d\mu  = \frac{1}{T} \int_\X \Phi \, d\nu.
    \end{equation}
    \begin{lemma}\label{lem:freeupdown1}
    The free energies satisfy
    $$
    E(f, \phi,\mu) = \frac{1}{T} E(\sigma, \Phi, \nu).$$
    In particular, $E(f, \phi, \mu)$ has the same sign (positive, negative or zero) as $E(\sigma, \Phi, \nu)$. 
    \end{lemma}
    \begin{proof}
        This follows immediately from Abramov's formula and~\eqref{eq:AbePhi}.
    \end{proof}

   \begin{lemma}[{\cite[Lemma~4.1]{IomTeq}}] \label{lem:Pneg1}
       Suppose $P(f, \phi) = 0$ and $\Phi$ is locally H\"older continuous. Then
        $P_\sigma(\Phi) \leq 0$. 
    \end{lemma}
    \begin{proof}
    By~\eqref{eq:PressPhi}, if $P_\sigma(\Phi) > 0$, then there is a measure $\nu \in \M_\sigma$ with $E(\sigma, \Phi,\nu) >0$
    and supported on $\X_A$ for some finite alphabet $A$. 
    On $H^{-1}(\X_A)$, $\tau$ is bounded, so $\nu$ spreads to a measure $\mu \in \M_f$. By Lemma~\ref{lem:freeupdown1}, $E(\mu, \phi)$ is strictly positive,  
    contradicting the definition of pressure as the supremum of free energies. 
    \end{proof}

    \subsection{Pressure above and below, II} \label{sec:pupdown}
    Let us return to our specific maps and potentials. 
    Let $(\Y, \hat F_\Y, \tau)$ be a level-$R$ primitive induced map of a  $d$-branched piecewise-monotone map $f$ with non-positive Schwarzian derivative. Let $(\X, \sigma)$ be the corresponding finitely irreducible Markov shift and $H$ the conjugacy between the two maps from~\eqref{eq:Hconj}.  
    We consider the potential 
    $$\phi_{t,v} := -t \log |Df| - v.$$ 
    The corresponding induced potential on $\X$
    \begin{equation} \label{eqn:phitvdef}
    \Phi_{t,v} := S_\tau \phi_{t,v} \circ \pi \circ H^{-1}
\end{equation}
    satisfies
    $$
    \Phi_{t,v} \circ H= 
    -t\log |D \hat F_\Y| - v\tau.$$
From Lemma~\ref{lem:locHol}, $\Phi_{t,v}$ is locally H\"older continuous (note $\tau$ is constant on cylinders, so $V_n(\tau)=0$). 

    Suppose that $\hat \nu \in \M_{\hat F_\Y}$ spreads to $\mu \in \M_f$ and let $\nu$ denote $H_*\hat \nu \in \M_\sigma$. 
    From Lemma~\ref{lem:freeupdown1}, 
    $E(f, \phi_{t,v}, \mu)$ has the same sign (positive, negative or zero) as $E(\sigma, \Phi_{t,v}, \nu)$.

    Let us set $p(t) := P(-t\log|Df|),$ which forces $P(\phi_{t,p(t)}) = 0.$
        Combining this with Lemma~\ref{lem:Pneg1}, we obtain:
    \begin{lemma} \label{lem:Pneg}
        $P_\sigma(\Phi_{t, p(t)}) \leq 0$. 
    \end{lemma}
    
    \begin{lemma} \label{lem:eqlift}
        Suppose that $\mu\in \M_f$ is an equilibrium measure for a potential $-t\log|Df|$ with $\hat \mu(\Y) >0$. Then $P_\sigma(\Phi_{t,p(t)}) = 0.$ The potential~$\Phi_{t,p(t)}$ is summable.
        The corresponding induced measure in $\M_\sigma$
        \begin{equation}\label{eq:nuind}
            \nu = H_* \frac{1}{\hat \mu(\Y)} \hat \mu_{|\Y}
        \end{equation}
        is the unique equilibrium measure for $\Phi_{t,p(t)}$ and is a Gibbs measure. 
    \end{lemma}
    \begin{proof}
        We have $E(f, \phi_{t, p(t)}, \mu) = 0$, hence (by Lemma~\ref{lem:freeupdown1}) the corresponding measure in $\M_\sigma$ has zero free energy for the potential $\Phi_{t,p(t)}$. Therefore $P_\sigma(\Phi_{t,p(t)}) \geq 0$ which, combined with Lemma~\ref{lem:Pneg}, gives the first statement. Since the pressure is finite (equal to $0$), the potential is summable by Remark~\ref{rem:summ}. The remaining statement now follows from Lemma~\ref{lem:GibbsUniq} and~\eqref{eq:PressPhi}.
    \end{proof}

    \subsection{Bounds on the number of equilibrium states}
    In this section, we first show that the number of equilirbium states with a given entropy is bounded.
           \begin{proof} [Proof of Theorem~\ref{thm:eqstates}]
               Let $d\geq 2, \eps >0$. Take $R$ sufficiently large that $2\eps(R,d) < \eps.$ By Remark~\ref{rem:prim}, there are at most $(2dR)^2d^R$  distinct level-$R$ primitive induced maps. If $f$ is transitive on $J(f)$, Lemma~\ref{lem:primi3} says there is only one level-$R$ primitive induced map. 
    There is at most one equilibrium state which lifts to a given primitive map, by Lemma~\ref{lem:eqlift}, and each equilibrium state with entropy at least $2\eps(R)$ does lift to such a map, by Corollary~\ref{cor:Rtowerprim}. Thus the number of level-$R$ primitive induced maps bounds the number of equilibrium states with entropy at least $\eps$, completing the proof. 
\end{proof}
This allows us to complete the proof of our main technical result. 
           \begin{proof} [Proof of Theorem~\ref{thm:SSSS}]
               Most of Theorem~\ref{thm:SSSS} follows from Proposition~\ref{prop:SSSS}, on taking appropriate subsequences. Only the final statement of Theorem~\ref{thm:SSSS} then needs verifying. From the final statement of Proposition~\ref{prop:SSSS},  if we assume that $E_L = P(-t_0 \log |Df_0| )$, then there is a hyperbolic light limit measure $\mu_*$ with strictly positive entropy which is a convex combination of equilibrium states. 
               Suppose the convex combination is described by a measure $\sigma$ on the space $\M_{f_0}$. If $$A = \{ \mu \in \M_{f_0} : h(\mu) \geq h(\mu_*)/2\},$$
               then $\sigma(A) > 0$ and 
               $$\frac{1}{\sigma(A)} \mu_*' = \int_{A} \mu \, d\sigma(\mu)$$
               is a light limit measure which is a convex combination of equilibrium states with entropy at least $h(\mu_*)/2$, of which there are only a finite number by Theorem~\ref{thm:eqstates}. Hence one of these equilibrium states is a light limit measure (absolutely continuous with respect to $\mu_*$ and necessarily hyperbolic). 
\end{proof}

    \subsection{Exponential tails} 
    We show existence of a full-branched map with exponential tails. 
    Let $f$ be a $d$-branched piecewise-monotone map with non-positive Schwarzian derivative and 
    set $p(t) := P(-t\log|Df|).$ 

        \begin{theorem}\label{thm:fullbranch}
            Given $\theta >0$, there is a finite collection ${\mathcal{K}}$ of full-branched expanding induced Markov maps for $f$ satisfying the following. 
        Suppose 
        $$p(t)  -P^0(-t\log|Df|) >2\theta,$$ 
        so $f$ has a positive-entropy equilibrium state $\mu$ for the potential $-t\log|Df|$. 
        For some 
        $(Y_0, G, \tau_0, \pi) \in \mathcal{K}$, the measure $\mu$ lifts to a measure $\hat \nu_0 \in \M_G$, and  
        \begin{itemize}
            \item
                $\hat \nu$ is a Gibbs measure for the potential $S_{\tau_0} \phi_{t,p(t)} \circ \pi$ with $p_{\hat\nu}= 0$;
            \item
                $\log$ of the Jacobian of $\hat \nu_0$ is locally H\"older;
            \item
                the map is $\kappa$-extensible for some $\kappa>0$: if $\tau_0 = n$ on a $1$-cylinder $V$ of $G$, then $f^n$ maps a neighbourhood of $\pi(V)$ onto the $\kappa$-neighbourhood of $\pi(Y_0)$;
            \item
                $|DG| >2$;
            \item
        there exists constants $C_2, \theta_0 > 0$ for which 
        $$
          \hat \nu_0\left( \tau_0^{-1}(n)\right) \leq C_2 \exp(-n\theta_0)
        $$
        for all $n\geq1$. 
    \end{itemize}
    Moreover, if $f$ is transitive on $J(f)$, one can take $\# \mathcal{K} = 1.$
\end{theorem}
    \begin{proof}
    Given $\theta >0$, 
    take $R$ large enough that $\eps(R) < \theta$. 
    There are only finitely many level-$R$ primitive induced maps, and only one when $f$ is transitive. 
    For each such induced map $(\Y, \hat F_\Y, \tau)$, for $K$ large enough $|D\hat F_\Y^n| >2$ for all $n \geq K$. Select some $K$-cylinder $Y_0$ of $\hat F_\Y$. The first return map $G$ to $Y_0$ is full-branched, expanding and extensible. 
    Note that $G$ is the first return map for $\hat f$ and for $\hat F_\Y$; we choose the return time $\tau_0$ to be with respect to $\hat f$). We thus obtain a finite collection $\mathcal{K}$ of maps $(Y_0, G, \tau_0, \pi)$, each one corresponding to a primitive induced map. 
    We shall obtain tail estimates for $(\Y, \hat F_\Y, \tau)$ and subsequently transfer them to the corresponding $(Y_0, G, \tau_0, \pi)$. 

    Suppose $t \in \R$ satisfies 
    \begin{equation}\label{eq:pdiff}
        p(t) - P^0(-t\log|Df|) > 2\theta.
    \end{equation}
By Corollary~\ref{cor:eqstate2}, there exists an equilibrium measure $\mu$ for the potential $-t\log|Df|$. The entropy of $\mu$ is necessarily at least $2\theta$, by~\eqref{eq:pdiff}.  
    Since $h(\mu) > 2\eps(R)$, $\mu(\Y) >0$ for some level-$R$ primitive induced map 
    $(\Y, \hat F_\Y, \tau)$. Obtain the conjugacy $H$ to $(\X,\sigma)$ as before. 
    Let $\phi_{t,v}, \Phi_{t,v}$ be defined as per~\eqref{eqn:phitvdef}. 
            Let $$\hat \nu :=  \frac{1}{\hat \mu(\Y)} \hat \mu_{|\Y}$$
            and set $\nu := H_*\hat \nu$, 
    so $\nu$ is the corresponding equilibrium measure $\nu$ for $\sigma$ given by~\eqref{eq:nuind}. By Lemma~\ref{lem:eqlift}, $\nu$ is a Gibbs measure for $\Phi_{t,p(t)}$. 
    By Lemma~\ref{lem:nsdbounds},  
    $$-t \log|D\hat F_\Y| - \tau p(t) \leq -2\theta \tau + C,$$ for some constant $C >0$. 
    By the Gibbs property of $\nu$, for each $1$-cylinder $\indcyl \subset \Y$ of $\hat F_\Y$ with $\tau(\indcyl) = n$, 
    \begin{eqnarray*}
        \hat \nu(\indcyl)  & \leq & K \sup_{\indcyl} \exp\left(-t \log |D \hat F_\Y| -p(t) \tau \right) \\
                           &\leq & K \exp (-2\theta \tau + C).
    \end{eqnarray*}
    In Lemma~\ref{lem:primi4}, we gave a counting estimate on the number of cylinders with a given inducing time, 
    $$
        \# \{\indcyl : \tau (\indcyl) = n \} \leq C_0\exp(\eps(R)n),$$
    for some $C_0 >0$ and all $n\geq1$. 
    Combining the two estimates, we obtain 
    $$ \hat \nu( \tau^{-1}(n)) \leq K C_0 e^C \exp(-n\theta)$$
    a basic exponential tails estimate from which much follows. 
    We use this to obtain exponential tails for a full-branched induced map. 
    
    Let $V = X_{i_1i_2\ldots i_p}$ denote the set of points for which 
    $$\hat F^j_\Y(V) \subset \tau^{-1}(i_j)$$
    for $j = 1, 2, \ldots, p$. 
    If $n := \sum_{j=1}^p i_j$, then
    $$ 
    \hat \nu\left(V\right) \leq K \exp\left( -2\theta n +C\right)  C_0^p \exp\left(\eps(R)n\right)
    $$
    so 
    \begin{equation} \label{eq:numsum}
        \nu\left(V\right) \leq C_1^p \exp(-n\theta),
    \end{equation}
    for some $C_1 >1$ and all $n$. 
    
    Let $(Y_0, G, \tau_0, \pi)$ be the full-branched induced map corresponding to $\Y$. 
    By the Gibbs property, $Y_0$ has positive measure. 
    Let $\hat \nu_0 \in \M_G$ denote the normalised restriction of $\hat \nu$ to $Y_0$. 
       The Gibbs property of $\hat \nu_0$ is inherited from $\hat \nu$; alternatively, one could reproduce~\S\ref{sec:pupdown}. 
        That $\log$ of the Jacobian is locally H\"older follows from the Gibbs property and bounded distortion. 

                We first consider points which make many returns to $\Y$ before returning to $Y_0$. By Lemma~\ref{lem:primi4}\ref{enum:cc3}, there exists $N$ such that for every $1$-cylinder $\indcyl$ of $\hat F_\Y$, 
        $$\Y \subset \bigcup_{j=1}^N \hat F_\Y^j (\indcyl).$$
        Together with the Gibbs property and bounded distortion, this implies the existence of some $\beta >0$ such that, given an $n$-cylinder $\indcyl$, 
        $$\hat \nu \left(\{ \hat x \in \indcyl : \hat F^j_\Y(\hat x) \notin Y_0, \quad j=n,\ldots, n+N-1 \} \right) \leq \exp(-\beta).
        $$
        Since the complement of $\{ \hat x \in \indcyl : \hat F^j_\Y(\hat x) \notin Y_0, \quad j=n+1,\ldots, n+N \}$ is a union of $(n+N+K)$-cylinders, we can inductively obtain
        \begin{equation} \label{eq:nnnoreturns}
            \hat \nu \left(\{ \hat x \in Y_0 : \hat F^j_\Y(\hat x) \notin Y_0, \quad j=1,\ldots, n(N+K) \} \right) \leq \exp(-n\beta).
        \end{equation}

        For $T \leq n$, we can split the set $\tau_0^{-1}(n)$ in two depending on whether $\tau_0 \geq S^\sigma_T\tau$ or not (that is, whether a point making its first return to $Y_0$  has passed more or fewer than $T$ times through $\Y$).
        From~\eqref{eq:nnnoreturns} and~\eqref{eq:numsum} we deduce that
        \begin{eqnarray*}
        \hat \nu\left( \tau_0^{-1}(n)\right)
        &\leq \sum_{p=T}^n\exp( -\lfloor p/N \rfloor \beta) + \sum_{p=1}^T\sum_{i_1 + \cdots + i_p = n}  C_1^p \exp(-n\theta)\\
        &\leq C_3 \exp(-T \beta/N) + C_1^T 2^T \exp(-n\theta),
    \end{eqnarray*}
    for some $C_3 >1$. 
    If one takes $T :=\lfloor \gamma n\rfloor$ with $\gamma := \theta/2\log{C_1 2}$, each term is exponentially small in $n$. Therefore $\hat \nu_0(\tau^{-1}_0(n))$ is also exponentially small in $n$, as required. 
    \end{proof}

    \subsection{Analyticity of pressure}
\begin{proof}[Proof of Theorem~\ref{thm:anal}]
    We assume $t \in (t^-, t^+)$, so 
    $$p(u)  -P^0(-u\log|Df|) >2\theta$$ 
    for some $\theta >0$, for all $u$ in a small neighbourhood of $t$. Moreover $f$ is assumed transitive, so we have a unique system $(Y_0, G, \tau_0, \pi)$ given by Theorem~\ref{thm:fullbranch}. To each equilibrium state $\mu_u$ for $-u\log|Df|$, we obtain $\hat \nu_0^u \in \M_G$ also from the theorem. 

    The system  has an invariant set 
    $$\Lambda_0 := \bigcap_{n\geq0} G^{-n}(Y_0).$$
    We obtain a new coding map $H_0 : \Lambda_0 \to \X_0$ and conjugacy $H_0 \circ G \circ H_0^{-1} = \sigma_0$, where $(\X_0, \sigma_0)$ is the full shift. 
    Let $\nu_0^u := (H_0)_*\hat \nu_0^u$. 
    Analogously to~\eqref{eqn:phitvdef}, we set $\phi_{u,v} := -u\log|Df| -v$ and define the potential
    \begin{equation} \label{eqn:psitvdef}
        \Psi_{u,v} := S_{\tau_0} \phi_{u,v} \circ \pi \circ H_0^{-1},
    \end{equation}
    Now $\nu_0^u$ is a $\sigma_0$-invariant Gibbs measure for the potential $\Psi_{u,p(u)}$ with $p_{\nu_0^u}=0$.
    From Theorem~\ref{thm:fullbranch}, $\nu_0^t$ has exponential tails:
        There exist constants $C_2, \theta_0 > 0$ for which 
        \begin{equation}\label{eqn:fullt}
        \nu_0^t \left( \tau_0^{-1}(n)\right)  \leq C_2 \exp(-n\theta_0)
    \end{equation}
        for all $n\geq1$. 
   From the Gibbs property of the measure and~\eqref{eqn:fullt}, we recover estimates on the potential:
        For some $\theta_0 > 0$, 
        $$
        \sum_{\indcyl \in \P^{\sigma_0}, \tau_0(\indcyl)=n} \sup_{x \in \indcyl} \exp(\Psi_{t,p(t)}(x)) =O(e^{-\theta_0 n}).$$

    From this, it immediately follows that $\Psi_{t,p(t)}$ is summable and moreover that, for all $(u,v) \in \C^2$ in a complex ball $U$ containing $(t,p(t))$, $\Psi_{u,v}$ is in $\K^s_\beta$.
    Since $\sigma_0$ is a full shift, applying Proposition~\ref{prop:anal} produces: $(u,v) \mapsto P_{\sigma_0}(\Psi_{u,v})$ is analytic on $U$. 
   
    Let $U^\R := \Re(U)$ be a real ball containing $(t,p(t))$ on which $(u,v) \mapsto P_{\sigma_0}(\Psi_{u,v})$ is real-analytic. 
    For $(u,v) \in U^\R$, $\Psi_{u,v}$ is real-valued and summable, so Lemma~\ref{lem:GibbsUniq} implies the existence of a Gibbs measure $\nu_{u,v}$ for $\Psi_{u,v}$. By the variational definition of pressure~\eqref{eq:PressPhi}, consideration of the free energy of $\nu_{u,v}$
    gives, in $U^\R$, 
    \begin{equation*} 
        \frac{\partial P_{\sigma_0}(\Psi_{u,v})}{\partial v}\Big|_{u,v}=-\int\tau~d\nu.
    \end{equation*}
    In particular for each $u$, there is at most one $v$ with $P_{\sigma_0}(\Psi_{u,v}) =0.$ 

    Consideration of $\nu_0^u$  and Lemma~\ref{lem:GibbsUniq} gives  $P_{\sigma_0}(\Psi_{u,p(u)}) = 0$ on a neighbourhood of $t$. 
        By the Implicit Function Theorem, we deduce that $u \mapsto p(u)$ is analytic at $u=t$. 
        This holds for each $t \in (t^-,t^+)$, completing the proof of Theorem~\ref{thm:anal}.
    \end{proof}

        \subsection{ASIP and decay of correlations}
\begin{proof}[Proof of Theorem~\ref{thm:ASIP DCor}]
    Here we will apply the results of \cite{YoungIsrael, MelNic05} to prove Theorem~\ref{thm:ASIP DCor}. 
Let $t\in (t^-, t^+)$; there is a corresponding equilibrium state $\mu$ for the potential $-t\log|Df|$. Obtain $(Y_0, G, \tau_0, \pi)$ from Theorem~\ref{thm:fullbranch}.  It is a full-branched expanding induced Markov map with a $G$-invariant Gibbs measure $\hat \nu_0$  with exponential tails which spreads to $\mu$. In particular,  $\hat \nu_0(\tau\ge n)=O(e^{-\eta n})$.  
From the $\kappa$-extensibility of $G$ and non-positive Schwarzian derivative, 
there exists $C>0$ so that for $x, y$ in the same connected component  $V$ of the domain of $G$, and identifying $x, y$ with $\pi(x), \pi(y)$,
\begin{equation}\label{eqn:melnic3}
    |f^k(x)-f^k(x)|\le C|G^{R}(x)-G^R(y)| \text{ for } 1\le k\le R-1
\end{equation}
where $R=\tau(x)=\tau(y)$. 

We apply \cite[Theorem~2.9]{MelNic05}, noting that
conditions (1)--(5) in \cite[\S2(e)]{MelNic05} hold for our system: (1) ($G$ is full-branched); (2) ($G$ is uniformly expanding); (3) (inequality~\eqref{eqn:melnic3});
 (4)  (the log of the Jacobian of $\hat \nu_0$ is locally H\"older); 
 (5) (integrable inducing time). 
 We obtain that mean-zero H\"older observations satisfy the ASIP, as required. 
 
 The induced Markov map can be viewed as a Young tower \cite{YoungIsrael}. One applies \cite[Theorem~3]{YoungIsrael} to establish decay of correlations (\cite[\S6.4]{YoungIsrael} shows how to go back from the tower to the original system).
\end{proof}

\chapter{Instability for Collet-Eckmann maps}
\label{sec:low ent}
The class of unimodal maps with exponential growth along critical orbits is usually considered to have almost as good statistical properties as uniformly hyperbolic maps.  However, in this chapter we will prove Theorem~\ref{THM:THUN2}, showing that these properties do not extend to statistical stability. We no longer deal with potentials, freeing up the use of the notation $\phi, \psi$ for some induced maps. 

    Let $\F$ denote the set of $C^2$ unimodal maps $f : I \to I$ with non-positive Schwarzian derivative,
	$f(\partial I) \subset \partial I$ and $|Df_{|\partial I}| > 1$, where $I$ denotes the compact interval $[0,1]$. Let us endow $\F$ with the $C^0$ topology.  Each map $f$ has a unique turning (critical) point $c$ in the interior of $I$.

        In \S\ref{ss:1} we will show that, for a map in $\F$ with a parabolic orbit and a certain sequence of Misiurewicz maps converging to the parabolic map, the acips for the Misiurewicz maps converge to the equidistribution along the parabolic orbit. We require stronger control over the acips than mere existence, so we provide another proof of the existence of such measures, while keeping some control on the densities; nevertheless, our proof is a little simpler than the comparable proof in  \cite{BM:flat}. 

In \S\ref{ss:tent}, we examine conjugacies from unimodal maps to tent maps and show existence of tent maps with certain combinatorial properties. 

In the final section of this chapter, we  conclude the proof of Theorem~\ref{THM:THUN2}. In particular, we show that once the topological entropy varies in a continuous family in $\F$, starting from a non-renormalisable map, there are parabolic maps converging to that map for which the equidistributions along the  parabolic period orbits converge to a given periodic orbit of the initial map. Moreover, we show that sequences of Misiurewicz maps as considered in Section~\ref{ss:1} exist. Taking a diagonal subsequence gives a sequence of Misiurewicz maps whose acips converge to the given periodic orbit. 

	Given an interval $J$ and a differentiable map $h: J \to \R$, we say that $h$ has distortion bounded by $C$ if 
        \begin{equation} \label{eq:Cdistn}
            \sup_{x,y \in J} |Dh(x)|/|Dh(y)| \leq C.
        \end{equation}
        If $h$ has bounded distortion clearly it is a diffeomorphism. If $h$ is a diffeomorphism with \emph{non-positive Schwarzian derivative} on a larger interval $J' \supset J$, then the Koebe Principle gives an \emph{a priori}  distortion bound for $h$ on $J$ which depends solely on the relative lengths of $h(J)$ and the connected components of $h(J')\setminus h(J)$. 
        If the domain of $h$ is  instead a collection of pairwise-disjoint intervals, we will say that $h$ has distortion bounded by $C$ if (\ref{eq:Cdistn}) holds on each connected component of its domain. 
        
        We denote by $(a,b)$ the open interval with boundary points $a,b$,  regardless of whether $a < b$ or vice versa, and by $[a,b]$ the corresponding closed interval. We denote the Lebesgue measure by $m$.\index{m@$m=$Lebesgue measure}

        \section{Preliminary lemmas}
        
        The following lemmas will be useful to quantify properties of the acips we construct in the proof of 
 Proposition~\ref{prop:gn}.
	
    \begin{lemma} \label{lem:logsum}
    Let $a_k$ satisfy $0 < a_k < 1$ and $\sum_{k=1}^\infty  a_k k \leq C < \infty$.  Then $\sum_{k=1}^\infty -a_k \log a_k < C+12$.
    \end{lemma}

    \begin{proof}
Split the sum into two parts, one where $-\log a_k \leq k$, where the sum is bounded by $C$, and a second part where $a_k < \exp(-k)$. Now $- \log r \leq 2 \sqrt{1/r}$ for $r >0$, so $- a_k \log a_k \leq 2\sqrt{a_k} < 2 \exp(-k/2)$ provided $a_k < \exp(-k)$. The expression 
        $$2\sum_{k\geq 1} \exp(-k/2) < 2(1 + \exp(1/2))\sum_{k\geq 1} \exp(-k) $$
        is bounded by $12$.
    \end{proof}

\begin{lemma}\label{lem:Misi}
Let $N, \Delta, \lambda >1, \delta >0$ be given. Let $I$ be an interval of length 1. Then there exists $\beta>0$ such that the following holds. 
Let $U$ be an open subinterval of $I$ of length $\geq \delta$. Let $f: I \to I$ be a $C^2$ map such that 
	\begin{itemize}
	\item 
	$$\dist\left(U, \bigcup_{i\geq 1} f^i(\partial U)\right) \geq \delta;$$
	\item
		 $|Df^n(x)| \geq \lambda$ for all $n \geq N$ and $x$ satisfying $f^j(x) \notin U$ for $j = 0,1,\ldots, n-1$;
		 \item
		  $\inf\{|Df(x)| : x \notin U\} \geq \delta$ and $f(\partial I) \subset \partial I$; 
		\item
		given any interval $V$ and $n$ such that $V, f(V), \ldots, f^{n-1}(V)$ are each disjoint from $U$: then the distortion of $f^j$ on $V$ for $j = 1, \ldots, n$ is bounded by $\Delta$. 
	\end{itemize}

	Let $e(x)$ denote the first entry time of $x$ into $U$, that is, $e(x) = \inf\{k \geq 0 : f^k(x) \in U\}$ (if $f^k(x) \notin U$ for all $k\geq 0$, then $e(x) = \infty$). For all $n \geq 0$, 
	$$m(\{x : e(x) \geq n\}) \leq \exp(-n\beta ).$$
	\end{lemma}

	\begin{proof}
            Mimic the proof of \cite[Lemma~V.3.3]{MSbook}.
        \end{proof}
	\begin{lemma} \label{lem:expintegral}
	For each $\beta > 0$ there is a constant $C$ such that the following holds. Let $I$ be an interval and $r$ a measurable function from $I$ to $\N$ such that, for each $k \geq0$, $m(\{x\in I : r(x) = k\}) \leq \exp(-k\beta)$. 
	
	Let $J \subset I$ be a measurable set  of measure bounded by $\gamma >0$. Then $\int_J r(x) dx \leq C\gamma(C - \log\gamma)$
	\end{lemma}
	\begin{proof}
		For  integers $n$, let $S_n := \sum_{k=n}^\infty \exp(-k\beta) = \exp(-n\beta)/(1-\exp(-\beta))$. 
		By examining  the worst case scenario, it is easy to see that the integral of $r$ over any set of measure bounded by $S_n$ is bounded by $T_n := \sum_{k=n}^\infty k \exp(-k\beta)$. 

		Let $N$ be maximal such that $\gamma \leq S_N$. Then rearranging $S_{N+1} \leq \gamma \leq S_N$ gives 
		$$ 
		\frac{\log (\gamma(1-\exp(-\beta))}{-\beta} -1 \leq N \leq \frac{\log (\gamma(1-\exp(-\beta))}{-\beta} +1. $$

		Now, $\int_J r(x) dx \leq T_N$.
		We have 
		\begin{eqnarray*}
		T_N &=& \left(N + \frac{e^{-\beta}}{1-e^{-\beta}}\right) \frac{e^{-N\beta}}{1- e^{-\beta}} \\
		  &\leq& \left(\frac{\log(\gamma(1-e^{-\beta}))}{-\beta} +1 + \frac{e^{-\beta}}{1-e^{-\beta}}\right) \frac{e^\beta\gamma(1-e^{-\beta})}{1-e^{-\beta}},
		  \end{eqnarray*}
		  which can be written in the desired form for some constant $C$.
    \end{proof}
    For motivation, one can think of $h$ in the following lemma as a unimodal map $f$ restricted to a one-sided neighbourhood of the critical point, and of $r$ as the first entry time (on a neighbourhood of the critical value) to a neighbourhood of the critical point. Then this lemma will control integral of the return time on a critical neighbourhood.  
    \begin{lemma}\label{lem:intint}
    	For each pair of constants $\beta>0, C_1 >1$ there is a constant $L$ such that the following holds. 
    	Let $h$ be a $C^1$ diffeomorphism from an open interval $U$ onto an interval $W := h(U)$, where both intervals have  length bounded by 1. Suppose $|Dh| < C_1$ and  
	 $-\int_U \log|Dh(x)| dx  < C_1$.
	Let $r$ be a measurable function from $W$ to $\N$ such that, for each $k \geq0$, $m(\{x\in I : r(x) = k\}) \leq \exp(-k\beta)$. 

	Then $\int_U r(h(x)) dx < L$.
     \end{lemma}
     \begin{proof}
     	Let $C$ be given by Lemma \ref{lem:expintegral}, so if $J$ is a set with measure bounded by $\gamma >0$, then 
        \begin{equation}\label{eqn:Cgamma}
        \int_J r(x) dx \leq C\gamma(C-\log\gamma).
        \end{equation}
     	
     	For each $k \geq 0$, let $U_k$ denote the set of $x \in U$ for which $e^{-k} \geq |Dh(x)|/C_1 > e^{-(k+1)}$. Then $U = \bigsqcup_{k=0}^\infty U_k$, and
	$$
	C_1 > -\int_U \log |Dh| dx \geq -\log C_1 + \sum_{k=0}^\infty k m(U_k).
	$$
	Now, $\int_{U_k} r(h(x)) dx \leq e^{k+1} \int_{h(U_k)} r(y) dy$. 
        We have $m(h(U_k)) \leq e^{-k} m(U_k)$, so (\ref{eqn:Cgamma}) gives
	$$
		\int_{h(U_k)} r(y) dy \leq C e^{-k}m(U_k) ( C - \log(e^{-k}m(U_k))) \leq Ce^{-k} m(U_k)( C + k - \log m(U_k)).$$
		Then
		$$
		\int_{U_k} r(h(x)) dx \leq C e m(U_k)(C+k-\log m(U_k)).$$
		Summing over $k$ and using Lemma \ref{lem:logsum} with $a_k = m(U_k)$, we get,
		$$
		\int_U r(h(x)) dx \leq C e \left( m(U) C + (C_1+ \log C_1) + C_1 + 12\right). $$
		We can choose $L := Ce(C + 2C_1 + \log C_1 + 12)$. 
	\end{proof}

\section{Acips for some almost-parabolic Misiurewicz maps} \label{ss:1} 

    Let $g \in \F$ have a parabolic periodic orbit of (minimal) period $k$, and suppose that $g^k$ is orientation-preserving in a neighbourhood of the orbit. 
    Let $\alpha$ be the point of the orbit such that the critical point $c$ (of $g$) is in the immediate basin of attraction of $\alpha$ for the iterate $g^k$ (\cite[Theorem~III.6.1]{MSbook}).
    Let $c'$ be the nearest critical point of $g^k$ on the other side of $\alpha$ from $c$, so for some $j < k$, $g^j(c')=c$. Then $g^k$ is orientation-preserving on $(c',c)$. 

\begin{lemma}
	The graph of $g^k$ does not traverse the diagonal on $(c',c)$ and $\alpha$ is the unique fixed point of $g^k$ in this interval.
\end{lemma}
\begin{proof}
    If there were a second fixed point $\alpha'$ of $g^k$ in $(c',c)$, one could use the  Minimum Principle to show that it must be hyperbolic attracting. Thus it would not be in the orbit (under $g$) of $\alpha$ and $c$ would be in the basin of two different attracting orbits, which would be impossible.  It remains to rule out the case that $c$ and $c'$ are on opposite sides of the diagonal. Assume that they are on opposite sides, so $\alpha$ is attracting on both sides. Then for $n$ large, $g^{nk}((c',c))$ is a small neighbourhood of $\alpha$. But for some $0 < j < k$, $g^j(c') = c$. Thus $g^{nk+ j}(\alpha) = g^j(\alpha)$ is close to $\alpha$ and so equals $\alpha$, contradicting minimality of $k$. 
\end{proof}

    Let $q$ be a preimage of some repelling periodic point so that $g^k$
    is monotone on $(q,\alpha)$ and $g^i(q), g^i(\alpha) \notin (q,\alpha)$ for all $i\geq 1$. It follows that $\alpha \in (q,c)$, since all points in $(\alpha, c)$ are in the basin of attraction of $\alpha$.

Given a point $x \in I$ and $f \in \F$, we denote by $x^*$ the other (`symmetric') point satisfying $f(x) = f(x^*)$. For maps close enough to $g$, there are corresponding pre-periodic points to $q$ and $q^*$ and we denote these also by $q, q^*$, without indicating the dependence on the map. The critical point is denoted by $c$, again without indicating dependence on the map. 

\begin{proposition} \label{prop:gn}
Suppose there are $g_n \in \F$ converging to $g$ in the $C^0$ topology such that $g^{ik}_n(c) \in (q,c)$ for $1\leq i < n$ and   $g^{nk}_n(c) = q$, 
and such that there exists $C>0$ with  $\int \log |Dg_n(x)| dx > -C$ for all large $n$. 

Then, for large $n$, each $g_n$ has an acip   $\mu_n$ and the sequence of measures $\mu_n$  converges weakly to the equidistribution on the orbit of $\alpha$.
\end{proposition}

\begin{proof}
    Assume $n$ is large, so $g_n$ is close to $g$. 

    Define $q_i$, for $1\leq i <n$, by $q_i \in (q,c)$, $g_n^k(q_1) = q$ and $g_n^k(q_{i+1}) = q_i$. Then $g^k_n(c) = q_{n-1}$. Set $U := (q_{n-1}, q_{n-1}^*)$. Note that $\limsup_{n \to \infty} |c-q_{n-1}| >0$.

    The \emph{first entry time} $e_V(x)$\index{eV@$e_V$, first entry time} of a point $x$ to $ V := (q_2, q_2^*)$ is defined as 
    $$
    e_V(x) := \inf \{j\geq 0  : g_n^j(x) \in V\}.
    $$
    Define the \emph{first entry map} $\phi_V$\index{phiV@$\phi_V$, first entry map} to $V$ by $\phi_V(x) = g_n^{e_V(x)}(x)$. Now $V$ is a nice interval containing the unique critical point, and $q$ never returns to $(q, q_1^*)$, by choice of $q$. Then on each connected component of the domain of definition of $\phi_V$, the 
    restriction of $\phi_V$ is a diffeomorphism onto $V$ which extends to a larger domain and this larger domain gets mapped by the corresponding iterate of $g_n$ diffeomorphically onto $(q,q_1^*)$. 
        Of course, $\phi_V$ is just the identity on $V = (q_2,q_2^*)$.

    Set $W := (q_1, q_2)$ and define the corresponding first entry time $e_W$ and first entry map $\phi_W$ to $W$.
     Note that $e_W(x)= (i-1)k$ on $(q_i,q_i^*)\setminus [q_{i+1}, q_{i+1}^*]$ for $2 \leq i < n-1$.  
By non-positive Schwarzian derivative, there is a uniform bound $\Delta$, independent of large $n$, on the distortion of each of $\phi_V$, $\phi_{W|(V\setminus U)}$ and $\psi := \phi_V\circ  \phi_W$.

	By the Ma\~n\'e Hyperbolicity Theorem (\cite[Theorem~III.5.1]{MSbook}), $g$ is uniformly hyperbolic away from $V$. Therefore the same holds for $g_n$ for all large $n$: 
	 there exist $\lambda, j>1$ such that $|D(g_{n|(I\setminus V)})^j| > \lambda$.

Define $r(x)$ by $r := e_W + e_V \circ \phi_W $.\index{rew@ $r := e_W + e_V \circ \phi_W $} On $V$, $\psi(x) = g_n^{r(x)}(x)$.
By the Folklore Theorem, there is an acip $\nu_n$ for $\psi : V \to V$ ($\psi$ defined almost everywhere) and, because of the distortion bound, 
	\begin{equation} \label{eqn:measuredistn}
	\Delta^{-1} \leq \frac{d\nu_n}{dx} \cdot|V| \leq \Delta
	\end{equation}
	on $V$.
One can then define an absolutely continuous, $g_n$-invariant measure $\mu_n$ by, writing $\chi_S$ for the indicator function of a set $S$, 
$$
\mu_n(S) = \int_V \sum_{j=0}^{r(x)-1} \chi_S(g_n^j(x)) ~d\nu_n(x)
$$
for all measurable sets $S$. 
Of course, $\mu_n$ is finite if and only if the integral is finite, with $I$ in place of $S$ --- this is just $\int_V r(x) ~d\nu_n(x)$. We shall show finiteness and more now.

Our goal is to show that, for large $n$, most of $\mu_n$ is concentrated in a small neighbourhood of the orbit of $\alpha$. So let $B$ be an $\varepsilon$-neighbourhood of the orbit of $\alpha$. It suffices to prove that $\mu_n(B)$ and $\mu_n(I\setminus B)$ are finite (showing finiteness of $\mu_n(I)$) and that $\lim_{n\to \infty} \mu_n(B)/\mu_n(I \setminus B) = \infty$.

        For some $M$, for all large $n$,  $\# \{i \leq e_W(x) : g_n^{i}(x) \notin B \} < M$ for all $x \in V$. 
	Set $U_j := (q_j, q_{j+1}) \cup (q_j^*, q^*_{j+1})$ for $j = 2,\ldots, n-2$; on each $U_j$, $e_W(x) = kj$, and $V = U \cup \bigcup_{j=2}^{n-2}U_j$. 
	Then,
        $$
        \int_V e_W(x) ~d\nu_n (x) \geq \mu_n(B) = \int_V \sum_{j=0}^{r(x)-1} \chi_B(g_n^j(x)) ~d\nu_n(x) \geq \int_V (e_W(x)-M) ~d\nu_n(x).
        $$
        Using (\ref{eqn:measuredistn}) we deduce that 
        $$
            |V|\Delta k\left(n|U| + \sum_{j=2}^{n-2} |U_j|j\right) \geq \mu_n(B) \geq |V|\Delta^{-1} \left(-M|V| + k\sum_{j=2}^{n-2} |U_j|j\right).
        $$
	Clearly $\mu_n(B)$ is finite.
	Denoting by $U^n_j$ the set $U_j$ for the map $g_n$, 
	we have the following.  For each $j$, the sets $U^n_{n-j}$ converge as $n$ tends to infinity to a set $U^\infty_{-j}$, say, of definite size. Therefore, 
        the term $\sum_{j=2}^{n-2} |U^n_j| j \geq |U^n_{n-2}|(n-2) $ tends to infinity with $n$, forcing $\mu_n(B)$ along with it. 

        It remains to show that $\mu_n(I\setminus B)$ is bounded independently of $n$.  Now $\mu_n(I\setminus B) \leq \int_V(M + e_V(g^{e_W}(x))) ~d\nu_n(x)$. 
 Again using (\ref{eqn:measuredistn}),
        $$
        \mu_n(I \setminus B) \leq |V|\Delta\left(|V|M + \Delta\sum_{j=2}^{n-2} |U_j| \int_W e_V(x) dx /|W| + \int_U e_V(g_n ^{kn}(x)) dx\right).
        $$
        By Lemma \ref{lem:Misi} there is a $\theta > 0$ such that $m(\{x \in I : e_V(x) = j \}) \leq \exp(-\theta j)$, again independently of $n$, provided  $n$ is large. Therefore $\int_W e_V(x) dx$ is bounded. 
	 Now consider the term $\int_U e_V(g^{kn}_n(x)) dx$.
		On $U$, $g_n$ is a unimodal map; $g_n^{nk-1}$ maps $g_n(U)$ diffeomorphically with distortion bound $\Delta$ and derivative bounded away from $0$ and infinity independently of $n$, into $W$. On $W$, the return time $e_V$ to $V$ decays exponentially by Lemma \ref{lem:Misi}, and the constants do not depend on $n$, so we can apply Lemma \ref{lem:intint}. Note we do not apply it to $g_n^{nk}$ directly, but rather to  $h:= g_n^{nk}$ restricted to the two connected components of $U\setminus \{c\}$ one at a time. Applying the lemma gives a bound independent of $n$ on $\int_U e_V(g^{kn}_n(x)) dx$. 
		Therefore $\mu_n(I \setminus B)$ is uniformly bounded and $\mu_n(B)/\mu_n(I\setminus B)$ tends to infinity. Normalise the measure $\mu_n$ to finish. 
	\end{proof}

\section{Conjugacies to tent maps} \label{ss:tent}
   
    \begin{definition} For $s \in (1,2]$ denote by $T_s : I \to I$ the \emph{tent map}\index{Tent@ $T_s$, tent map}  with slope $\pm s$, defined  by 
        \begin{itemize}
        \item
            for $0 \leq x \leq 1/2$, $T_s(x) = sx$;
            \item
            for $1/2 \leq x \leq 1$, $T_s(x) = s - sx$.
            \end{itemize}
        \end{definition}
	The tent map $T_s$ has topological entropy equal to $\log s$ (\cite[Theorem~II.8.1]{MSbook}).

    \begin{definition}
        We say a tent map is \emph{periodic} if the orbit of the turning point is periodic. We call the tent map \emph{non-recurrent} if the turning point (1/2) is non-recurrent.
        \end{definition}
        The following is \cite[Theorem~II.8.1]{MSbook}, due to Parry and to Milnor and Thurston. Note that for unimodal maps, showing uniqueness of the semi-conjugacy is relatively straightforward.
    \begin{fact} \label{fact:semic}
        Let $f : I \to I$ be a continuous unimodal map with positive entropy which fixes the boundary. Then there exists a unique continuous semi-conjugacy $h_f$ to $T_s$, where $\log s$ is the (topological) entropy of $f$. 
    \end{fact}

    \begin{lemma} Let $f \in \F$  have entropy $\log s >0$ and let $h_f$ be the conjugacy to $T_s$. If $f$ has a periodic attractor then $T_s$ is periodic. 
        If $T_s$ is non-recurrent, the semi-conjugacy $h_f$ is a conjugacy and 
    all periodic points of $f$ are hyperbolic repelling and, if also $s > \sqrt{2}$, $f$ is non-renormalisable. 
    \end{lemma}
    \begin{proof}
        Since $f$ has non-positive Schwarzian derivative, the turning point of $f$ is contained in the immediate basin of a periodic attractor. The connected component of the immediate basin containing the critical point gets mapped by $h_f$ to the turning point of $T_s$, which is therefore periodic, showing the first statement. 

        If $T_s$ is non-recurrent then $c$ is non-recurrent, so $f$ has no periodic attractors and all periodic points are hyperbolic repelling, again by non-positive Schwarzian derivative. If $s > \sqrt{2}$ then $T_s$ is non-renormalisable. 
	Thus either $f$ is non-renormalisable or $f$ contains a restrictive interval containing $c$ which is collapsed by $h_f$. In the latter case the (periodic) restrictive interval is mapped to the turning point, contradicting non-recurrence of $T_s$. 
	Therefore $f$ is non-renormalisable. Let us show (still in the case $s > \sqrt{2}$) that $h_f$ is a conjugacy. By \cite[Proposition~III.4.3]{MSbook}, we only need to show that $f$ has no wandering intervals. 
	Let $W$ be a neighbourhood of $c$ compactly contained in a larger neighbourhood $W'$ disjoint from the post-critical orbit. If $f^k(x) \in W$ then $|Df^k(x)| > C$ for some constant $C$ which depends only on the relative lengths of $W$ and the connected components of $W' \setminus W$, using non-positive Schwarzian derivative. Thus no wandering interval can accumulate in $W$, as it cannot shrink indefinitely in size. 
	On the other hand, any wandering interval must accumulate on $c$, for example by \cite[Theorem~IV.7.1]{MSbook}. In particular, $f$ cannot have a wandering interval. This completes the case $s >\sqrt{2}$. 

        If $s\leq \sqrt{2}$, there is a least $k \geq 1$ such that $s' := 2^k \log s > 1/2$. The tent map $T_s$ is then $k$ times renormalisable of type 2. The $k$th renormalisation of $T_s$ is the map $T_{s'}$. Pulling back, $f$ is also $k$ times renormalisable of type 2, with $k$th renormalisation $f'$ say. But $h_f$ now conjugates $f'$ to $T_{s'}$, since $s' > \sqrt{2}$. It follows again that $f$ and $T_s$ are conjugate. 
    \end{proof}

    \begin{corollary}
        If $f \in \F$ has entropy $\log s> \log 2/2$ and the turning point of $T_s$ is non-recurrent, then $f$ is a non-renormalisable Misiurewicz map.
    \end{corollary}
    \begin{corollary}
        If $f \in \F$ has entropy $\log s$ and $T_s$ is non-recurrent, and if $g \in \F$, then $f$ and $g$ are conjugate if and only if the entropy of $g$ also equals $\log s$. 
    \end{corollary}
    	
	Let us continue with some further facts, deducible from or contained in \cite[Chapter~II]{MSbook}, to which we also refer the reader for definitions and conventions. Since tent maps have no periodic attractors or wandering intervals, two tent maps with the same \emph{kneading invariants} are conjugate and thus have the same entropy. Since the entropy is just logarithm
        of the slope, the maps themselves coincide. In particular, $T_s$ and $T_{s'}$ have the same kneading invariants if and only if $s=s'$. With a little more work, one can then show the following: 
        parameters $s$ for which $T_s$ is periodic are dense in $(1,2]$; parameters $s$ for which $T_s$ is non-recurrent are dense in $(1,2]$.
	Writing $\nu(s)$ for the kneading invariant of $T_s$, $\nu$ is a strictly monotone function. Moreover, \emph{at non-periodic parameters} $s$, the kneading invariant $\nu(\cdot)$ is continuous (in general it is not continuous at periodic parameters). 
        Fixing $T_s$, the itinerary of a point $x$ is denoted $\underline{i}_s(x)$, so $\nu(s) = \underline{i}_s(1/2)$. The itinerary map $\underline{i}_s$ is also monotone. It is continuous at points whose orbits do not contain $1/2$.  The itinerary of $x$ is a sequence, for $n \geq 0$, whose $n$th element specifies whether $T^n_s(x) - 1/2$ is positive, negative or zero (that is, on which side of the turning point $T^n_s(x)$ lies). Denote by $\underline{i}_s^N(x)$ the itinerary sequence truncated to its first $N+1$ elements. 

    \begin{lemma} \label{lem:follow}
        Let $\sqrt{2} < s \leq 2$ and  $p\ne 0$ be a periodic point for $T_s$ with orbit $\Orb(p)$\index{Orb@$\Orb(p)$, orbit of $p$} not containing the turning point. Then for all $\varepsilon >0$ and $\delta \ne 0$ ($\delta < 0$ if $s=2$), there is a parameter $r \in [s, s+\delta]$ for which $T_r$ has a periodic turning point with period $N$ say, and
        $$
        \# \{k \in \{1, \ldots N\} : T^k_r(1/2) \notin B(\Orb(p), \varepsilon)\} < \varepsilon N.
        $$
        \end{lemma}
    \begin{proof}
	We can suppose $T_s$ is non-recurrent because non-recurrent parameters are dense in $(1,2]$. By a simple perturbation argument, we can also assume that $p$ is not in the orbit of the turning point. 
 $T_s$ is expanding on a neighbourhood of $\Orb(p)$, so if $\varepsilon>0$ is sufficiently small,  $$B(\Orb(p), \varepsilon) \subset T_s(B(\Orb(p), \varepsilon)).$$ 
 There exist $q = q_{n_0}$ with $|q-p| < \varepsilon/2$ and some $n_0\geq 1$ such that $T_s^{n_0}(q) = 1/2$. Therefore, there exists a sequence of points $q_n \in B(\Orb(p), \varepsilon/2)$, $n \geq n_0$, for which $T_s(q_{n+1}) = q_n$  and $T_s^n(q_n) = 1/2$. 

 We can assume $\delta$ is small enough that there is a continuation $r \mapsto q_n^r$ for each $n$ and all $r \in [s, s+\delta]$ satisfying $q_n^r \in B(\Orb(p, \varepsilon))$, $T_r(q_{n+1}^r) = q_n^r$ for $n \geq n_0$, and  $T_r^{n_0}(q^r_{n_0}) = 1/2$. 

Let $W$ denote the connected component of $I\setminus \Orb(1/2)$. 
By transitivity, there are symmetric points $a, a_* \in W$, one on each side of $1/2$, and $l \geq 1$ such that $T_s^{l}(a) = T_s^{l}(a_*) = p$ and such that the orbits of $a, a_*$ avoid the neighbourhood $V:= (a,a_*)$ of $1/2$. 
We claim that there is a sequence of points $y_j$ converging to $T_s(1/2)$ in $(T_s(1/2), T_{s+\delta}(1/2))$ and numbers $k_j$ such that $T_s^{k_j}(y_j) = p $ and $\Orb(y_j) \cap V = \emptyset. $
Indeed, let $\eta_j$ be small, positive or negative, and consider the images of the interval $(T_s(1/2), T_s(1/2) + \eta_j)$. Since $T_s$ is 
expanding away from the turning point, and the orbit of $T_s(1/2)$ avoids $W$, the images grow until there is a first one which contains either $a$ or $a_*$ or both. Pull back the point to get a point $y_j$ inside $(T_s(1/2), T_s(1/2) + \eta_j)$. The claim follows.   

Combining these statements, we deduce that there are some (slightly smaller) neighbourhood $V$ of $1/2$ and, given any increasing sequence of numbers $n_j > n_0$,  a sequence of points $x_j$ (close to the $y_j$) such that $T_s^{k_j}(x_j) =  q_{n_j}$ and $T_s^l(x_j) \notin V$ for all $l < k_j + n_j$, while $T_s^{k_j+n_j}(x_j) = 1/2$. Since the $k_j$ were already defined, the sequence of numbers $n_j$ can be chosen with $n_j > k_j^2$, say, so $n_j /k_j \to \infty$. Provided the  $n_j$ are chosen large enough, the $\ul{i}_s(x_j)$ converge to $\ul{i}_s(T_s(1/2))$ in  $\left(\ul{i}_s(T_s(1/2)), \ul{i}_{s+\delta}(T_{s+\delta}(1/2))\right) $.

Therefore, since the dynamics outside of $V$ persists under small perturbation,  there exist $\gamma \in (0, \delta)$, $N \geq 1$ and $x_r$ for $r \in [s, s+\gamma]$ such that
    \begin{itemize}
        \item
            $r \mapsto x_r$ is continuous;
        \item
            $T_r(x_r) \ne 1/2$ for all $l < N$ and $r \in [s, s+\gamma]$. 
        \item
            $T_r^N(x_r) = 1/2$;
        \item
            $\underline{i}_s^N (x_s) = \underline{i}_r^N(x_r) \in \left[\underline{i}_s^N (T_s(1/2)), \underline{i}_{s+\gamma}^N (T_{s+\gamma}(1/2))\right]$; 
        \item
            $\# \{ k \leq N : T_r^k(x_r) \notin B(\Orb(p), \varepsilon) \} < \varepsilon N$.
    \end{itemize}
    In particular, for some $r \in [s, s+\gamma]$, $\ul{i}_r^N(T_r(1/2)) = \ul{i}_r^N(x_r)$, so  $T_r(1/2) = x_r$ and $T_r^{N+1}(1/2) = 1/2$. 
    \end{proof}

    \begin{lemma} \label{lem:sn}
        Let $\sqrt{2} < s < 2$ and let $T_s$ be periodic with periodic turning point of period $k$. Let $b = s/(1+s)$ be the orientation-reversing fixed point of $T_s$. There are $N >0, a$ such that $T_s^N(a) = b$ and $T_s^k$ is monotone and orientation-preserving on $[a,1/2]$. For all $r$ in some neighbourhood of $s$, the point $a$ has a continuation $a_r$ such that $T_r^N(a_r) = r/(1+r)$. 

        There exist $s_n \to s$ such that, for all large $n$, the sequence 
        $$
        1/2, T_n^k(1/2), T_n^{2k}(1/2), \ldots, T_n^{nk}(1/2) = a_n
        $$
        is monotone, where (abusing notation) $T_n$ denotes the tent map with slope $s_n$ and $a_n$ the corresponding continuation of the point $a$. 
    \end{lemma}
    \begin{proof}
         Fix a symmetric neighbourhood $U$ of $1/2$ on which the slope of $T^k_s$ only changes at $1/2$. Since $T_s$ is non-renormalisable, there are pre-images of $b$ in $U$. Let $a$ be one such pre-image, and $a_*$ its symmetric point, so for some $N>0$, $T_s^n(a) = T_s^n(a_*) = b$. We can choose $a$ such that $T_s^j(a) \notin (a, a_*)$ for $1 \leq j < N$. Then $T_s^N$ is monotone on $(a, 1/2)$ and on $(1/2, a_*)$. Thus on one of those intervals, $T_s^N$ is orientation-preserving. 

        The second statement is obvious. 

	The final statement follows from strict monotonicity of the kneading invariant for tent maps (and corresponds to the critical periodic orbit no longer existing, so the graph of $T^k_n$ no longer quite touches the diagonal near $1/2$). 
    \end{proof}

    \section{Proof of Theorem~\ref{THM:THUN2}} \label{ss:pf}
  \begin{proof}  Consider a continuous one-parameter family of maps $f_t \in \F$ (recall $\F$ is endowed with the $C^1$ topology) where $t \in [0, \theta)$, for some $\theta >0$. Let $\log s(t)$ denote the entropy of $f_t$. Suppose that $f_0$ is non-renormalisable and has all periodic points repelling, and that $t \mapsto s(t)$ is not locally-constant at $t=0$. Non-renormalisability implies $s(0) > \sqrt{2}$. Taking $\theta$ smaller if necessary, we can assume that $\sqrt{2} < s(t) \leq 2$ for all $t \in [0, \theta)$. 

        Let $p \notin \partial I$ be a periodic point of $f_0$ with orbit $\Orb(p)$ under iteration by $f_0$. 
        Fix $\eps >0$. We shall show that
        there exists $t\in (0, \theta)$ for which $f_t$ is a non-renormalisable Misiurewicz map and for which the corresponding acip $\mu_t$ satisfies  $\mu_t(B(\Orb(p)), 2\eps) > 1 - 2\eps$, which suffices [note that approximate equidistribution along the orbit follows when $\eps$ is small].

    Denote by $h_t$ the semi-conjugacy map from $f_t$ to $T_{s(t)}$.         The following lemma says that if you are close enough to the corresponding tent map periodic orbit $h_0(\Orb(p))$, then the pullback by $h_t$ for small enough $t$ is contained in an $\eps$-neighbourhood of $\Orb(p)$.
    \begin{lemma} \label{lem:gammy}
        There exists $\gamma > 0$ such that, for all $t \in [0,\gamma]$,
        $$
        h_t(B(\Orb(p), \eps)) \supset B(h_0(\Orb(p)), \gamma).
        $$
    \end{lemma}
    \begin{proof}
        The point $p$ is accumulated on both sides by repelling periodic points (note that $h_0$ is a conjugacy), and these move continuously as $t$ varies, as do their images by $h_t$. Choose repelling periodic points $q$ on each side of each point of $\Orb(p)$  and choose $\gamma'$ sufficiently small that the continuations $q_t$ of these points for $t \in [0, 2\gamma')$ lie within $B(\Orb(p), \eps)$ but $h_t(q_t) \notin h_0(\Orb(p))$. For each point $x \in h_0(\Orb(p))$, there is a $\gamma_x>0$ such that $h_t(q_t) \notin B(x, \gamma_x)$ for any $q_t$ and any $t \in [0, \gamma']$. Take $\gamma$ to be the minimum over $x \in h_0(\Orb(p)) $ of $\gamma_x$ and of $\gamma'$. 
            \end{proof}

    \begin{lemma} \label{lem:speriod}
        Let $\gamma$ be given by Lemma \ref{lem:gammy}. 

        There exists $s \in (s(0), s(\gamma))$ for which $T_s$ is periodic with turning point of period $k$ say, and 
        $$
        \# \{i \in \{1, \ldots k\} : T^i_s(1/2) \notin B(h_0(\Orb(p)), \gamma)\} < \varepsilon k.
        $$
        \end{lemma}
    \begin{proof}
	By Lemma \ref{lem:follow}.
    \end{proof}

        Let $\gamma, s$ be given by Lemmas~\ref{lem:gammy} and \ref{lem:speriod}.
        Let the sequence $s_n \to s$ be given by Lemma~\ref{lem:sn}. For large $n$, $s_n \in (s(0), s(\gamma))$. 
        Thus there are maps $g_n \in \{f_t : t \in [0,\gamma)\}$ with entropy $\log s_n$ and there is a subsequence of the $g_n$ converging to some map $g \in \F$ with entropy $\log s$. 
        
        It is straightforward to check that $g$ has an orientation-preserving parabolic periodic point $\alpha$, of period $k$ say, which gets mapped by $h_g$ to the turning point of $T_s$. We have from Lemma~\ref{lem:speriod} that  
        $$
        \# \{i \in \{1, \ldots k\} : T^i_s(1/2) \notin B(h_0(\Orb(p)), \gamma)\} < \varepsilon k,
        $$
        so by Lemma~\ref{lem:gammy}
        $$
        \# \{i \in \{1, \ldots k\} : g(\alpha) \notin B((\Orb(p)), \eps)\} < \varepsilon k.
        $$

        Applying Proposition~\ref{prop:gn}, for sufficiently large $n$, the acip $\mu_n$ for $g_n$ approximates the equidistribution on $\Orb(\alpha)$ and so satisfies
       $\mu_n(B(\Orb(p), 2\eps)) > 1 - 2\eps$.  
       Since $\theta$ can be chosen arbitrarily small and  $f_t := g_n$ is a non-renormalisable Misiurewicz map (by choice of $s_n$), this concludes the proof of Theorem~\ref{THM:THUN2}.
\end{proof}

\chapter{Positive entropy does not imply statistical stability}
\label{sec:secexo1}

In this chapter we show how lack of statistical stability may occur. The example will be explicit and shows (absolutely continuous) equilibrium states, with entropy bounded below, converging to a combination of an equilibrium state and a delta mass on a repelling fixed point. 

    We start with some general estimates for normalised quadratic maps. These estimates will transfer to any map affinely conjugate to the normalised maps.  In the second section of this chapter we define the sequence of maps and show instability, making use of the estimates in the first section. 

    Throughout this chapter, there will be topologically defined points (preimages of critical points and fixed points). Their dependence on parameters (of the maps considered) is frequently dropped from notation. 

\section{Quadratic maps which nearly have a fixed point}
We are interested in the local behaviour of a family of quadratic maps where a parabolic fixed point appears in the limit. The most convenient representation is 
 the family 
 $$P_\kappa : x \mapsto x^2 + \kappa,\index{Pkapp@$P_\kappa(x)= x^2 + \kappa$} 
 $$
 where $\kappa$ is a real parameter. Each quadratic map $x \mapsto ax^2 + bx + c$ is affinely conjugate (as a dynamical system) to some $P_\kappa$. 
 For $\kappa = 1/4$, $P_\kappa$ has a parabolic fixed point at $x_1 = 1/2$. For $\kappa > 1/4$, there is no fixed point and $P_\kappa(x) > x$ for all $x>0$, that is, all points get mapped to the right. We are interested, for now, in how long it takes for points to pass some threshold. 

    Let $M > 10$ and consider the first entry time to $[M, \infty)$\index{emk@$e_{M,\kappa}$, first entry time to $[M, \infty)$}
        $$
        e_{M,\kappa}(x) :=\inf \{k \geq 0 : P_{\kappa}^k(x) \geq M \}.$$  
        For $M' \leq M^{2^j}$, 
        \begin{equation}\label{eq:MMj}
            e_{M', \kappa} \leq e_{M,\kappa} + j. 
        \end{equation}
        Each function $e_{M,\kappa}$ is monotone-decreasing and piecewise-constant on $[0,M]$. 
        For all $x >0$, $j \geq1$ and $\kappa > 1/4$, 
        \begin{equation}\label{eq:dkap}
            \frac{dP_{\kappa}^j(x)}{d\kappa} > 0
        \end{equation}
     (visually, as one increases $\kappa$, the graph distances itself from the diagonal and images move to the right). 
     Hence the bigger $\kappa$ is, the less time it takes to enter $[M,\infty)$:  for $1/4 <  \kappa_0 < \kappa$, 
         $e_{M, \kappa} \leq e_{M, \kappa_0}$.

         Thanks to~\eqref{eq:dkap}, given $V >1$ and $n \geq 1$, there is a unique  $\kappa(V,n) > 1/4$ such that $P_{\kappa(V,n)}^n(0) = V$. 
         As a function of $V$, $\kappa(V,n)$ is strictly increasing. 
     As $n$ increases, $\kappa(V,n)$ decreases monotonically with limit  $1/4$. 
     Rather obviously, for $M \geq V$, 
     \begin{equation}\label{eqn:kvnover4}
         \int_{[0,1/4]} e_{M, \kappa(V, n)} \geq (n-1)/4,
     \end{equation}
     since $P_\kappa(0) > 1/4$. 
     
    We wish to show that the first entry time to $[M, \infty)$ does not increase too fast as $n$ increases, nor as $V$ varies.

        \begin{lemma}\label{lem:evvn}
            Given $V_0 >100$, for all $M,M',V,V' \in (V_0, 100V_0)$ and $x \in [0,M]$,
            $$e_{M',\kappa(V',n+1)}(x) \leq e_{M,\kappa(V,n)}(x)+3.$$
            \end{lemma}

            \begin{proof}
    For now, set $M = V$ and let us 
    compare $e_{V,\kappa(V,n)}$ and $ e_{V,\kappa(V,n+1)}$.  The discontinuities of the entry time functions lie along the orbits of $0$, since $P_{\kappa(V,n)}^n(0) = V$. Let 
    \begin{align*}
        p_j & := P_{\kappa(V,n)}^j(0),\\
        q_j & := P_{\kappa(V,n+1)}^j(0).
    \end{align*}
    For $j >0$, $q_j < p_j$. Since $p_n = q_{n+1} = M$ and $\kappa(V,n+1) < \kappa(V,n)$, $p_j < q_{j+1}$ for $j < n$. Since $e_{V,V,n} = n-j$ on $[p_j, p_{j+1})$ and 
        $$
        [p_j, p_{j+1}) \subset [q_j, q_j+2)$$
                we deduce 
                \begin{equation}\label{eq:evvn1}
                    e_{V,\kappa(V,n)}(x) \leq e_{{V,\kappa(V,n)}+1}(x) \leq e_{V,\kappa(V,n)}(x)+1.
            \end{equation}

            Letting $V_0 > 100$ and $V,V' \in [V_0, 100 V_0]$,~\eqref{eq:MMj} and~\eqref{eq:dkap} imply that 
            \begin{eqnarray*}
                \kappa(V,n+1) &\leq&  \kappa(\sqrt{V},n) \\
                              & \leq  &
                \kappa(V_0,n) 
                \leq \kappa(V',n).
            \end{eqnarray*}
            Thus, $e_{V,\kappa(V',n)} \leq e_{V,\kappa(V,n+1)}$ for the same range of $V,V'$, and so
            $$e_{V,\kappa(V',n)}(x) \leq e_{V,\kappa(V,n)}(x)+1.$$
                Replacing $n$ by $n+1$ and using~\eqref{eq:evvn1} yields
                \begin{equation}\label{eq:evvn2}
                    e_{V,\kappa(V',n+1)}(x) \leq e_{V,\kappa(V,n)}(x)+2.
            \end{equation}
            From~\eqref{eq:MMj} and~\eqref{eq:evvn2} we obtain
            \begin{eqnarray*}
                e_{M',\kappa(V',n+1)}(x) & \leq & e_{V,\kappa(V',n+1)}(x) + 1 \\
                                         & \leq & e_{V,\kappa(V,n)}(x) +2 
                                          \leq  e_{M, \kappa(V,n)}(x) +3
            \end{eqnarray*}
                for all $M,M',V,V' \in (V_0, 100V_0)$, as claimed.
            \end{proof}

            We also need to estimate the distortion. 
            \begin{lemma}\label{lem:ykdistn}
                Let $\kappa > 1/4$. 
                Let $y_0 \in [0,1/4)$ and set $y_k := P^k_\kappa(y_0).$ Let $k_0$ be the minimal $k \geq 1$ for which $y_{k} \geq 1$. In particular, $DP(y_{k}) < 2$ if and only if $k < k_0$.  
                First, if $k_0 \ne 1$ then 
                \begin{equation}\label{eqn:pkap}
                \frac{DP_\kappa^{k_0 - j}(z)}{DP_\kappa^{k_0-j}(z')} < 21
            \end{equation}
                for all $z,z' \in (y_{j-1}, y_j)$, for each $j = 2, 3, \ldots, k_0$. 
                Second, 
                \begin{equation}\label{eqn:pkapy}
                y_{k_0} - y_{k_0 -1} \geq 1/8. 
            \end{equation}
            \end{lemma}
            \begin{proof}
                We have $(y_1, y_{k_0-1}) \subset (1/4, 1)$, and on these intervals $1/2 < DP_\kappa < 2$ and $D^2P_\kappa = 2.$
                Hence $D\log {DP_\kappa} < 4$, and [noting $\log (Df(a)/Df(b)) = \int_b^a D\log Df$] we deduce
                the distortion bound 
                $$
                \frac{DP_\kappa^{k_0-j}(z)}{DP_\kappa^{k_0-j}(z')} \leq \exp \left(\int_{(1/4,1)} D\log DP_\kappa\right) < e^3 < 21,
        $$
         for $j = 2, 3, \ldots, k_0$ and all $z,z' \in (y_{j-1}, y_j)$. 

        To show the length estimate, since $y_{k_0}\geq 1$, we can assume $y_{k_0 -1} > \sqrt{3}/2$, for otherwise the estimate holds trivially. 
        Since $DP_\kappa(x) >1$ for $x > 1/4$,
        $$
        y_{k_0}  - y_{k_0-1} > P_\kappa(\sqrt{3}/2) - \sqrt{3}/2 > 1-\sqrt{3}/2 > 1/8.
        $$
    \end{proof}

    \section{Proof of  Theorem~\ref{THM:EXO1BIS}} \label{sec:exo1}

Theorem~\ref{THM:EXO1BIS} 
avers the existence of a sequence $(f_k)_{k\geq1}\in \FNSD$ having decreasing critical relations with $f_k \to f_0$ as $k \to \infty$ and for which each $f_k,$ $k \geq 0$ has an acip $\mu_k$ with entropy uniformly bounded away from $0$, but with $\mu_k$ converging to a convex combination of $\mu_0$ and a Dirac mass on a repelling fixed point. 

As in Theorem~\ref{THM:EXO1BIS}, consider the family of maps 
    $$f_a : [0,1] \to [0,1],$$ 
    with $a > 0$, defined by 
        $$f_a(x) = \left\lbrace \begin{split}
            1 -2x \quad &  \text{ if } \quad  x\in [0,1/2]\\
    a(x-1/2)(x-1) + 1 \quad & \text{ if } \quad x\in (1/2,1]. \end{split} \right.
            $$
            This map has two branches: one orientation-reversing, expanding  branch on $[0, 1/2]$ and  one unimodal branch  on $(1/2, 1)$ with a critical point $c = 3/4$, see Figure~\ref{fig:exo1}. 
            The map has fixed points at $1, 1/3$ and at $\alpha = \alpha(a)$, for some $\alpha \in (1/2, 1)$.  Only parameters in $(0,16)$ satisfy $f_a(c) \in (0,1)$. 
            For $a_0 = 32/3$, $f_{a_0}(c) = 1/3$, so the map is post-critically finite. Standard arguments (essentially a subset of those appearing at the end of this chapter) imply the existence of an acip for $f_{a_0}$; the acip is the unique equilibrium state for the potential $-\log |Df_{a_0}|$ and has positive entropy. 
        \begin{figure}
            \centering
        \def\svgwidth{0.9\columnwidth}
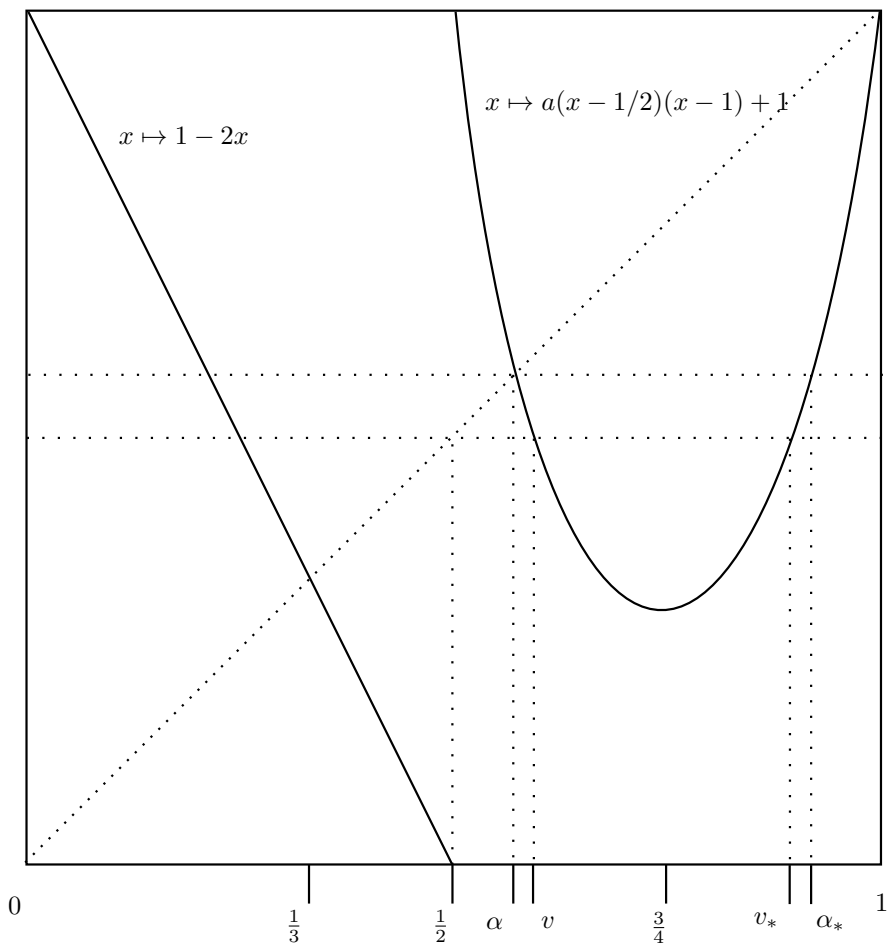
    \caption{The map $f_a$ for $a > a_0$}
    \label{fig:exo1}
\end{figure}

        The idea of the construction in Theorem~\ref{THM:EXO1BIS} is the following.  For $a$ slightly bigger than $a_0$, the critical orbit typically starts off spending a good amount of time close to $1/3$ and, after an even number of steps, escapes back to the interval $(1/2, 1]$. 
        For certain parameters, the map $f_a$ will have an acip which has a proportion of its mass very close to the repelling fixed point at $1/3$. We shall show this.

        \begin{remark} \label{rem:deccr}
            If the orbit of $c$ is disjoint from $\{0, 1/2, 1\}$, then $f_{a}$ has the same critical relations as $f_{a_0}$. 
        \end{remark}

        The following proposition implies Theorem~\ref{THM:EXO1BIS}.
        \begin{proposition} \label{prop:betadel}
            Given $\beta >0$, 
            there exists a sequence $(a_k)_{k\geq1}$ of parameters with $a_k \to a_0$ as $k \to \infty$ for which each $f_{a_k}$ has an acip $\mu_k$, the entropy of the $\mu_k$ is uniformly bounded below, and 
            $$
            \lim_{k\to \infty} \mu_k = \frac{1}{1+\beta} \mu_0 + \frac{\beta}{1+\beta} \delta_{1/3},$$
            where $\delta_{1/3}$ denotes the Dirac mass at the repelling fixed point $1/3$. Moreover the sequence $(f_{a_k})_{k \geq 0}$ has decreasing critical relations. 
        \end{proposition}

        The proof of Proposition~\ref{prop:betadel} occupies the remainder of this chapter. 
        First we construct a first return map with one unimodal branch (and lots of full branches). In particular, the unimodal branch is the connected component of the domain of the first return map which contains the critical (turning) point.  For a subset of parameters, taking higher iterates on the unimodal branch will give a full-branched induced map. This will have an acip. These acips will converge to the acip for the induced map corresponding to $a_0$. A careful selection of parameters will give control of the integral of the inducing time. If the integral of the inducing time on the (shrinking as $a \to a_0$) unimodal branch  converges to a non-zero constant, there is a good chance of proving the proposition.

        \begin{proof}[Proof of Proposition~\ref{prop:betadel}]
        For $a > a_0$, there is a point $\alpha_* > c$ (depending on $a$) with $f_a(\alpha_*) = f_a(\alpha) = \alpha$. There are  points $v, v_* \in (\alpha, \alpha_*)$ (depending on $a$), with $v < v_*$, such that 
        $$f_a(v) = f_a(v_*) = 1/2.$$
        For each  even number $k \geq 2$, there exists a maximal interval $A_k$ of parameters $a$ for which $f_a^j(c) \in (0, 1/2)$ for $j = 1, \ldots, k-1$, while $f_a^{k}(c) \in (1/2, 1)$. Moreover the map $\xi : a \to f_a^{k}(c)$ maps $A_k$ continuously onto $(1/2, 1)$. 
        By continuity, there is a subinterval $A'_k$ of parameters in $A_k$ mapped by $\xi$ onto $(\alpha, \alpha_*)$. 
        For $a$ in $A'_k$, denote by  $q < q_*$ the points closest to $c$ with $f^k_a(q) = f^k_a(q_*) = \alpha$. As $k$ is even,  $f^k_a$ has a local maximum at $c$ and $f^k_a((q,q_*)) = (\alpha,c)$. The point $v$ gets mapped by $f_a$ to $1/2$, so $v, v_* \notin [q,q_*].$ 
        In particular,  $(q,q_*)$ is compactly contained in $(v,v,_*)$ is compactly contained in $(\alpha, \alpha_*)$. 

        Inside the subinterval $A'_k$, there is a parameter $a(k,n)$ for  which, writing $f_{k,n}$ for $f_{a(k,n)}$, 
        $$ 
        v = f^{kn}_{k,n} (c) < q < f^{k(n-1)}_{k,n}(c) < \cdots < f^{k}_{k,n}(c) < c.$$
        The graph of $f^k_{k,n}$ restricted to $(q,q_*)$ lies below the diagonal. For $n \geq 2$ and $k$ large, the graph very nearly touches the diagonal. 
        Let us fix $f_{k,n}$ for now. 

        The first return map $\phi$ to $(\alpha, \alpha_*)$ has a unimodal branch defined on $(q,q_*)$ and $\phi$ coincides  with $f_{k,n}^k$ on $(q,q_*)$, see Figure~\ref{fig:returnmap}. 
        For $j = 1, \ldots, n$, 
        $$\phi^j(c) = f^{kj}_{k,n}(c).$$ 
        Let us set 
        $$X := (\alpha, \alpha_*)\setminus [q,q_*].$$
        Branches of $\phi$ in $X$
        are full, mapping diffeomorphically onto $(\alpha, \alpha_*)$ [the proof of this is straightforward, noting that $\Orb(\partial X) \cap X = \emptyset$].
        Points in $[q,q_*]$ eventually leave $[q,q_*]$, and the first time they do, they enter $[\alpha, q]$. Denote by $\psi$ the first entry map to 
        $X$.
        Restricted to $X$ it is just the identity map, while on $(q,q_*)$ it consists of iterates of $\phi$. 
        \begin{figure}
    \centering
        \def\svgwidth{0.9\columnwidth}
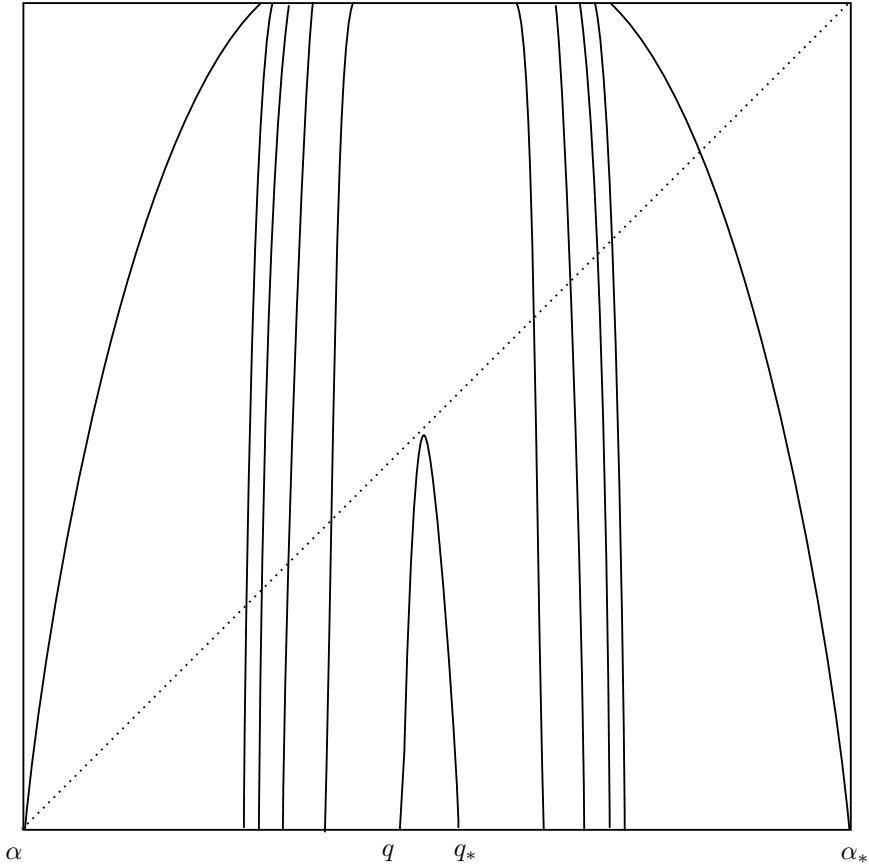
\caption{The first return map $\phi$ to $(\alpha, \alpha_*)$  has one (almost parabolic) unimodal branch $(q,q_*)$ and infinitely many diffeomorphic branches.}
    \label{fig:returnmap}
\end{figure}
                Let us define the map 
                $$\Psi = \Psi_{k,n} := 
                \phi \circ \psi. $$

                \begin{lemma}\label{lem:psifull}
                    This map $\Psi$ is full-branched and the following distortion bound holds for some constant $C$ (independent of $k,n$): 
            \begin{equation}\label{eqn:nsdpsi}
                \left|\frac{D\Psi^n(x) }{D\Psi^n(y)}\right| \leq C|\Psi^n(x) - \Psi^n(y)|
            \end{equation}
            for all $x,y$ in the same branch of $\Psi^n$.
        \end{lemma}
        \begin{proof}
        Since $\psi (c) = v$ and $f_{k,n}(v) = 1/2$, a fixed point at a positive (and bounded away from $0$ for the parameters under consideration) distance from $[\alpha,\alpha_*]$, $\Psi$ has no critical points (in the sense zero-derivative). 
        Branches of $\psi$ in $(q,q_*)$ are mapped (by $\psi$) diffeomorphically onto $(\alpha, q)$ or onto $(\alpha,v)$. 
         Neither $\alpha$ nor $v$ nor $q$ is in any branch of $\phi$. It follows that branches of $\Psi$ in $(q,q_*)$ are full. 

    Noting that the orbit of $c$ never enters $(1/2, \alpha]$ nor $[\alpha_*,1)$, each branch of $\Psi$ extends (coinciding with the corresponding iterate of $f$) to map diffeomorphically onto $(1/2, 1)$. Moreover, the extensions have non-positive Schwarzian derivative. Therefore there is a uniform distortion bound for all iterates of $\Psi$, independently of $k$.
    \end{proof}

            On each branch of $\Psi$, $\Psi$ coincides with some iterate of $f_{k,n}$. Thus there is a function $\tau = \tau_{k,n}$ defined on the domain of $\Psi$, constant on each branch of $\Psi$, for which 
            $$\Psi = f_{k,n}^\tau.$$ 
            Estimating the size of $\tau^{-1}(j)$, for $j \in \N$, will give us information about the acip for $f_{k,n}$, namely its distribution, its Lyapunov exponent and (therefore) its entropy. 
            Outside $(q,q_*)$, the estimates for $\tau$ which follow are independent of $k, n$. 

        Let $E_j$ denote the set of points whose first entry time to $[\alpha, \alpha_*]$ is $j$. We shall now obtain uniform exponential bounds on the size of $E_j$. 
    \begin{lemma}\label{lem:expaj}
        $m(E_j) \leq e^{-j/4}.$
    \end{lemma}
    \begin{proof}
    Calculation gives $f_{a_0}(5/8) = 1/2$, so for $a > a_0$, $v < 5/8$. Therefore $\alpha < 5/8$ and $\alpha_* > 7/8$, so
    $$
    m(E_0) = m((\alpha, \alpha_*)) \in (1/4, 1/2)$$
    and 
    $$
    m(E_1) = 2m((\alpha_*,1)) < 2^{-2}.$$
By convexity of $f_{k,n}$ on $(1/2, 1]$ and since $f_{k,n}(3/4) \leq 1/3$ and $f_{k,n}(1)=1$, the graph of $f_{k,n}$ restricted to $(3/4, 1)$ lies under the line of slope $8/3 >2$ passing through the repelling fixed point $(1,1)$. Moreover, $f((\alpha_*,1)) = (\alpha,1)$. 
        Therefore, for $j \geq 1$, 
        $$m(E_j \cap (\alpha_*,1)) < 2^{-j-2}. $$
        By symmetry, 
        $$m(E_j \cap (1/2, 1)) < 2^{-j-1}. $$

        On $(0, 1/2)$, the set of points with first entry time to $(1/2,1)$ equal to $i$ has measure $2^{-i-1}$, for $i \geq 1$, and this set is an interval mapped by $f^i_{k,n}$ onto $(1/2,1)$. We deduce, for $j \geq 1$,
        \begin{align*}
            m(E_j \cap (0,1/2)) & =  \sum_{i=1}^{j} 2^{-i} m(E_{j-i}\cap(1/2,1)) \\
                                & <  \sum_{i=1}^{j} 2^{-i} 2^{-j+i -1} \\
                                & = j 2^{-j-1}.
        \end{align*}
        say. 
        Hence $m(E_j) < (j+1) 2^{-j-1}$, so $$m(E_j) < e^{-j/4},$$ say. 
    \end{proof}

    The distortion of $f_{k,n}$ on $(1/2,1)$ is unbounded, due to the critical point. However, it is quadratic, which gives a certain amount of control.  Denote by $R_j$ the set of points in $[\alpha, \alpha_*]$ whose first return time to $[\alpha, \alpha_*]$ is equal to $j$. 
    \begin{lemma}\label{lem:exprj}
        $m(R_j) \leq e^{-j/8}.$
    \end{lemma}
    \begin{proof}
        Given a measurable set $X \subset (0,1)$, since $f_a$ is quadratic on $(1/2,1)$,   $$m((1/2,1) \cap f_a^{-1}(X)) \leq \sqrt {\frac{m(X)}{a}} \leq \sqrt{m(X)}/3,$$
        provided $a \geq 32/3$. 
        From the preceding lemma, we obtain 
        $$ 
        m(R_j) \leq \sqrt{m(E_{j-1})}/3 < e^{-\frac{j-1}8}/3 < e^{-j/8},$$
        as required. 
    \end{proof}

    \begin{remark}\label{rem:psicon}
        In the limit, $\Psi$ converges (in some canonical sense) to the first return map for $f_{a_0}$ to $[\alpha, \alpha_*]$. As the estimates of Lemma~\ref{lem:exprj} are uniform, they also hold for the return map for $f_{a_0}$ and the integral of the first return time, $\sum_j jm(R_j)$,
     is bounded.  
    \end{remark}

    Returning to $f_{k,n}$, 
    we can denote by $e_{k,n}$ the first entry time to $X = [\alpha,\alpha_*]\setminus [q,q_*]$. It is a multiple of $k$ on $(q,q_*)$. We can decompose $\tau$ on $(q,q_*)$ as
    $$
    \tau = e_{k,n} + \tau \circ \psi.$$
    We shall estimate these terms separately. 
    \begin{lemma}\label{lem:taubnd}
        $\lim_{k \to \infty} \sup_{n\geq 2} \int_{(q,q_*)} \tau \circ \psi\, dx = 0.$ 
\end{lemma}
    \begin{proof}
        The unimodal branch is just 
        \begin{equation}\label{eqn:centralb}
            x \mapsto -2^ka (x-3/4)^2 + 3/4 - \delta
        \end{equation}
        for some small $\delta = \delta(k,n) >0$. 
        There is a sequence of points $q= q_0 < q_1 < q_2 < \cdots < q_{n-1} < c$ for which $\phi^l(q_l) = q$. 
        Define the intervals $J_l := (q_l, q_{l+1})$ for $l \leq n-2$ and $J_{n-1} := (q_{n-1},c).$
        Each interval $J_l$, $l\geq1$, gets mapped by $\phi$ onto $J_{l-1}$. The interval $J_0$ gets mapped onto $(\alpha, q)$. 
        On $J_0$ we have 
        $$m(\phi^{-1}(R_j) \cap J_0) \leq \sqrt{\frac{m(R_j)}{2^ka}}.$$

        Let $l_0 \geq 0$ be the maximal $l$ for which $D\phi(q_l) > 2$. 
        For $1 \leq l < l_0$, on the interval $J_l$ the derivative is at least $2$. Hence
                \begin{equation}\label{eqn:psi1}
        m(\psi^{-1}(R_j) \cap J_l) \leq 2^{-l}\sqrt{\frac{m(R_j)}{2^ka}}. 
            \end{equation}

        Going one step back, we obtain on $J_{l_0}$ (taking a poor bound for better legibility),
        $$
        m(\psi^{-1}(R_j) \cap J_{l_0}) \leq \left(\frac{m(R_j)}{2^ka}\right)^{1/4}.$$ 
        
        Consider now the affine conjugacy $A$ such that
        $$
        \phi = A^{-1} \circ P \circ A,$$
        where $P$ is the quadratic map $P : y \mapsto y^2 +\kappa$, and $\kappa > 1/4$ depends on $\delta$ and thus on $k, n$. One can calculate 
        \begin{equation}\label{eqn:conjA}
            A : x \mapsto -2^ka (x-3/4)
        \end{equation}
        (recalling $a=a(k,n)$). 
        Since $A$ is affine, 
                \begin{equation}\label{eqn:PAx}
        D\phi(x) = DP(Ax). 
            \end{equation}
            Writing $y_0 := Aq_{n-1}$, we can apply Lemma~\ref{lem:ykdistn}. With the notation of that lemma, $k_0 = n -1 - l_0$, while $y_{k_0 -r}  = A q_{l_0+r}.$ Thus $(y_{k_0 -r-1}, y_{k_0 - r}) = A J_{l_0 +r}.$
                  For $k_0 -r = 2,3,\ldots, k_0$ or, equivalently, for $r= 0, 1, \ldots, n-3-l_0$, 
        by~\eqref{eqn:pkap} and~\eqref{eqn:PAx}, for $x,x' \in J_{l_0 + r}$, 
        $$
        \frac{D\phi^r(x)}{D\phi^r(x')} < 21.
        $$
        Hence
        $$
        m(\psi^{-1}(R_j) \cap J_{l_0 +r})  
        \leq 21 \frac{|J_{l_0 +r}|}{|J_{l_0}|} \left(\frac{m(R_j)}{2^ka}\right)^{1/4}.$$ 
        Summing over intervals $J_{l_0 + r}$, we obtain
        $$
        m(\psi^{-1}(R_j) \cap (q_{l_0}, q_{n-2})  
        \leq 21 \frac{|q_{n-2}-q_{l_0}|}{|J_{l_0}|}   \left(\frac{m(R_j)}{2^ka}\right)^{1/4}. 
        $$
        By~\eqref{eqn:pkapy}, since $A$ is affine, we have $|q_{n-2}-q_{l_0}| < 8|J_{l_0}|$, so 
        \begin{equation}\label{eqn:psi3}
        m(\psi^{-1}(R_j) \cap (q_{l_0}, q_{n-2})  
        \leq 168   \left(\frac{m(R_j)}{2^ka}\right)^{1/4}. 
    \end{equation}

        We have now dealt with $(q_0, q_{n-2})$; the intervals $J_{n-2}, J_{n-1}$ remain to be treated.
        Pulling back~\eqref{eqn:psi3} (once for $J_{n-2}$ and once more for $J_{n-1}$), we obtain, for large $k$,
        \begin{equation}\label{eqn:psi4}
        m(\psi^{-1}(R_j) \cap (J_{n-2} \cup J_{n-1}))  
        \leq  \left(\frac{m(R_j)}{2^ka}\right)^{1/16}. 
    \end{equation}
        
    Combining the estimates $a>1$,~\eqref{eqn:psi1} (summed),~\eqref{eqn:psi3} and~\eqref{eqn:psi4} gives
    \begin{align*}
        m(\psi^{-1}(R_j) \cap (q,c))  
        & \leq 
        2 \sqrt{\frac{m(R_j)}{2^ka}} + 
         168 \left(\frac{m(R_j)}{2^ka}\right)^{1/4}
             +
          \left(\frac{m(R_j)}{2^ka}\right)^{1/16}\\
          & \leq 2^{10-k/16} m(R_j)^{1/16}.
      \end{align*}
             Note that there is no $n$-dependence in the bound.  

             From Lemma~\ref{lem:exprj}, $m(R_j) \leq \exp(-j/8).$
             Hence
             $$
        m(\psi^{-1}(R_j) \cap (q,c))  
        \leq 
        2^{10 -k/16} \exp(-j/128).$$
        Therefore
        $$\int_{(q,q_*)} \tau \circ \psi\, dx = 2\sum_j jm(\psi^{-1}(R_j) \cap (q,c)) 
        \leq
        2^{11-k/16} \sum_j j \exp(-j/128).$$
        This upper bound has no $n$-dependence and tends to zero as $k \to \infty$, as required. 
    \end{proof}

    \begin{lemma}\label{lem:qlong}
        There is a constant $C_0 > 1$ such that, for all $k, n$, 
        \begin{equation}\label{eqn:qlong}
            \frac1{C_0} \sqrt{\frac{1}{2^ka}} 
            \leq |q_* - q|
            \leq C_0 \sqrt{\frac{1}{2^ka}} .
        \end{equation}
    \end{lemma}
    \begin{proof}
        The image of $\phi$ on $[q,q_*]$ contains $\alpha$ and $v$, and $\inf_{k,n} |v - \alpha| > 0$. The range of $\phi$ has length bounded by $1/2$. From~\eqref{eqn:centralb}, the result then follows. 
    \end{proof}

    \begin{lemma}\label{lem:ekn}
        The first entry time $e_{k,n}$ to $X = [\alpha,\alpha_*]\setminus [q,q_*]$ has the following properties. 
        \begin{itemize}
            \item
                \begin{equation}\label{eqn:ekn1}
        \lim_{k\to \infty} \int_{(q,q_*)} e_{k,2}\, dx = 0;
        \end{equation}
            \item
                \begin{equation}\label{eqn:ekn2}
        \lim_{n\to \infty} \int_{(q,q_*)} e_{k,n}\, dx = \infty;
        \end{equation}
            \item
                \begin{equation}\label{eqn:ekn3}
        \int_{(q,q_*)} e_{k,n+1}\, dx 
        \leq 
        \int_{(q,q_*)} e_{k,n}\, dx + 4k|q_*-q|;
            \end{equation}
            \item
    For $x \in (q,q_*)$ and $j=1,2,\ldots, k-s$, 
                \begin{equation}\label{eqn:ekn4}
    \lim_{k \to \infty} \inf_{n \geq 1} \frac{\# \{j \leq e_{k,n}(x) : f^j_{k,n}(x)\in B(1/3, 1/k)\}}{e_{k,n}(x)} = 1.
            \end{equation}
        \end{itemize}
    \end{lemma}
    \begin{proof}
        The entry time satisfies $e_{k,n} \leq nk$, while~\eqref{eqn:qlong} bounds the length of
        the interval $(q,q_*)$. This implies~\eqref{eqn:ekn1}. 

        Recall the affine map $A$ of~\eqref{eqn:conjA} which conjugates the unimodal branch of $\phi = \phi_{k,n}$ with a normalised quadratic $P : y \mapsto y^2 + \kappa$, with $\kappa$ depending on $k,n$. [The map $A$ also depends slightly on $k,n$.] 
        By~\eqref{eqn:qlong}, 
        $$
        \frac{1}{2C_0} \sqrt{{2^ka}} \leq Aq \leq C_0 \sqrt{{2^ka}}.
        $$
        Set $M_{k,n} := P(Aq) = A\alpha$ and $V_{k,n} := Av$. 
        Then $P^n(0) = V_{k,n}$  (in particular, $\kappa = \kappa(V_{k,n},n)$) and the first entry time at $Ax$ to $[M_{k,n},\infty)$ is $1 + \frac{e_{k,n}(x)}{k}$. 
            From~\eqref{eqn:kvnover4}, we deduce~\eqref{eqn:ekn2}.

            Recall that $v \leq 5/8$, so 
         $|\alpha -v| < |v-3/4|$, so 
        $$M_{k,n}/2 \leq V_{k,n} \leq M_{k,n}.$$
        Moreover, for fixed $k$, $$V_{k,2}/2 \leq V_{k,n} \leq 2V_{k,2}.$$ 
        Setting $V_0 := 10 V_{k,2}$, we can apply Lemma~\ref{lem:evvn} with
        $M = M_{k,n}$, $V=V_{k,n}$, $M' = M_{k,n+1}$, $V' = V_{k,n+1}$. 
        Via the conjugacy, this gives
        $$
        e_{k,n+1} \leq e_{k,n}(x)+3k.$$
        This almost gives~\eqref{eqn:ekn3}, we just replace $3$ by $4$ as the interval $(q,q_*)$ depends on $k,n$ and will vary a little for $n\geq 2$ (the entry time equals $k$ on the non-common part). 

        For $x \in (q,q_*)$ and $j=1,2,\ldots, k-s$, $f_{k,n}^j(x) \in B(1/3, 2^{-s})$. This implies~\eqref{eqn:ekn4}.
    \end{proof}

    To proceed, we need to recall some results concerning full-branched expanding maps applied to our situation. 
    \begin{definition}
    Let us call a collection of maps $(F_k)_{k\geq 0}$ a \emph{convergent Markov system} \index[def]{Convergent Markov system} if the following holds, denoting by $\D_k$ the domain of $F_k$: 
    \begin{itemize}
        \item
            There is an open interval $I$ with $\D_k \subset I$ for all $k$, and $m(I \setminus \D_k) = 0$; 
            
                \item
                    each $\D_k$ is a countable union of open intervals; 
                \item
                    for each branch $J$ of $F_k$, $F_k : J \to I$ is a (surjective) $C^1$ diffeomorphism; 
                \item
                    for some $K$, for all $n, k \geq 0$, the  distortion estimate 
                    $$
                    \left|\frac{DF^n_k(x)}{DF_k^n(y)}\right| \leq K|F_k^n(x) - F_k^n(y)|$$
                holds for all $x,y$ in the same branch of $F_k^n$;
                \item
                    the maps $F_k$ converge uniformly to $F_0$ on compact subsets of $\D_0$ as $k \to \infty$. 
            \end{itemize}
\end{definition}
Then the following results are well-known, we provide only a sketch of their proof. 
\begin{fact} \label{fact:rho}
        Each $F_k$ has an acip $\nu_k$ with Lipschitz density $\rho_k = \frac{d\nu_k}{dm}$ on $I$. The $\rho_k$ are  uniformly Lipschitz, bounded and bounded away from 0. 
        The densities converge: $\rho_k \to \rho_0$ (uniformly) as $k \to \infty$. 
        If $\tau_k$ is a continuous integer-valued function on $\D_k$, $\tau_k$ converges to $\tau_0$ uniformly on compact subsets of $\D_0$, and if there exists $C >0$ such that, for all $k$, 
        $$
        m(\{\tau_k^{-1}([n,\infty)) < C n^{-3},$$
                then for each $\gamma >0$, there exists a compact set $\Lambda \subset \D_0$ satisfying
        \begin{equation}\label{eqn:taukconv}
            \int_{\D_k \setminus \Lambda}\tau_k \, d\nu_k < \gamma
        \end{equation}
        for $k =0$ and for all large $k$, while $\tau_k = \tau_0$ on $\Lambda$ for all large $k$. 
    \end{fact}
    \begin{proof}
        The first two statements are proven in \cite[Theorem~V.2.2]{MSbook} (at least for one map, but the estimates are uniform, depending only on the distortion estimate). 
        We now sketch the proof of the final statement, using ideas similar to those earlier in the article. One can find, thanks to the distortion bounds, a uniform $N, \lambda$ for which  $|DF_k^N| > \lambda > 1$.
        It follows that return maps to small intervals have exponential tails: if $P$ is a subinterval of $I$ then the set of points with first return time $\geq j$ has measure $\leq C(|P|)e^{-j\delta}$, where $C(\cdot)>0$ can be taken monotone increasing and only depending on the length of $P$, and both $C(\cdot)$ and $\delta>0$ are independent of $k$. 
        If $n$ is large enough and $P$ is a branch of $F^n_0$, then $P$ is tiny and by the uniform bounds on the densities, the densities $\rho_k$ are approximately constant on $P$. By Kac' Lemma, 
        \begin{equation}\label{eqn:kacprk}
            \int_P r_k \,d\rho_k = 1,
        \end{equation}
        where $r_k$ denotes the return time to $P$ for $F_k$. 
        Given small $\eps>0$, let $\Lambda$ be a finite union of intervals compactly contained in some $\Lambda_1$ compactly contained in the domain of the return map to $P$ under $F_0$, and choose $\Lambda$ (and $\Lambda_1$) such that 
        \begin{equation}\label{eqn:plam}
            \int_{P\setminus \Lambda} r_k\, d\rho_k< \eps.
        \end{equation}
        This we can do by the uniform exponential tails estimate. Note that $r_0$ is bounded on $\Lambda$, and that $\Lambda$ does not depend on $k$. 
        For $k$ large enough, $r_k$ coincides with $r_0$ on $\Lambda$ (by uniform convergence of $F_k$ on $\Lambda_1$). 
        But then, by~\eqref{eqn:kacprk} and~\eqref{eqn:plam},
        $$ 
        \left|1 - \int_\Lambda r_0\, d\rho_k\right| =
        \left|1 - \int_\Lambda r_k\, d\rho_k\right|  
        \leq \eps
        $$
        and similarly
        $$
        \left|1 - \int_\Lambda r_0\, d\rho_0 \right|\leq \eps.$$
        Since the $\rho_k$ are each approximately constant (\emph{a priori} varying a lot with $k$) on $P$ and bounded away from zero and infinity, we deduce that $\rho_0 \approx \rho_k$ for large $k$, or $\rho_k \to \rho_0$ as $k \to \infty$. 

        It remains to prove the statements concerning $\tau_k$. As a continuous integer-valued function, $\tau_k$ is constant on each connected component of $\D_k$. Given any compact subset of $\Lambda \subset \D_0$, for $k$ large enough $\tau_k \equiv \tau_0$ on $\Lambda$.
                The tail estimate implies that, given $\gamma >0$, there exists $\eps >0$ such that on any set $Y$ of Lebesgue measure $\eps >0$, 
        $$
        \int_Y \tau_k \, d\nu_k < \gamma.$$ 
        In particular, for $\Lambda$ large enough that $\D_k \setminus \Lambda$ has measure at most $\eps$ for all large $k$ (and for $k=0$), 
        $$
         \int_{\D_k \setminus \Lambda} \tau_k \, d\nu_k < \gamma$$
         for $k$ large and for $k = 0$. 
         showing~\eqref{eqn:taukconv}.
    \end{proof}

    Returning to our collection of maps $f_{k,n}$ with their induced maps $\Psi_{k,n}$ and inducing times $\tau_k$, note that $(q,q_*)$ tends to the point $\{3/4\}$ as $k \to \infty$ (independently of $n$). Moreover the maps $\Psi_{k,n}$ have a uniform distortion bound~\eqref{eqn:nsdpsi}. Let us denote by $\Psi_0$   the first return map to $[\alpha, \alpha_*]$ under $f_{a_0}$, and by $\tau_0$ its return time. 
    Then for any sequence $(n_k)_k$, setting $\Psi_k := \Psi_{k,n_k}$, the sequence
    $(\Psi_k)_{k\geq 0}$ is a convergent Markov system.

    In particular, the acips $\nu_k$ for $\Psi_k$ converge the acip $\nu_0$ for $\Psi_0$. By Remark~\ref{rem:psicon}, 
    $$
    S_0 := \int \tau_0 \, d\nu_0 < \infty.$$
    On the other hand, if we choose $n_k$ carefully, the integral of the inducing times will converge, but not to the corresponding integral $S_0$ for $\Psi_0$. Presently we shall choose $n_k$. 

    Denote by $S_{k,n}$ the integral of the inducing time for $\Psi_{k,n}$, for $k \geq 0$: 
    $$
    S_{k,n} := \int_{[\alpha, \alpha_*]} \tau_k \, d\nu_k.$$
    Recall $X = X_{k,n} = [\alpha, \alpha_*]\setminus [q,q_*]$. It will be useful to define 
    $$
    S_{k,n}^X:= \int_{X} \tau_k \, d\nu_k$$
    and 
    $$
    S_{k,n}^\delta := \int_{(q,q_*)} \tau_k \, d\nu_k$$
    so $S_{k,n}^X + S_{k,n}^\delta = S_{k,n}.$

    Given $\beta >0$,
    for $k \geq 1$, let $n(k) \geq 2$ be the minimal $n$ such that $S_{k,n}^\delta > \beta S_0$. 
    For large $k$, 
    $$S_{k,n}^\delta \approx \int_{(q,q_*)} e_{k,n}\, d\nu_k$$ 
    by Lemma~\ref{lem:taubnd}.
    Now~\eqref{eqn:qlong} and the first three statements of Lemma~\ref{lem:ekn} together imply that $n(k)$ exists and 
    $S_{k,n(k)}^\delta \approx \beta S_0.$
    In particular, $$\lim_{k\to \infty} S^\delta_{k,n(k)} = \beta S_0.$$
    Henceforth, we fix the sequence $n_k := n(k)$. We can drop the $n$-dependence from the notation in the terms considered, and write $S_k$ for $S_{k,n}$, etc.
    We denote the parameter $a(k,n_k)$ by $a_k$.
    As $k \to \infty$, $S_k^X$ converges to $S_0.$ This follows from convergence of the Lipschitz densities $\rho_k$ (Fact~\ref{fact:rho}), convergence of the maps $\Psi_k$ and the exponential tail estimate for the inducing times $\tau_k$ (Lemma~\ref{lem:exprj}, as $\tau_k(x) = j$ on $R_j \cap X$). 

    Now let us look at the spread measures. 
    For $k \geq 0$, let $\nu_k^j$ denote the restriction of $\nu$ to $\tau_k^{-1}(j)$ and set 
    $$
    \mu_k := \frac1{S_k} \sum_{j\geq1}  \sum_{i=0}^{j-1} (f_{a_k}^i)_* \nu^j_k.
    $$
    Then $\mu_k$ is an ergodic acip for each $k \geq 0$. We wish to show that $\mu_k$ does not converge to $\mu_0$ as $k \to \infty.$
    With this aim,  for $k \geq 1$, let us restrict $\nu_k$ to $X$ and then spread:
    $$
    \eta_k^X 
    := \frac1{S_k} \sum_{j\geq1}  \sum_{i=0}^{j-1} (f_{a_k}^i)_* (\nu^j_k)_{|X}.
    $$
    We do the same for $\nu_k$ restricted to $(q,q_*)$:
    $$
    \eta_k^\delta 
    := \frac1{S_k} \sum_{j\geq1}  \sum_{i=0}^{j-1} (f_{a_k}^i)_* (\nu^j_k)_{|(q,q_*)}.
    $$
    Then $\mu_k = \eta_k^X + \eta_k^\delta$ and 
    $$\eta_k^X([0,1]) = \frac{S_k^X }{ S_k}.$$

    Because tails on $X$ are exponential and $\Psi_k$ restricted to $X$ converges to $\Psi_0$, one can use Fact~\ref{fact:rho} and in particular~\eqref{eqn:taukconv} to show that $\eta_k^X$ converges to 
    $$
    \lim_{k\to \infty} \frac{S_k^X }{ S_k} \mu_0 = \frac{1}{1+\beta} \mu_0.$$
    So it remains to check what happens to the limit of $\eta_k^\delta$.
    Well, let $J_j$ denote the set of $x \in (q,q_*)$ with $e_{k,n_k}(x) = j$. The difference between $\eta_k^\delta$ and the measure
    $$
    \hat \eta_k^\delta := 
     \frac1{S_k} \sum_{j\geq1}  \sum_{i=0}^{j} (f_{a_k}^i)_* (\nu^j_k)_{|J_j}
    $$
    is tiny, by~Lemma~\ref{lem:taubnd}. In particular, the limits of $\eta_k^\delta$ and $\hat \eta_k^\delta$ coincide (assuming they exist). 

    By~\eqref{eqn:ekn4}, 
    $$
    \lim_{k\to \infty} \hat \eta_k^\delta ([0,1] \setminus B(1/3, 1/k)) = 0.$$
    Therefore $\hat \eta_k^\delta$ and $\eta_k^\delta$ converge to an atomic measure supported on the point $1/3$ with mass given by
    $$
    \lim_{k\to \infty} \frac{S_k^\delta }{ S_k}  = \frac{\beta}{1+\beta}.$$

    It remains to  bound from below the entropies of $\mu_k$. The Lyapunov exponents $\int_{[\alpha,\alpha_*]} \log |D\Psi_k|\, d\nu_k$ are uniformly bounded below away from $0$. By Pesin's formula, the entropies $h(\nu_k)$ are also uniformly bounded below away from $0$. 
    The integrals of the inducing time converge to $S_0(1+\beta)$. By Abramov's formula, the entropies $h(\mu_k)$ are also bounded away from $0$. 

    That the critical relations are decreasing follows from Remark~\ref{rem:deccr} and Definition~\ref{def:deccr}.

    This concludes the proof of Proposition~\ref{prop:betadel} which, we recall, implies Theorem~\ref{THM:EXO1BIS}.
\end{proof}

\backmatter

\printindex[def]
\printindex

\bibliographystyle{amsalpha}

\end{document}

%% file: pressure.pdf_tex
\begingroup%
  \makeatletter%
  \providecommand\color[2][]{%
    \errmessage{(Inkscape) Color is used for the text in Inkscape, but the package 'color.sty' is not loaded}%
    \renewcommand\color[2][]{}%
  }%
  \providecommand\transparent[1]{%
    \errmessage{(Inkscape) Transparency is used (non-zero) for the text in Inkscape, but the package 'transparent.sty' is not loaded}%
    \renewcommand\transparent[1]{}%
  }%
  \providecommand\rotatebox[2]{#2}%
  \ifx\svgwidth\undefined%
    \setlength{\unitlength}{527.07411207bp}%
    \ifx\svgscale\undefined%
      \relax%
    \else%
      \setlength{\unitlength}{\unitlength * \real{\svgscale}}%
    \fi%
  \else%
    \setlength{\unitlength}{\svgwidth}%
  \fi%
  \global\let\svgwidth\undefined%
  \global\let\svgscale\undefined%
  \makeatother%
  \begin{picture}(1,0.77025257)%
    \put(0,0){\includegraphics[width=\unitlength,page=1]{pressure.pdf}}%
    \put(0.42563179,0.32957675){\color[rgb]{0,0,0}\makebox(0,0)[lb]{\smash{$P(-t\log|Df|)$}}}%
    \put(0.2665416,0.48712474){\color[rgb]{0,0,0}\makebox(0,0)[lb]{\smash{$h_{\mathrm{top}}(f)$}}}%
    \put(0.25931036,0.37613322){\color[rgb]{0,0,0}\makebox(0,0)[lb]{\smash{$h(\mu)$}}}%
    \put(0.65535767,0.01666247){\color[rgb]{0,0,0}\makebox(0,0)[lb]{\smash{$E(f,\mu,-t\log|Df|)$}}}%
    \put(0.26260421,0.04875878){\color[rgb]{0,0,0}\makebox(0,0)[lb]{\smash{$\frac{h(\mu)}{\lambda(\mu)}$}}}%
    \put(0.75691927,0.15022941){\color[rgb]{0,0,0}\makebox(0,0)[lb]{\smash{$t^+$}}}%
    \put(0.03547475,0.15391433){\color[rgb]{0,0,0}\makebox(0,0)[lb]{\smash{$t^-$}}}%
    \put(0.2317077,0.16036282){\color[rgb]{0,0,0}\makebox(0,0)[lb]{\smash{$0$}}}%
    \put(-0.00227759,0.6393735){\color[rgb]{0,0,0}\makebox(0,0)[lb]{\smash{$P^0(-t\log|Df|)$}}}%
    \put(0.87812202,0.10416877){\color[rgb]{0,0,0}\makebox(0,0)[lb]{\smash{$P^0(-t\log|Df|)$}}}%
    \put(0.41206562,0.51768686){\color[rgb]{0,0,0}\makebox(0,0)[lb]{\smash{}}}%
    \put(0,0){\includegraphics[width=\unitlength,page=2]{pressure.pdf}}%
  \end{picture}%
\endgroup%

%% file: exo1.pdf_tex
\begingroup%
  \makeatletter%
  \providecommand\color[2][]{%
    \errmessage{(Inkscape) Color is used for the text in Inkscape, but the package 'color.sty' is not loaded}%
    \renewcommand\color[2][]{}%
  }%
  \providecommand\transparent[1]{%
    \errmessage{(Inkscape) Transparency is used (non-zero) for the text in Inkscape, but the package 'transparent.sty' is not loaded}%
    \renewcommand\transparent[1]{}%
  }%
  \providecommand\rotatebox[2]{#2}%
  \ifx\svgwidth\undefined%
    \setlength{\unitlength}{354.49705811bp}%
    \ifx\svgscale\undefined%
      \relax%
    \else%
      \setlength{\unitlength}{\unitlength * \real{\svgscale}}%
    \fi%
  \else%
    \setlength{\unitlength}{\svgwidth}%
  \fi%
  \global\let\svgwidth\undefined%
  \global\let\svgscale\undefined%
  \makeatother%
  \begin{picture}(1,0.94761704)%
    \put(0,0){\includegraphics[width=\unitlength,page=1]{exo1.pdf}}%
    \put(0.27994426,0.00656368){\color[rgb]{0,0,0}\makebox(0,0)[lb]{\smash{$\frac13$}}}%
    \put(0.42931301,0.00649812){\color[rgb]{0,0,0}\makebox(0,0)[lb]{\smash{$\frac12$}}}%
    \put(0.48429063,0.01003581){\color[rgb]{0,0,0}\makebox(0,0)[lb]{\smash{$\alpha$}}}%
    \put(0.54025825,0.01175149){\color[rgb]{0,0,0}\makebox(0,0)[lb]{\smash{$v$}}}%
    \put(0.65165486,0.00497514){\color[rgb]{0,0,0}\makebox(0,0)[lb]{\smash{$\frac34$}}}%
    \put(0.75679687,0.01503933){\color[rgb]{0,0,0}\makebox(0,0)[lb]{\smash{$v_*$}}}%
    \put(0.82034958,0.0127052){\color[rgb]{0,0,0}\makebox(0,0)[lb]{\smash{$\alpha_*$}}}%
    \put(0,0){\includegraphics[width=\unitlength,page=2]{exo1.pdf}}%
    \put(-0.00280988,0.02562582){\color[rgb]{0,0,0}\makebox(0,0)[lb]{\smash{$0$}}}%
    \put(0.88039965,0.02880277){\color[rgb]{0,0,0}\makebox(0,0)[lb]{\smash{$1$}}}%
    \put(0.48327315,0.84529514){\color[rgb]{0,0,0}\makebox(0,0)[lb]{\smash{$x \mapsto a(x-1/2)(x-1) +1$}}}%
    \put(0.10997407,0.81034799){\color[rgb]{0,0,0}\makebox(0,0)[lb]{\smash{$x \mapsto 1-2x$}}}%
    \put(0,0){\includegraphics[width=\unitlength,page=3]{exo1.pdf}}%
  \end{picture}%
\endgroup%

%% file: returnmap.pdf_tex
\begingroup%
  \makeatletter%
  \providecommand\color[2][]{%
    \errmessage{(Inkscape) Color is used for the text in Inkscape, but the package 'color.sty' is not loaded}%
    \renewcommand\color[2][]{}%
  }%
  \providecommand\transparent[1]{%
    \errmessage{(Inkscape) Transparency is used (non-zero) for the text in Inkscape, but the package 'transparent.sty' is not loaded}%
    \renewcommand\transparent[1]{}%
  }%
  \providecommand\rotatebox[2]{#2}%
  \ifx\svgwidth\undefined%
    \setlength{\unitlength}{420.34360962bp}%
    \ifx\svgscale\undefined%
      \relax%
    \else%
      \setlength{\unitlength}{\unitlength * \real{\svgscale}}%
    \fi%
  \else%
    \setlength{\unitlength}{\svgwidth}%
  \fi%
  \global\let\svgwidth\undefined%
  \global\let\svgscale\undefined%
  \makeatother%
  \begin{picture}(1,0.88118068)%
    \put(0,0){\includegraphics[width=\unitlength,page=1]{returnmap.pdf}}%
    \put(0.38105255,0.00978796){\color[rgb]{0,0,0}\makebox(0,0)[lb]{\smash{$q$}}}%
    \put(0.45437629,0.00978784){\color[rgb]{0,0,0}\makebox(0,0)[lb]{\smash{$q_*$}}}%
    \put(-0.00236971,0.00673277){\color[rgb]{0,0,0}\makebox(0,0)[lb]{\smash{$\alpha$}}}%
    \put(0.8484917,0.00673277){\color[rgb]{0,0,0}\makebox(0,0)[lb]{\smash{$\alpha_*$}}}%
    \put(0,0){\includegraphics[width=\unitlength,page=2]{returnmap.pdf}}%
  \end{picture}%
\endgroup%